\documentclass[a4paper,11pt]{article}
\usepackage{mystyle}

\usepackage{bm}
\usepackage[normalem]{ulem}
\usepackage{float}
\usepackage{subfig}
\usepackage{dsfont}
\usepackage{verbatim}
\usepackage[export]{adjustbox}

\usepackage{cleveref}

\usepackage{etoolbox}
\apptocmd{\thebibliography}{\raggedright}{}{}

\newcommand{\cEpn}{\mathcal{E}^{(p)}_n}
\newcommand{\cEpnCon}{\mathcal{E}^{(p)}_{n,\mathrm{con}}} 

\newcommand{\cFpn}{\mathcal{F}^{(p)}_{\eps_n}}
\newcommand{\cFpnCon}{\mathcal{F}^{(p)}_{\eps_n,\mathrm{con}}} 
\newcommand{\cFp}{\mathcal{F}^{(p)}_{\eps_n}}  
\newcommand{\cFpCon}{\mathcal{F}^{(p)}_{\eps,\mathrm{con}}} 

\newcommand{\cEpinfty}{\mathcal{E}^{(p)}_\infty}
\newcommand{\cEpinftyCon}{\mathcal{E}^{(p)}_{\infty,\mathrm{con}}}

\newcommand{\diff}{\,\dd}
\newcommand{\dH}{d_{\mathrm{H}}}
\newcommand{\Sep}{\mathrm{Sep}}
\newcommand{\tr}{\mathrm{tr}}
\newcommand{\Bias}{\mathrm{Bias}}

\theoremstyle{plain}

\theoremstyle{definition}
\newtheorem{Definition}{Definition}[section]

\title{PDE-Inspired Algorithms for Semi-Supervised Learning on Point Clouds}

\author[1,3]{Oliver M. Crook}
\author[2]{Tim Hurst}
\author[3]{Carola-Bibiane Sch\"{o}nlieb}
\author[3]{Matthew Thorpe}
\author[2]{Konstantinos C. Zygalakis}

\affil[1]{MRC Biostatistics Unit, School of Clinical Medicine, \protect\\ University of Cambridge, \protect\\ Cambridge CB2 0SR, UK\vspace{\baselineskip}}
\affil[2]{School of Mathematics and the Maxwell Institute for Mathematical Sciences,\protect\\ University of Edinburgh,\protect\\ Edinburgh, EH9 3FD, UK \vspace{\baselineskip}}
\affil[3]{Department of Applied Mathematics and Theoretical Physics,\protect\\ University of Cambridge,\protect\\ Cambridge, CB3 0WA, UK}

\date{September 2019}

\begin{document}

\maketitle

\begin{abstract}
Given a data set and a subset of labels the problem of semi-supervised learning on point clouds is to extend the labels to the entire data set.
In this paper we extend the labels by minimising the constrained discrete $p$-Dirichlet energy. 
Under suitable conditions the discrete problem can be connected, in the large data limit, with the minimiser of a weighted continuum $p$-Dirichlet energy with the same constraints.
We take advantage of this connection by designing numerical schemes that first estimate the density of the data and then apply PDE methods, such as pseudo-spectral methods, to solve the corresponding Euler-Lagrange equation.
We prove that our scheme is consistent in the large data limit for two methods of density estimation: kernel density estimation and spline kernel density estimation.
\end{abstract}

\noindent

\noindent

\section{Introduction}\label{Sec:Intro}

In many machine learning problems, such as classification or labelling, one often aims to exploit the usually large quantities of data in order to capture its geometry. Frequently in applications labels for some of the data are available but often in low quantities because of the cost of labelling points. In the semi-supervised learning setting we are given
a data set $\bm{x}_i\in \bbR^d$, $i=1,\dots,n$ sampled from an unknown probability measure $\mu$, and a small subset of labelled pairs $(\bm{x}_i,y_i)$, $i=1,\dots, N$ where $y_i\in\bbR$ and we work in the regime $N\ll n$.
Here, the labels $y_i$ for the first $N$ data points are known and we aim to estimate the labels $\{y_i\}_{i=N+1}^n$ for the remaining data points $\{\bm{x}_i\}_{i=N+1}^n$.

One method to assign these labels is to minimise an objective function, which penalises smoothness of assigned labels, under the constraint that known labels are preserved.
A common choice of such an objective function is the graph $p$-Dirichlet energy~\cite{zhu03,zhou05}, which allows one to define a discrete version of a Dirichlet energy.
More precisely, we consider the graph $(\Omega_n,\bm{W})$ of nodes $\Omega_n=\{\bm{x}_i\}_{i=1}^n$ and edge weights $\bm{W}=(W_{ij})_{i,j=1}^n$ where $W_{ij}$ is the edge weight between data points $\bm{x}_i$ and $\bm{x}_j$ (by convention we say there is no edge between $\bm{x}_i$ and $\bm{x}_j$ if $W_{ij} = 0$).
We use the random geometric graph model with length scale $\eps_n$ for defining the edge weights.
Approximately, the parameter $\eps_n$ determines the range at which two nodes become connected; the explicit construction is given in the following section.
The objective functional is defined as the difference between labels weighted by edge weights:
\begin{equation} \label{eq:Intro:conGraphpDirichlet}
	\cEpnCon(f) = \begin{cases} \frac{1}{\eps_n^p n^2} \sum_{i,j=1}^n W_{ij}|f(\bm{x}_i)-f(\bm{x}_j)|^p,\quad &\text{ if }f(\bm{x}_i) = y_i\text{ for all }i=1,\dots,N,\\ +\infty\quad & \text{else.} \end{cases}
\end{equation}

Computing the minimiser of the graph $p$-Dirichlet energy becomes computationally expensive when considering a large number of data points.
However, it has been shown~\cite{SlepcevThorpe2017} that under an admissible scaling regime in $\eps_n$ (and when $p>d$), minimisers of~\eqref{eq:Intro:conGraphpDirichlet} converge to minimisers of the continuum $p$-Dirichlet energy:
\begin{equation} \label{eq:Intro:conpDirichlet}
	\cEpinftyCon(f;\rho) = \begin{cases} \sigma_\eta\int_\Omega |\nabla f(\bm{x})|^p\rho^2(\bm{x})\diff \bm{x} & \text{if } f\in W^{1,p}(\Omega)\text{ and } f(\bm{x}_i) = y_i \text{ for } i=1,\dots,N,\\ +\infty & \text{else}\end{cases}
\end{equation}
where $\rho$ is the density of the data points, and $\sigma_\eta$ is some constant depending only on a weight function $\eta$ (satisfying Definition~\ref{def:weightFunction}).
This result shows that the minimiser of the continuum $p$-Dirichlet energy is an accurate estimate of the minimiser of the graph $p$-Dirichlet energy when considering a large amount of data.
However, the dependency of the continuum $p$-Dirichlet energy on the underlying data density poses a new problem as it is unrealistic to assume that we know the density of the data.
The objective of this paper is to develop the framework for a new numerical method for finding minimisers of $\cEpn$ based on its connection to the continuum variational problem of minimising $\cEpinftyCon$.
In particular, it is our aim to develop a numerical scheme that is efficient for large ($n\gg 1$) datasets.
We refer to the work by Flores Rios, Calder and Lerman~\cite{floresrios19AAA} for algorithms based on the discrete problem. 

We note that an associated {\em non-local} continuum $p$-Dirichlet energy is also of interest;
\[	\cFpnCon(f;\rho) = \begin{cases}  
	\begin{aligned}[t]
	\frac{1}{\eps^p}\int_\Omega\int_\Omega& \eta_{\eps_n}(|\bm{x}-\bm{z}|)
	|f(\bm{x})-f(\bm{z})|^p\\&\times\rho(\bm{x})\rho(\bm{z})\diff \bm{x} \diff \bm{z}. 
	\end{aligned}
	& \text{if } f\in W^{1,p}(\Omega)\text{ and } f(\bm{x}_i) = y_i, i=1,\dots,N,\\ +\infty & \text{else}\end{cases}
\]
for example, $\cFpnCon(f;\rho)$ was considered in \cite{SlepcevThorpe2017} as an intermediary functional to provide convergence between $\cEpnCon(f)$ and $\cEpinftyCon$ and has appeared in~\cite{hafiene18aAAA,hafiene18,hafiene18b} as the continuum limit of $\cEpnCon$ when $\eps_n=\eps$ is fixed (although not using the hard constraint).
We also develop a numerical method for computing minimisers of the above functional that are efficient for large $n$.
\vspace{\baselineskip}

In our approach we, rather than minimise~\eqref{eq:Intro:conGraphpDirichlet}, aim instead to minimise~\eqref{eq:Intro:conpDirichlet}.
Since the density $\rho$ is unknown we are required to estimate it from the data $\Omega_n$.
Density estimation is a well-studied problem in statistics dating back, at least, to Fix and Hodges in 1951~\cite{fix89} and Akaike in 1954~\cite{akaike54}.
The first method we consider here to estimate the density is \emph{kernel density estimation}~\cite{rosenblatt1956,parzen62}.
The idea behind kernel density estimation is to replace the empirical measure $\mu_n=\sum_{i=1}^n\delta_{x_i}$, where $x_i\iid\mu$, with a smooth approximation; in particular, $\mu_n(A) = \frac{1}{n} \sum_{i=1}^n \delta_{\bm{x}_i}(A) \approx \int_A \frac{1}{n} \sum_{i=1}^n K_h(\bm{x} - \bm{x}_i) \, \dd \bm{x}$.
Formally, for $n$ large enough $\mu_n\approx \mu$ and then one expects $\frac{1}{n} \sum_{i=1}^n K_h(\cdot - \bm{x}_i)$ to approximate the density $\rho$ of $\mu$.
Analysing this approximation has been the interest of many statisticians, see for example~\cite{jiang17b,rinaldo10,GINE2002,Silverman1978,stute84,nadaraya65,tsybakov08}.
Our methods are an adaptation of the results by Gin\'{e} and Guillon~\cite{GINE2002}.
We refer to~\cite{devroye01} for an overview on kernel density estimation.

The second type of density estimation we consider is a regularised version of the kernel density estimate.
Our method is to use the kernel density estimate to estimate the value of the density at knot points (that we are free to choose), we then use smoothing splines to produce an estimate of the density.
We call this method the \emph{spline kernel density estimate}.
Although we are unaware of previous work using splines to estimate the density of distributions the idea of introducing regularisation in density estimation is not new, see for example~\cite{whittle58}.

The advantage of using splines is that they introduce additional smoothness into our estimate of the density which allows for better approximations of the density and, due to the volume of results in the literature, are theoretically well understood.
Furthermore, we can take advantage of fast computational methods for solving the spline smoothing problem.
We refer to~\cite{wahba90} for an overview and mention a few select references here.
Convergence in norm of special splines under various settings have been studied in~\cite{bissantz04,bissantz07,claeskens09,hall05,kauermann09,lai13,lukas06,wang11,arcangeli93}, and general splines in~\cite{wahba85,kimeldorf70,cox88,nychka89,carroll91,mair96}.
Similarly weak (pointwise) convergence of special splines has been studied in~\cite{li08,shen11,xiao12,yoshida12,yoshida14} and general splines in~\cite{Thorpe2017}.
\vspace{\baselineskip}

Due to its interest in machine learning the convergence of variational problem~\eqref{eq:Intro:conGraphpDirichlet} to~\eqref{eq:Intro:conpDirichlet} has attracted much interest.
For example, pointwise convergence results have been used to motivate the choice of $p>d$~\cite{NaSrZh09,alamgir11,elalaoui16,belkin07,coifman1,GK,hein_audi_vlux05,hein06,singer06,THJ}; however, pointwise convergence is not enough, in general, to imply variational convergence (convergence of minimisers).
Spectral convergence~\cite{belkin07,vonluxburg08,PelPud11,SinWu13,garciatrillos15aAAA,BIK} (and error bounds~\cite{wang14AAA,GGHS}) shows convergence of minimisers only when $p=2$.
The framework to analyse the discrete-to-variational was developed by Garc\'ia-Trillos and Slep\v{c}ev~\cite{garciatrillos16} and later applied to the constrained problem to show variational convergence when $p>d$ and $\eps_n$ satisfies an upper bound~\cite{SlepcevThorpe2017}.
Using PDE methods Calder \cite{floresrios19AAA} studies the large data limits of two closely related problems.
The first is Lipschitz learning (which corresponds to choosing $p=\infty$)~\cite{calder17AAA}, and the second is the game theoretic $p$-Dirichlet energy~\cite{Calder19}.
\vspace{\baselineskip}

In this paper we show that minimisers of the continuum $p$-Dirichlet energy, in which the density is estimated from the data, converge to a minimiser of the continuum $p$-Dirichlet energy in the large data limit.
After setting up notation and listing the main results in Section~\ref{sec:MainResults}, we do this in two parts.
The first part, in Section~\ref{sec:ConvMini}, gives sufficient conditions for the convergence of $\rho_n\to\rho$ to imply the convergence of minimisers of $\cEpinftyCon(\cdot;\rho_n)$ to minimisers of $\cEpinftyCon(\cdot;\rho)$, and for minimisers of $\cFpnCon(\cdot;\rho_n)$ to converge to minimisers of $\cEpinftyCon(\cdot;\rho)$.
This complements the results of~\cite{SlepcevThorpe2017} which prove convergence of minimisers of $\cEpnCon$ to minimisers of $\cEpinftyCon(\cdot;\rho)$ via the intermediary functional $\cFpnCon(\cdot;\rho)$ (see Figure~\ref{fig:pDirichletDiagram} for a summary).
Then, in the second part, we provide two examples of density estimation schemes, the kernel density estimate and the spline kernel density estimate, that satisfy the conditions in the previous section (see Section~\ref{sec:DensityEstimates}).
Numerical illustration of the results in two dimensions are provided in Section~\ref{sec:Numerics} and we conclude in Section~\ref{sec:Conc}.

\tikzstyle{format} = [draw, thin]
\tikzstyle{medium} = [ellipse, draw, thin, fill=green!20, minimum height=2.5em]

\begin{figure}[ht!]
\centering
\begin{tikzpicture}[node distance=4cm, auto,>=latex', thick]
\path[->] node[format] (discrete) {$\cEpnCon(f)$};
	\path[->] node[format, right of=discrete] (nonLocal)  {$\cFpnCon(f;\rho)$}
	(discrete) edge node[text width=4cm,align=center] {``$n\to\infty$, $\eps_n=\eps$'',\\ \cite{SlepcevThorpe2017}} (nonLocal);
	\path[->] node[format, right of=nonLocal] (continuum)
	{$\cEpinftyCon(f;\rho)$}
	(nonLocal) edge node[text width=1.6cm,align=center] {``$\eps_n\to0$'',\\ \cite{SlepcevThorpe2017}} (continuum);
	\path[->] node[format, below of=nonLocal] (nonLocalDE)
	{$\cFpnCon(f;\rho_n)$}
	(nonLocalDE) edge node[text width=3.5cm,align=center] {``$n\to\infty,\eps_n\to0$'',\\
		(Theorem~\ref{thm:MainResults:minpLapConvergence})} (continuum);
	\path[->] node[format, below of=continuum] (continuumDE)
	{$\cEpinftyCon(f;\rho_n)$}
	(continuumDE) edge[right] node[text width=3.5cm,align=center]
	{``$n\to\infty$'',\\ (Theorem~\ref{thm:MainResults:Minimizersnp})} (continuum);
\end{tikzpicture}
\caption{Diagram depicting how various $p$-Dirichlet energies discussed in this paper are related, and where their convergence results can be found.}\label{fig:pDirichletDiagram}
\end{figure}
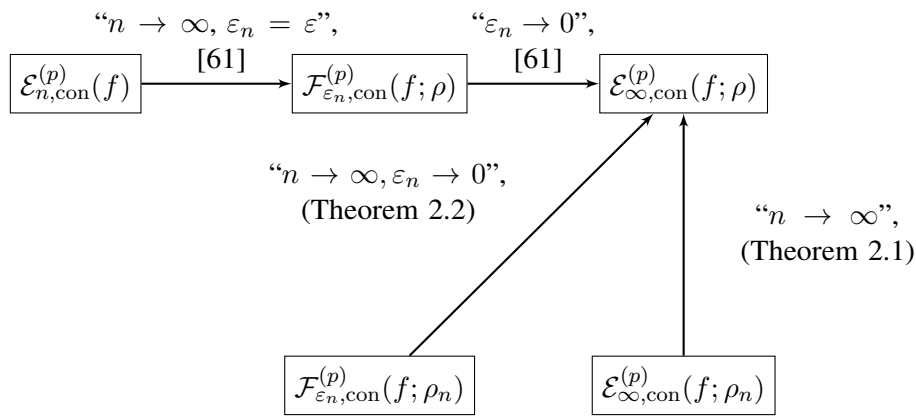

\section{Setting and Main Results} \label{sec:MainResults}

\subsection{Notation}

We consider functions on an open, bounded and connected domain $\Omega\subseteq\mathbb{R}^d$ with Lipschitz boundary.
Given a positive Radon measure $\mu\in\cM_+(\Omega)$ (where usually $\mu$ is a probability measure) we let $L^p(\mu)$ denote the space of functions for which the $p^{\text{th}}$ power of the absolute value is integrable with respect to $\mu$ and the usual norm $\|f\|_{L^p(\mu)}$.
When $\mu=\cL\lfloor_\Omega$, the Lebesgue measure on $\Omega$ we write, with a small abuse of notation, $L^p(\Omega)$ instead of $L^p(\cL\lfloor_\Omega)$.
Sobolev spaces, denoted by $W^{m,p}(\Omega)$, are the space of functions where the $p^{\text{th}}$ power of the absolute value of the first $m$ (weak) derivatives are integrable with respect to the Lebesgue measure.
When $p=2$ we also write $W^{m,2}= H^m$.
The norm $\|f\|_{W^{m,p}(\Omega)}$ on $W^{m,p}$ is defined in the usual way.
We use $C^{0,\alpha}$ to define the H\"{o}lder space with norm $\|f\|_{C^{0,\alpha}(\Omega)}$.

Throughout we assume we have a data set $\Omega_n=\{\bm{x}_i\}_{i=1}^n$ where $\bm{x}_i\iid \mu$ and $\mu\in \cP(\Omega)$.
Given such a data set we define the \emph{empirical measure} by
\[ \mu_n(\bm{x}) = \frac{1}{n}\sum_{i=1}^n \delta(\bm{x}-\bm{x}_i), \]
where $\delta(\bm{x})$ is the Dirac function.

We write $\Omega^\prime\subset\subset\Omega$ to mean that $\Omega^\prime$ is a compact subset of $\Omega$ and $\dH$ is the Hausdorff distance between sets in $\bbR^d$.

\subsection{Dirichlet Energies: Setup}

The following definition is used to construct the weights of the graph given data $\Omega_n = \{\bm{x}_i\}_{i=1}^n$.

\begin{Definition}
	\label{def:weightFunction}
	A function $\eta:[0,+\infty)\to[0,+\infty)$ is a \textbf{weight function} if it is a decreasing function with $\lim_{r\rightarrow+\infty}\eta(r)=0$, and is positive and continuous at $r=0$.
\end{Definition}

We prescribe weights between the points $\bm{x}_i$ and $\bm{x}_j$ using a weight function $\eta$: for a fixed $\eps>0$, the weight between two points $\bm{x}_i$ and $\bm{x}_j$ is defined by
\begin{align} 
	W_{ij} = \eta_{\eps}(|\bm{x}_i-\bm{x}_j|), \label{eq:Weights}
\end{align}
where $\eta_{\eps}(\cdot) = \frac{1}{\eps^d}\eta\l\frac{\cdot}{\eps}\r$.
This can be used to define a maximum distance $\eps$ for which particles have a non-zero weight; for example $\eta(t)=1$ for $t<1$ and $\eta(t) = 0$ otherwise. Then then $W_{ij}$ is positive only when $|\bm{x}_i-\bm{x}_j|<\eps$.
In the remainder of this section we define the various Dirichlet energies which are considered in the sequel.

\begin{Definition}
\label{def:MainResults:cEpn}
Let $p\in (1,+\infty)$, $\Omega_n = \{\bm{x}_i\}_{i=1}^n \subset \bbR^d$ and define $\mu_n$ to be the empirical measure and $W_{ij}$ as in~\eqref{eq:Weights} where $\eta$ is a weight function.
We define the \textbf{discrete $p$-Dirichlet energy} by
\begin{align*}
	& \cEpn: L^p(\mu_n) \to [0,+\infty) \notag \\
	& \cEpn(f) = \frac{1}{\eps^pn^2}\sum_{i,j=1}^n W_{ij}|f(\bm{x}_i)-f(\bm{x}_j)|^p.
\end{align*}
\end{Definition}

When the data generating distribution $\mu$ has density $\rho$ (with respect to the Lebesgue measure) then the large data limit, in the sense of $\Gamma$-convergence, is given by the continuum $p$-Dirichlet energy $\cEpinfty$ defined below.
Indeed, when $p=1$ it was shown in~\cite{garciatrillos16} that $\Glim_{n\to+\infty} \cE^{(1)}_n = \cE^{(1)}_\infty(\cdot;\rho)$ where $\cE^{(1)}_\infty(\cdot;\rho)$ is a weighted total variation (and takes a slightly different form to the class of energies given below).
The proof, however, generalises to any $p\in [1,+\infty)$. 

\begin{Definition}
\label{def:MainResults:cEpinfty}
Let $p\in (1,+\infty)$, $\Omega\subset\bbR^d$ be an open, bounded and connected domain with Lipschitz boundary, and $\rho\in L^\infty(\Omega)$ be a non-negative function.
Let $\eta$ be a weight function and assume that
\begin{equation}\label{eq:etaIntegral}
	\sigma_\eta := \int_{\bbR^d}\eta(|\bm{x}|)|\bm{x}\cdot \bm{e}_1|^p\diff \bm{x} <+\infty,
\end{equation}
where $\bm{e}_1 = (1,0,\dots,0)$.
We define the \textbf{continuum $p$-Dirichlet energy with respect to $\rho$} by
\begin{align}
	& \cEpinfty(\cdot;\rho):L^{p}(\Omega)\to [0,+\infty], \notag \\
	& \cEpinfty(f;\rho) = \begin{cases} \sigma_\eta\int_\Omega|\nabla f(\bm{x})|^p \rho^2(\bm{x}) \diff \bm{x} & \text{if } f\in W^{1,p}(\Omega),\\ +\infty & \text{else} \end{cases} \label{eq:MainResults:cEpinfty}
\end{align}
\end{Definition} 

The parameter $p$ controls the amount of regularity.
The results of this paper concern a finite choice of $p$ since for $p=+\infty$ the Dirichlet energy loses sensitivity to the density of the data.
In fact, it is easy to show that the variational limit as $p\to+\infty$ (with $n$ fixed and after renormalising with respect to $p$) is the Lipschitz learning problem:
\[ \cE^{(\infty)}_n(f) = \max_{i,j\in \{1,\dots, n\}} W_{ij} |f(\bm{x}_i)-f(\bm{x}_j)|, \]
see for example~\cite{cristoferi18} for the computation with a similar objective.
The objective of this paper is to build numerical methods by estimating the density of data.
For Lipschitz learning the data distribution appears in the continuum limit only through it's support; in particular, the intensity is irrelevant.
More precisely, it is known (see~\cite{elalaoui16} for pointwise limits and \cite{calder17AAA} for variational limits in the semi-supervised setting) that the large data limit of $\cE^{(\infty)}_n$ is
\[ \cE^{(\infty)}_\infty(f) = \sup_{\bm{x}\in\Omega} |\nabla f(\bm{x})|. \]
Hence, one needs only to estimate the support of the data, not the density and so Lipschitz learning falls outside the scope of our method.
\vspace{\baselineskip}

We work in the semi-supervised setting; that is, we assume that we have labels $y_i$, $i=1,\dots, N$ for the first $N$ data points where $N$ is fixed.
To estimate labels at the remaining $n-N$ data points we use the Dirichlet energies to define a notion of regularity.
More precisely, we minimise the Dirichlet energies subject to agreeing with the training data:
\begin{equation} 
\text{minimise } \cEpn(f) \qquad \text{subject to} \qquad f(\bm{x}_i) = y_i \quad \forall i=1,\dots, N. \label{eq:minimizationProblem}
\end{equation}
Analogously for the continuum Dirichlet energy.
It will be convenient to define the constrained energies as follows. 

\begin{Definition}
\label{def:MainResults:cEpnCon}
Under the setting and notation of Definition~\ref{def:MainResults:cEpn}, and given labels $y_i$ for $i=1,\dots, N$ we define the \textbf{constrained discrete $p$-Dirichlet energy} $\cEpnCon:L^p(\mu_n)\to [0,+\infty]$ by \cref{eq:Intro:conGraphpDirichlet}.
\end{Definition}

\begin{Definition}
In addition to the setting and notation of Definition~\ref{def:MainResults:cEpinfty} assume that $p>d$.
Then given labels $y_i$ for $i=1,\dots, N$ we define the \textbf{constrained continuum $p$-Dirichlet energy with respect to $\rho$} $\cEpinftyCon:L^p(\Omega)\to [0,+\infty]$  by \cref{eq:Intro:conpDirichlet}
\end{Definition}

Note that we require $p>d$ in order for the constrained continuum $p$-Dirichlet energy to be well defined.
More precisely, for any $f$ with $\cEpinfty(f;\rho)<+\infty$ we necessarily have that $f\in W^{1,p}(\Omega)$ and hence by Sobolev embedding (Morrey's inequality) $f$ can be identified with a continuous function, and therefore pointwise evaluation $f(\bm{x}_i)$ can be defined.
For $p\leq d$ the constrained Dirichlet energy can no longer be defined in the continuum setting.
\vspace{\baselineskip}

We also define the non-local continuum approximation $\cFpn$ of $\cEpn$.
This has been used as an intermediary functional in the discrete-to-continuum analysis of the $p$-Dirichlet energies, for example~\cite{garciatrillos16,SlepcevThorpe2017} (as mentioned in Section~\ref{Sec:Intro}), but is also of interest in its own right.

\begin{Definition}
\label{def:MainResults:cFpn}
Let $p\in (1,+\infty)$, $\Omega\subset\bbR^d$ be an open, bounded and connected domain with Lipschitz boundary, and $\rho\in L^\infty(\Omega)$ be a non-negative function.
Let $\eta$ be a weight function and then, we define the \textbf{non-local continuum $p$-Dirichlet energy with respect to $\rho$} by
\begin{align*}
	& \cFp(\cdot,\rho):L^p(\Omega)\to [0,+\infty), \\
	& \cFp(f;\rho) = \frac{1}{\eps^p}\int_\Omega\int_\Omega \eta_{\eps_n}(|\bm{x}-\bm{z}|)|f(\bm{x})-f(\bm{z})|^p\rho(\bm{x})\rho(\bm{z})\diff \bm{x} \diff \bm{z}.
\end{align*}
\end{Definition}

We note that we are no longer able to impose pointwise constraints on $\cFpn$.
Although $\cFpn$ is approximating a Sobolev semi-norm (when $\eps$ is small) we are still working on an $L^p$ space with no continuity implied, and therefore one cannot impose pointwise constraints
We overcome this by instead imposing the constraints on small balls around $\bm{x}_i$, $i=1,\dots, N$.
For our analysis we require that the balls have radius at least $\eps$ which leads us to define the constrained non-local continuum model as follows.

\begin{Definition}
Under the setting and notation of Definition~\ref{def:MainResults:cFpn}, and given labels $y_i$ for $i=1,\dots, N$ we define the \textbf{constrained non-local continuum $p$-Dirichlet energy with respect to $\rho$} by
\begin{align*}
	& \cFpCon:L^p(\Omega)\to [0,+\infty], \notag \\
	& \cFpCon(f;\rho) = \begin{cases} \cFpn(f;\rho), & \text{if } f(\bm{x}) = y_i \text{ for } i=1,\dots,N, \text{ and } \bm{x}\in B(\bm{x}_i,\eps), \\ +\infty, & \text{else.} \end{cases}
\end{align*}
\end{Definition}

In the above definition we make the assumption that $B(\bm{x}_i,\eps)\cap B(\bm{x}_j,\eps) = \emptyset$ for all $i,j=1,\dots, N$.
This is clearly satisfied for $\eps$ sufficiently small.

\subsection{Large Data Asymptotics for Dirichlet energies}

For the results given in this section we make the following assumptions:
\begin{enumerate}[label=(A{{\arabic*}})]
	\item\label{Assum1} $\Omega\subset\mathbb{R}^d$ is an open, connected, bounded domain with Lipschitz boundary, with associated probability measure $\mu\in\mathcal{P}(\Omega)$.
	\item\label{Assum2} The probability measure $\mu$ has continuous density $\rho\in L^\infty(\Omega)$.
	\item\label{Assum3} The density $\rho$ is bounded above and below by strictly positive constants.
	\item\label{Assum4} For $i=1,..,N$, the points $\bm{x}_i\in\Omega$ are labelled with values $y_i\in\bbR$.
	\item\label{Assum5} For $n \ge i>N$, the points $\bm{x}_i\in\Omega$ are i.i.d. samples of $\mu$.
	\item\label{Assum6} $\eta:[0,\infty)\rightarrow[0,\infty)$ is a weight function, and weights $W_{ij}$ are defined by~\eqref{eq:Weights} for $i,j=1,\dots,n$ and $\eps=\eps_n$.
	\item\label{Assum7} The integral $\sigma_\eta$ as defined in~\eqref{eq:etaIntegral} is finite.
	\item\label{Assum8} The smoothing parameter takes a value larger than the dimension of the data, $p>d$.
\end{enumerate}
Under assumptions \labelcref{Assum1,Assum2,Assum3,Assum4,Assum5,Assum6,Assum7,Assum8}, and the following scaling on $\eps_n = \eps$,
\begin{equation} \label{eq:MainResults:epsilonBounds}
	\l\frac{1}{n}\r^{\frac{1}{p}} \gg \eps_n \gg \lb \begin{array}{ll} \l\frac{\log(n)}{n}\r^{\frac{1}{d}} & \text{if } d\geq 3 \\ \frac{\l\log(n)\r^{\frac{3}{4}}}{\sqrt{n}} & \text{if } d=2 \end{array} \rd
\end{equation}
minimisers of~\eqref{eq:Intro:conGraphpDirichlet} converge to minimisers of~\eqref{eq:Intro:conpDirichlet}~\cite{SlepcevThorpe2017}.
Furthermore, when $p<d$, minimisers of~\eqref{eq:Intro:conGraphpDirichlet} converge to minimisers of~\eqref{eq:MainResults:cEpinfty} (i.e. constants), and so the constraints are lost as $n\to\infty$. This result allows us to approximate minimisers of~\eqref{eq:Intro:conGraphpDirichlet} by its continuum analogue,~\eqref{eq:Intro:conpDirichlet}. However, in general we may not know the density $\rho$.
To make use of the continuum formulation for finite data, it is therefore necessary to estimate the density $\rho$ using the information available; the data points $\bm{x}_i$, $i=1,\dots,n$. This is the focus of the first main result of this paper.

For this result, we include the following assumption on the estimate of $\rho$:
\begin{enumerate}[label=(A{{\arabic*}})]\setcounter{enumi}{8}
\item\label{Assum9} The density estimate $\rho_n:\Omega\to\mathbb{R}\in L^\infty(\Omega)$, satisfies $\sup_{n\in\bbN} \|\rho_n\|_{L^\infty(\bbR^d)}<+\infty$ and $\rho_n\to\rho$ in $L^\infty_\loc(\Omega)$, {\em i.e.}, for all $\Omega^\prime\subset\subset \Omega$
\[	\sup_{\bm{x}\in\Omega^\prime}|\rho_n(\bm{x}) - \rho(\bm{x})|\to 0. \]
\end{enumerate}

Under the above assumptions we can prove the convergence of minimisers of $\cEpinftyCon(\cdot;\rho_n)$ to a minimiser of $\cEpinftyCon(\cdot;\rho)$.
A sequence $f_n$ is a sequence of almost minimisers if there exists $\delta_n\to 0^+$ such that
\[ \min_{f\in W^{1,p}} \cEpinftyCon(f;\rho_n) \geq \cEpinftyCon(f_n;\rho) - \delta_n. \]

\begin{theorem}[Convergence of minimisers of the local model]
\label{thm:MainResults:Minimizersnp}
Assume $\Omega,\mu,\eta,p,\rho, \rho_n$ and $\{\bm{x}_i\}_{i=1}^n$ satisfy Assumptions \ref{Assum1}-\ref{Assum9}.
Then, (i) minimisers of $\cEpinftyCon(\cdot;\rho_n)$ are precompact in $L^p_{\loc}(\Omega)$,
\[ \text{(ii)} \, \min_{f\in W^{1,p}} \cEpinftyCon(\cdot;\rho) = \lim_{n\to\infty}\min_{f\in W^{1,p}} \cEpinftyCon(\cdot;\rho_n), \]
and (iii) any converging sequence of almost minimisers of $\cEpinftyCon(\cdot;\rho_n)$ converges in $L^\infty_{\loc}(\Omega)$ to a minimiser of $\cEpinftyCon(\cdot,\rho)$.
\end{theorem}

Furthermore, we show an analogous non-local result.

\begin{theorem}[Convergence of minimisers of the non-local model] \label{thm:MainResults:minpLapConvergence}
Assume $\Omega,\mu,\eta,p,\rho, \rho_n$ and $\{\bm{x}_i\}_{i=1}^n$ satisfy Assumptions \ref{Assum1}-\ref{Assum9}.
Then, (i) minimisers of $\cFpnCon(\cdot;\rho_n)$ are precompact in $L^p_{\loc}(\Omega)$,
\[ \text{(ii)} \, \min_{f\in W^{1,p}} \cEpinftyCon(\cdot;\rho) = \lim_{n\to\infty}\inf_{f\in W^{1,p}} \cFpnCon(\cdot;\rho_n), \]
and (iii) any converging sequence of minimisers of $\cFpnCon(\cdot;\rho_n)$ converges in $L^\infty_{\loc}(\Omega)$ to a minimiser of $\cEpinfty(\cdot;\rho_n)$.
\end{theorem}

The proof of both theorems is given in Section~\ref{sec:ConvMini}.
In Sections~\ref{subsec:MainResults:DensSetup} and~\ref{subsec:MainResults:DensResults} we give two examples on how to construct density estimates that satisfy Assumption~\ref{Assum9} (with probability one).

\subsection{Density Estimates: Set up} \label{subsec:MainResults:DensSetup}

Both of the convergence results from the previous section rely on having a density estimate which converge locally uniformly with probability one.
We consider two examples of density estimates with this convergence property: the kernel density estimate (KDE) and the closely related spline kernel density estimate (SKDE). 

\paragraph{Kernel Density Estimate}
Recall the empirical measure $\mu_n = \frac{1}{n} \sum_{i=1}^n \delta(\cdot - \bm{x}_i)$.
The kernel density estimate (KDE) can be viewed as a continuous approximation to the empirical measure, where each Dirac function is approximated by a function with particular properties, known as a {\em kernel function}; which is a function integrating to unity, i.e. $\int_{\bbR^d} K(\bm{x}) = 1$.
A popular choice is the {\em Gaussian kernel function}
\[	K(x) = \frac{1}{(2\pi)^{d/2}}\exp\left(-\frac{\|\bm{x}\|^2}{2}\right). \]
Other popular choices include the uniform and Epanechnikov kernels.
In general kernel functions do not have to be symmetric or positive, we refer to~\cite{QiJeffrey2007,hansen09} for more examples of kernel functions.
In our numerical experiments in Section~\ref{sec:Numerics} we choose the Gaussian kernel.

We define the kernel density estimate as follows.

\begin{Definition}
\label{def:MainResults:KDE}
Given $\bm{x}_i\subset\Omega$ for $i=1,\dots,n$ and a bandwidth $h>0$, the {\em Kernel density estimate} $\rho_{n,h}:\Omega\to\bbR$ is defined by
\begin{equation} \label{eq:MainResults:KDE}
	\rho_{n,h}(\bm{x}) = \frac{1}{n}\sum_{i=1}^n K_{h}\left(\bm{x}-\bm{x}_i\right).
\end{equation}
where $K_{h}(\bm{x}) = \frac{1}{h^d}K\l\bm{x}/h\r$ and $K:\bbR^d\to\bbR$ integrates to unity.
\end{Definition}

We note that, with additional notational complexity one can generalise the bandwidth to a positive semidefinite matrix $\bm{H}$, i.e. 
\[ \rho_{n,\bm{H}}(\bm{x}) = \frac{1}{n}\sum_{i=1}^n K_{\bm{H}}\left(\bm{x}-\bm{x}_i\right) \]
where $K_{\bm{H}}(\bm{x}) = \frac{1}{|\bm{H}|}K\l H^{-1}\bm{x}\r$.
In the sequel we treat the special case where $\bm{H} = h\Id$.

When $h\to 0$ we regain the empirical measure $\mu_n$, i.e. $\lim_{h\to 0^+} \int_A \, \dd \rho_{n,h}(\bm{x}) = \mu_n(A)$ for all open sets $A$.
We shall see that to guarantee convergence to the continuous density $\rho$, we will require a lower bound on the rate at which $h\to 0$.
We state the convergence result in the next subsection.

\paragraph{Spline Kernel Density Estimate}
To find the minimiser of $\cEpinftyCon(\cdot;\rho_n)$ we use a gradient flow which involves the density estimate $\rho_n$ and its derivative.
It is therefore of interest to have a smooth approximation of $\rho$.
Our strategy is to regularise the kernel density estimate.

One way to do this is to solve the variational problem:
\[ \text{minimise } \| u - \rho_{n,h}\|_{L^2(\Omega)}^2 + \lambda \|\nabla^m u\|_{L^2(\Omega)}^2 \quad \text{over } u\in H^m(\Omega). \]
Drawing inspiration from the spline smoothing community we approximate the first term by 
\[ \| u - \rho_{n,h}\|_{L^2(\Omega)}^2 \approx \frac{1}{T} \sum_{i=1}^T |u(\bm{t}_i) - \rho_{n,h}(\bm{t}_i) |^2 \]
where $\{\bm{t}_i\}_{i=1}^T$ are called \emph{knot points} (which we are free to choose).
We therefore consider the variational problem:
\begin{equation} \label{eq:MainResults:SlambdaT}
S_{\lambda,T}(f) := \argmin_{u\in H^m(\Omega)}\left\{ \frac{1}{T} \sum_{i=1}^T (u(\bm{t}_i)-f_i)^2 + \lambda\|\nabla^m u\|_{L^2(\Omega)}^2 \right\}.
\end{equation}
We define the projection operator
\begin{equation} \label{eq:MainResults:PT}
P_T:H^m(\Omega) \to \bbR^T, \qquad P_T(f) = f(\bm{t}_i)
\end{equation}
which is well-defined whenever $m>d/2$.

\begin{Definition}
\label{def:MainResults:SKDE}
Given a set of knot points $\{\bm{t}_i\}_{i=1}^T \subset \Omega$ and the kernel density estimate $\rho_{n,h}$ we define the \emph{spline kernel density estimate} (SKDE) by
\begin{equation} \label{eq:SKDE}
\rho_{n,h,\lambda,T} = S_{\lambda,T}(P_T(\rho_{n,h}))
\end{equation}
where $S_{\lambda,T}$ is defined by~\eqref{eq:MainResults:SlambdaT} and $P_T$ is defined by~\eqref{eq:MainResults:PT}.
\end{Definition}

In the following subsection we give $L^\infty$ convergence rates with probability 1 for the SKDE (in fact these results are a corollary of almost sure convergence results in $H^m$ which when $m>d/2$ imply uniform convergence via Sobolev embeddings - we refer to Section~\ref{sec:DensityEstimates} for details).
\vspace{\baselineskip}

We give a numerical comparison of the KDE and the SKDE in Section~\ref{sec:Numerics}.
Moreover, we use the KDE and SKDE estimates to construct numerical methods to calculate the minimisers of~\eqref{eq:Intro:conpDirichlet} with $\rho_n = \rho_{n,h}$ and $\rho_n = \rho_{n,h,\lambda,T}$, and provide examples showing rates of convergence and computational efficiency.

\subsection{Large Data Asymptotics for Density Estimation} \label{subsec:MainResults:DensResults}

We first discuss the almost sure locally uniform convergence of the kernel density estimate.
This result follows almost immediately from known results in the literature.
In particular, it is known from~\cite{GINE2002} that the $L^\infty$ norm between the KDE and its expected value converges with a certain rate to 0.
From here it is not difficult to show that the KDE converges locally uniformly to its true value.

The conditions we use to prove the convergence of the KDE estimate are the following.

\begin{enumerate}[label=(B{{\arabic*}})]
	\item\label{AssumB1} $\Omega\subset\bbR^d$ is an open and bounded domain with Lipschitz boundary, with associated probability measure $\mu\in\cP(\Omega)$. 
	\item\label{AssumB2} The probability measure $\mu$ has a bounded density $\rho$.
	\item\label{AssumB3} For $i=1,\dots$, the points $\bm{x}_i\in \Omega$ are i.i.d. samples of $\mu$.
	\item\label{AssumB4} $K:\bbR^d\to \bbR$ satisfies $\int_{\bbR^d} K(\bm{x}) \, \dd \bm{x} = 1$, has compact support in $B(0,M)$ for some $M>0$ and can be written $K=\phi\circ \xi$ where $\phi$ is a bounded function of bounded variation and $\xi$ is a polynomial. 
	\item\label{AssumB5} $h=h_n$ satisfies
\[	h_n \to 0^+, \quad \frac{nh^d_n}{|\log(h_n)|} \to \infty, \quad \frac{|\log(h_n)|}{\log\log(n)} \to \infty,\quad \text{and} \quad h_n \leq c h_{2n} \]
for some $c>0$.
	\item\label{AssumB6} $K$ satisfies the integrability condition $\int_{\bbR^d} |K(\bm{x})| \|\bm{x}\| \, \dd \bm{x}<+\infty$.
	\item\label{AssumB7} $\rho$ is continuous on $\Omega$.
\end{enumerate}

Assumptions~\ref{AssumB1}-\ref{AssumB5} are the same as those used in~\cite{GINE2002} to prove the almost sure convergence of the bias:
\[ \Bias(\rho_{n,h}) = \sup_{\bm{x}\in\Omega} |\bbE [\rho_{n,h}(x)] - \rho_{n,h}(x)| \]
to zero.
We use Assumptions~\ref{AssumB6}-\ref{AssumB7} to show that $\bbE [\rho_{n,h}] \to \rho$ locally uniformly with probability one.

In fact the assumption that $K= \phi\circ \xi$ where $\phi$ is bounded and of bounded variation and $\xi$ is a polynomial can be relaxed.
Our result is a simple application of~\cite[Theorem 2.3]{GINE2002} which uses a class of kernels that satisfy a technical condition that is sufficient to bound the Vapnik-\v{C}ervonenkis (VC) dimension of functions of the form $K\l\frac{\bm{x}-\cdot}{h}\r$ for $\bm{x}\in\bbR^d$ and $h>0$.
As the authors remark the technical assumption is satisfied for functions of the form $K= \phi\circ \xi$ and since this includes the kernels we are interested in, e.g. Gaussian kernels, we satisfy ourselves with this less general case that can be stated more easily.

We state the converge result for the KDE here, the proof is given in Section~\ref{subsec:DensityEstimates:KDE}.

\begin{theorem}
\label{thm:MainResults:DensResults:KDE}
Assume $\Omega,\mu,\rho,K,h_n$ and $\{\bm{x}\}_{i=1}^n$ satisfy Assumptions~\ref{AssumB1}-\ref{AssumB7}.
Define $\rho_{n,h}$ as in Definition~\ref{def:MainResults:KDE}.
Then, with probability one, for all $\Omega^\prime\subset\subset \Omega$ we have
\[ \lim_{n\to \infty} \| \rho_{n,h_n} - \rho \|_{L^\infty(\Omega^\prime)} \to 0. \]
\end{theorem}

We now turn our attention to the spline kernel density estimate.
To prove convergence we will need some additional assumptions which we state now.

\begin{enumerate}[label=(B{{\arabic*}})]\setcounter{enumi}{7}
	\item\label{AssumB8} The number of derivatives penalised is greater than half the dimension of the data, $m>d/2$.
	\item\label{AssumB9} Let $\dH(T) = \dH(\{\bm{t}_i\}_{i=1}^T,\Omega)$, where $\dH$ is the Hausdorff distance, and
\[ \Sep(T) = \min\lb |\bm{t}_i-\bm{t}_j| \,:\, i\neq j, i,j \in \{1,\dots, T\} \rb \]
then $T_n$ and $h_n$ satisfy
\[ \dH(T_n) \to 0 \qquad \text{and} \qquad \dH(T_n) = O(\Sep(T_n)) \]
	\item\label{AssumB10} $\lambda_n\to 0^+$ satisfies
\[ \lambda_n T_n \dH(T_n)^d = O(1), \quad T_n \lambda_n^{\frac{2m+d}{2m}} \gg n^\theta \quad \text{and} \quad T_n \lambda_n^{\frac{d}{2m}} \geq 1 \]
for some $\theta>0$.
	\item\label{AssumB11} $\Omega$ satisfies the uniform cone condition: i.e. there exists $r>0$ and $\tau>0$ such that for any $\bm{t}\in\Omega$, there exists a unit vector $\bm{\xi}(\bm{t})\in\bbR^d$ such that the cone
\[ C(\bm{t},\bm{\xi}(t),\tau,r) = \{\bm{t} + \lambda\bm{\eta}\,:\,|\bm{\eta}|=1, \bm{\eta}\cdot\bm{\xi}(\bm{t})\ge\cos(\tau),0\le\lambda\le r\} \]
is fully contained in $\Omega$.
	\item\label{AssumB12} $T_n$ and $h_n$ satisfy $\Sep(T_n) \geq 2Mh_n$ where $M$ is given in Assumption~\ref{AssumB4}.
	\item\label{AssumB13} $\rho$ is Lipschitz continuous on $\Omega$ and $\rho\in H^m(\Omega)$.
\end{enumerate}

These conditions are necessary to apply the spline smoothing results of~\cite{UTRERAS1988} and~\cite{arcangeli93}; in particular Assumptions~\ref{AssumB1}, \ref{AssumB8}-\ref{AssumB11} are needed for the $H^k$ convergence of splines in~\cite{arcangeli93} and Assumptions~\ref{AssumB12}-\ref{AssumB13} are used to match our setting here to their setting.
We note that if we distribute the knot points uniformly over $\Omega$ then $\dH(T_n)\sim \frac{1}{T_n^{\frac{1}{d}}}$ and $\Sep(T_n)\sim\frac{1}{T_n^{\frac{1}{d}}}$ which satisfy Assumption~\ref{AssumB9}.
Assumption~\ref{AssumB10} is satisfied for any $\lambda_n\to 0^+$ with $\lambda_nT_n^{\frac{2m}{d+2m}} \gg n^\theta$ for some $\theta>0$.
The result we are interested in is stated below, the proof is a simple corollary of Theorem~\ref{thm:DensityEstimates:SKDE:HmConv} in Section~\ref{subsec:DensityEstimates:SKDE} which gives convergence in $H^k_\loc(\Omega)$.

\begin{theorem}
\label{thm:MainResults:DensResults:SKDE}
Assume $\Omega,\mu,\rho,K,h_n,T_n,\lambda_n,m,\{\bm{t}_i\}_{i=1}^{T_n}$ and $\{\bm{x}\}_{i=1}^n$ satisfy Assumptions~\ref{AssumB1}-\ref{AssumB13}.
We define $\rho_{n,h_n,\lambda_n,T_n}$ as in Definition~\ref{def:MainResults:SKDE}.   
Then, with probability one, for all $\Omega^\prime\subset\subset \Omega$ we have
\[ \|\rho_{n,h_n,\lambda_n,T_n} - \rho\|_{L^\infty(\Omega^\prime)} \to 0. \]
\end{theorem}

\section{Convergence of Minimisers} \label{sec:ConvMini}

To show that minimisers of $\cEpinftyCon(\cdot;\rho_n)$ and $\cFpnCon(\cdot;\rho_n)$, where $\rho_n\to\rho$, converge to minimisers of $\cEpinftyCon(\cdot;\rho_n)$ we require a notion of convergence for functionals.
For variational convergence the correct notion is $\Gamma$-convergence that we recall now.

\begin{Definition}
Let $(\cX,d)$ be a metric space.
Let $F_n:\cX\to\bbR$ for each $n\in \bbN$.
We say that $(F_n)_{n\in\bbN}$ {\em $\Gamma$-converges} to $F:\cX\to\bbR$ and write $\Glim_{n\to+\infty} F_n=F$ if
\begin{enumerate}
\item (liminf inequality) for every $x\in\cX$ and every $(x_n)_{n\in\bbN}$ such that $x_n\to x$ in $\cX$,
\[ F(x) \leq \liminf_{n\to+\infty} F_n(x_n); \]
\item (existence of recovery sequences) for every $x\in\cX$, there exists some sequence $(x_n)_{n\in\bbN}$ such that $x_n\to x$ in $\cX$ and
\[ F(x)\geq \limsup_{n\to+\infty} F_n(x_n). \]
\end{enumerate}
\end{Definition}

The following result then provides conditions for convergence of minimisers, the proof can be found in, for example,~\cite{Braides2002,dalmaso93}.

\begin{theorem}
\label{thm:ConvMini:minimizers}
Let $(\cX,d)$ be a metric space and $F_n:\cX\to [0,+\infty]$ be a sequence of functionals.
Let $x_n$ be a minimising sequence for $F_n$.
If the set $\{x_n\}_{n=1}^\infty$ is precompact and $F_\infty = \Glim_{n\to+\infty} F_n$ where $F_\infty:\cX\to [0,+\infty]$ is not identically $+\infty$ then
\[ \min_\cX F_\infty = \lim_{n\to+\infty} \inf_\cX F_n. \]
Furthermore any cluster point of $\{x_n\}_{n=1}^\infty$ is a minimiser of $F_\infty$.
\end{theorem}

The following lemma will be useful when considering our cases of $\Gamma$-convergence, and compactness of minimisers.

\begin{lemma}[Morrey's Theorem]
	\label{lemma:ConvMini:RK}
	Let $\Omega\subset\mathbb{R}^d$ be an extension domain for $W^{1,p}(\Omega)$ with finite measure ({\em i.e.} there exist a bounded linear operator $E:W^{1,p}(\Omega)\to W^{1,p}(\mathbb{R}^d)$ such that $Ef|_{\Omega}=f$ on $\Omega$ for every $f\in W^{1,p}(\Omega)$).
	Let $\{f_n\}_{n\in\bbN}\subset W^{1,p}(\Omega)$ be a uniformly bounded sequence.
	Then, if $p>d$ there exist a subsequence $\{f_{n_k}\}_{k\in\bbN}$ of $\{f_n\}_{n\in\bbN}$ and a function $f\in C^{0,\alpha}(\Omega)$ such that $f_{n_k}\to f$ as $k\rightarrow\infty$ in $C^{0,\alpha}(\Omega)$, for any $0<\alpha<1-\frac{d}{p}$.
\end{lemma}

\begin{remark}
	We note that as $C^{0,\alpha}(\Omega)\subset L^{\infty}(\Omega)$, compactness in $L^\infty(\Omega)$ follows.
\end{remark}

To show compactness of minimisers, we require a slightly modified Poincar\'{e} inequality.

\begin{lemma}[Poincar\'{e} Inequality]
	\label{lem:ConvMini:Poincare}
	Let $p>d$ and $\Omega\subset\bbR^d$ be a connected extension domain for $W^{1,p}(\Omega)$ with finite measure.
	For $f\in W^{1,p}(\Omega)$, let
	\[ \bar{f} = \frac{1}{N}\sum_{i=1}^N f(\bm{x}_i) \]
	for a fixed set $\{\bm{x}_i\}_{i=1}^N\subset\Omega$.
	Then, there exists a constant $C$ such that for all $f\in W^{1,p}(\Omega)$, 
	\[  \|f-\bar{f}\|_{L^p(\Omega)} \leq C\|\nabla f\|_{L^p(\Omega)} \]
\end{lemma}

\begin{proof}
The proof is very similar to the proof of the Poincar\'{e} inequality found in \cite[Theorem 12.23]{Leoni2009}, we just check that one can take the average value of $f$ over finitely many points.
Assume for a contradiction that there exists a sequence $f_n\in W^{1,p}(\Omega)$ such that
\[ \|f_n-\bar{f}_n\|_{L^p(\Omega)} \geq n\|\nabla f_n\|_{L^p(\Omega)}>0. \]
Define a centralised, normalised sequence
\[ v_n = \frac{f_n-\bar{f}_n}{\|f_n-\bar{f}_n\|_{L^p(\Omega)}}, \]
then $v_n\in W^{1,p}(\Omega)$, $\|v_n\|_{L^p(\Omega)}=1, \bar{v}_n=0$, and
\[ \|\nabla v_n\|_{L^p(\Omega)} = \left\|\nabla\left( \frac{f_n-\bar{f}_n}{\|f_n-\bar{f}_n\|_{L^p(\Omega)}}\right)\right\|_{L^p(\Omega)} = \frac{\|\nabla f_n\|_{L^p(\Omega)}}{\|f_n-\bar{f}_n\|_{L^p(\Omega)}} \leq \frac{1}{n}. \]
Thus by Lemma~\ref{lemma:ConvMini:RK}, there exists a subsequence $\{v_{n_k}\}_{n\in\bbN}$ such that $v_{n_k}\to v$ in $L^{\infty}(\Omega)$.
Further, we must have $\|v\|_{L^p(\Omega)}=1$ and $\bar{v} = 0$.

Now consider a differentiable, compactly supported function $\varphi:\Omega\to\bbR$. Then for each derivative of $\varphi$ (using the Lebesgue dominated convergence theorem, integration by parts and H\"older's inequality)
\begin{align*}
\la \int_\Omega v\frac{\partial\varphi}{\partial x_i}\diff x \ra	& = \lim_{k\to\infty} \la \int_\Omega v_{n_k}\frac{\partial\varphi}{\partial x_i}\diff x \ra \\
 & = \lim_{k\to\infty} \la \int_\Omega \varphi\frac{\partial v_{n_k}}{\partial x_i}\diff x \ra \\
 & \leq \|\varphi\|_{L^{p^\prime}(\Omega)} \lda \frac{\partial v_{n_k}}{\partial x_i}\rda_{L^p(\Omega)} = 0.
\end{align*}
Then $\|\nabla v\|_{L^p(\Omega)}=0$ so $v$ is constant, and as $\bar{v}=0$ we must have that $v=0$ which contradicts $\|v\|_{L^p(\Omega)}=1$.
Hence the required result holds.
\end{proof}

\subsection{Convergence of the Local Model}

We now state the compactness property for $\cEpinftyCon(\cdot;\rho_n)$.
Compactness of minimisers is a corollary.

\begin{proposition}
\label{prop:ConvMini:CompactnpLap}
Assume that $\Omega,\mu,\eta,p,\rho_n$ and $\{\bm{x}_i\}_{i=1}^n$ satisfy Assumptions~\ref{Assum1}-\ref{Assum9}.
Then, any sequence $\{f_n\}_{n\in\bbN}$ satisfying $\sup_{n\in\bbN} \cEpinftyCon(f_n;\rho_n)<+\infty$ is bounded in $W^{1,p}(\Omega^\prime)$ and precompact in $C^{0,\alpha}(\Omega^\prime)$ for any $\Omega^\prime\subset\subset \Omega$ and any $0<\alpha<1-\frac{d}{p}$.
\end{proposition}

\begin{proof}
We show that $\{f_n\}_{n\in\bbN}$ is uniformly bounded in $W^{1,p}(\Omega^\prime)$.
Compactness in $C^{0,\alpha}(\Omega^\prime)$ then follows from Lemma~\ref{lemma:ConvMini:RK}.

First, consider $\|\nabla f_n\|_{L^p(\Omega^\prime)}$.
We note for sufficiently large $n$ that $\rho_n$ is strictly positive for a.e. $\bm{x}\in\Omega^\prime$, moreover by \labelcref{Assum3,Assum9}, 
\[	\rho_n(\bm{x})\geq\frac{\rho(\bm{x})}{\sqrt{2}}\geq \frac{\min_{\bm{x}\in\Omega}\rho(\bm{x})}{\sqrt{2}}, \]
for a.e. $\bm{x}\in\Omega^\prime$ and $n$ sufficiently large.
Hence,
\[ \cEpinftyCon(f_n;\rho_n) \geq \cEpinfty(f_n;\rho_n) = \sigma_\eta\int_{\Omega}|\nabla f_n(\bm{x})|^p\rho_n^2(\bm{x})\diff \bm{x} \geq \frac{\sigma_\eta}{2}\min_{\bm{x}\in\Omega}(\rho^2(\bm{x}))\|\nabla f_n\|_{L^p(\Omega^\prime)}^p. \]
Therefore $\sup_{n\in\bbN}\|\nabla f_n\|_{L^p(\Omega^\prime)}<+\infty$.

We are left to show $\sup_{n\in\bbN} \|f_n\|_{L^p(\Omega^\prime)}<+\infty$.
By Minkowski's inequality and Lemma~\ref{lem:ConvMini:Poincare}:
\[ \|f_n\|_{L^p(\Omega^\prime)} \leq \|f_n-\bar{f}_n\|_{L^p(\Omega^\prime)} + |\bar{f}_n| \sqrt[p]{\Vol(\Omega^\prime)} \leq C \l\|\nabla f_n\|_{L^p(\Omega^\prime)} + \frac{1}{N}\la \sum_{i=1}^N y_i \ra \r. \]
Since, by the previous argument, we can bound $\|\nabla f_n\|_{L^p(\Omega^\prime)}$ independently of $n$ we have that $\{f_n\}_{n\in\bbN}$ is bounded in $L^p(\Omega^\prime)$ as required.
\end{proof}

An immediate corollary of the previous result is that minimisers of $\cEpinftyCon(\cdot;\rho_n)$ are bounded.

\begin{corollary}
\label{cor:ConvMini:CompactnpLapMin}
Assume that $\Omega,\mu,\eta,p,\rho_n$ and $\{\bm{x}_i\}_{i=1}^n$ satisfy Assumptions~\ref{Assum1}-\ref{Assum9}.
Then, minimisers of $\cEpinftyCon(\cdot;\rho_n)$ are bounded in $W^{1,p}(\Omega^\prime)$ and precompact in $C^{0,\alpha}(\Omega^\prime)$ for any $\Omega^\prime\subset\subset \Omega$ and any $\alpha\in(0,1-\frac{d}{p})$.
\end{corollary}

\begin{proof}
Choose any $f^\dagger$ that smoothly interpolates between constraints, i.e. $\|f^\dagger\|_{W^{1,p}(\Omega)}<+\infty$ and $f^\dagger(\bm{x}_i)=y_i$ for all $i=1,\dots, N$.
Let $f_n$ be a minimiser of $\cEpinftyCon(\cdot;\rho_n)$.
Clearly $\cEpinftyCon(f_n;\rho_n) \leq \cEpinftyCon(f^\dagger;\rho_n) \leq \sigma_\eta \|\rho_n\|_{L^\infty(\Omega)}^2 \|\nabla f^\dagger\|_{L^p(\Omega)}^p$.
In particular $\sup_{n\in\bbN} \cEpinfty(f_n;\rho_n)<+\infty$, hence the result follows from Proposition~\ref{prop:ConvMini:CompactnpLap}.
\end{proof}

We now consider $\Gamma$-convergence for $\cEpinftyCon(\cdot;\rho_n)$.

\begin{lemma}
\label{lem:ConvMini:GammanpLap}
Assume that $\Omega,\mu,\eta,p,\rho_n$ and $\{\bm{x}_i\}_{i=1}^n$ satisfy Assumptions~\ref{Assum1}-\ref{Assum9}.
Then, $\cEpinftyCon(\cdot;\rho_n)$ $\Gamma$-converges to $\cEpinftyCon(\cdot;\rho)$.	
\end{lemma}

\begin{proof}
Let $\Omega^\prime\subset\subset \Omega$.
We note that for each $n$ there exists $\delta_n$ such that for a.e. $\bm{x}\in\Omega^\prime$,
\begin{equation} \label{eq:ConvMini:DensityBounds}
\rho(\bm{x})\l 1-\frac{\delta_n}{\inf_{\bm{x}\in\Omega}(\rho(\bm{x}))}\r \leq \rho_n(\bm{x}) \leq \rho(\bm{x})\l 1+\frac{\delta_n}{\inf_{\bm{x}\in\Omega}(\rho(\bm{x}))}\r
\end{equation}
where $\delta_n\to 0$ as $n\to+\infty$. This implies that
\begin{equation} \label{eq:ConvMini:EnergyBounds}
\l 1-\frac{\delta_n}{\inf_{\bm{x}}(\rho(\bm{x}))}\r^2 \cEpinfty(f;\rho\lfloor_{\Omega^\prime}) \leq \cEpinfty(f;\rho_n\lfloor_{\Omega^\prime}) \leq \l 1+\frac{\delta_n}{\inf_{\bm{x}}(\rho(\bm{x}))}\r^2 \cEpinfty(f;\rho\lfloor_{\Omega^\prime}).
\end{equation}
Note that $\cEpinfty(f;\rho) = \sigma_\eta \|\rho^{\frac{2}{p}} \nabla f \|_{L^p(\Omega)}^p$.
This identity and~\eqref{eq:ConvMini:EnergyBounds} will be used to prove the two conditions for $\Gamma$-convergence.

\paragraph{(Liminf inequality.)} 
Assume that $f_n\to f$ in $L^p(\Omega)$ and $\liminf_{n\to+\infty} \cEpinftyCon(f_n;\rho_n)<+\infty$ (else the result is trivial).
By recourse to a subsequence (not relabelled) we may assume that
\[ \liminf_{n\to+\infty} \cEpinftyCon(f_n;\rho_n) = \lim_{n\to+\infty} \cEpinftyCon(f_n;\rho_n) \]
and therefore by the compactness property, Proposition~\ref{prop:ConvMini:CompactnpLap}, we have that $\{f_n\}_{n\in\bbN}$ is bounded in $W^{1,p}(\Omega^\prime)$ and hence there exists a further subsequence (not relabelled) weakly converging in $W^{1,p}(\Omega^\prime)$ and, by Lemma~\ref{lemma:ConvMini:RK}, strongly in $L^\infty(\Omega^\prime)$.
Strong convergence in $L^\infty(\Omega^\prime)$ implies that $f$ must also satisfy the constraints $f(\bm{x}_i) = y_i$ for $i=1,\dots, N$ (where we assume $\bm{x}_i\in \Omega^\prime$ for all $i=1,\dots, N$).
We note also that $f_n\rho^{\frac{2}{p}}$ is weakly convergent in $W^{1,p}(\Omega^\prime)$.
Hence, by~\eqref{eq:ConvMini:EnergyBounds} and weak lower semi-continuity of norms
\begin{align*}
\liminf_{n\to+\infty} \cEpinftyCon(f_n;\rho_n) & \geq \liminf_{n\to \infty} \cEpinftyCon(f_n;\rho_n\lfloor_{\Omega^\prime}) \geq \liminf_{n\to+\infty} \cEpinftyCon(f_n;\rho\lfloor_{\Omega^\prime}) \\
 & = \liminf_{n\to+\infty} \sigma_\eta \|\nabla f_n \rho^{\frac{2}{p}}\|_{L^p(\Omega^\prime)}^p \geq \sigma_\eta \|\nabla f \rho^{\frac{2}{p}}\|_{L^p(\Omega^\prime)}^p.
\end{align*}
Taking $\Omega^\prime\to \Omega$ and applying Fatou's lemma we have
\[ \liminf_{n\to+\infty} \cEpinftyCon(f_n;\rho_n) \geq \sigma_\eta \|\nabla f \rho^{\frac{2}{p}}\|_{L^p(\Omega)}^p = \cEpinftyCon(f;\rho) \]
as required.

\paragraph{(Recovery sequence.)} 
For a given $f\in L^p(\Omega)$ we choose $f_n=f$ and applying the upper bound in~\eqref{eq:ConvMini:EnergyBounds}:
\[ \cEpinftyCon(f;\rho_n) \leq \l 1+\frac{\delta_n}{\inf_{\bm{x}\in \Omega} \rho(\bm{x})} \r^2  \cEpinftyCon(f;\rho) + \sigma_\eta \int_{\Omega\setminus\Omega^\prime} |\nabla f(x)|^p \rho_n^2(x) \, \dd x. \]
Hence,
\[ \limsup_{n\to\infty} \cEpinftyCon(f;\rho_n) \leq \cEpinftyCon(f;\rho) + C\sigma_\eta \lda \mathds{1}_{\Omega\setminus\Omega^\prime} \nabla f\rda^p_{L^p(\Omega)}. \]
Taking $\Omega^\prime\to\Omega$ and applying the dominated convergence theorem we get
\[ \lim_{\Omega^\prime \to \Omega} \lda \mathds{1}_{\Omega\setminus\Omega^\prime} \nabla f\rda^p_{L^p(\Omega)} \to 0 \]
as required.
\end{proof}

The proof of Theorem~\ref{thm:MainResults:Minimizersnp} is then a simple application of Theorem~\ref{thm:ConvMini:minimizers} to Corollary~\ref{cor:ConvMini:CompactnpLapMin} and Lemma~\ref{lem:ConvMini:GammanpLap}.

\subsection{Convergence of the Non-Local Model}

We start with the compactness result.

\begin{proposition}
\label{prop:ConvMini:midLapCompact}
Assume that $\Omega,\mu,\eta,p,\rho_n$ and $\{\bm{x}_i\}_{i=1}^n$ satisfy Assumptions~\ref{Assum1}-\ref{Assum9}.
Then, any sequence $\{f_n\}_{n\in\bbN}$ satisfying $\sup_{n\in\bbN}\cFpnCon(f_n;\rho_n)<+\infty$ is precompact in $L^p(\Omega^\prime)$ for any $\Omega^\prime\subset\subset \Omega$.
Furthermore, if $f$ is a cluster point of $\{f_n\}_{n\in\bbN}$ in $L^p(\Omega)$ then $f\in C^{0,\alpha}(\Omega^\prime)$ for any $0<\alpha<1-\frac{d}{p}$ and $f(\bm{x}_i) = y_i$ for all $i=1,\dots, N$.
\end{proposition}

\begin{proof}
By~\cite[Lemma 4.3]{SlepcevThorpe2017} there exists a mollifier $J$ such that $J$ has compact support in $\overline{B(0,1)}$, $J\leq C\eta$ for some $C$, and
\[ \cFpnCon(f_n;\rho_n) \geq C \cEpinfty(J_{\eps_n}\ast f_n;\mathds{1}_{\Omega^\prime}\rho_n) \]
for any $\Omega^\prime \subset\subset \Omega$ with $\dist(\Omega^\prime,\partial\Omega)>\eps_n$, where $\mathds{1}_{\Omega'}$ is an indicator function over $\Omega'$.
We have that $\sup_{n\in \bbN}\cEpinfty(\tilde{f}_n;\mathds{1}_{\Omega^\prime}\rho_n)<+\infty$ where $\tilde{f}_n=J_{\eps_n}\ast f_n$. 
Since, for all $i=1,\dots, N$,
\[ \tilde{f}_n(\bm{x}_i) = \int_{B(0,\eps_n)} J_{\eps_n}(\bm{z}) f_n(\bm{x}_i-\bm{z}) \, \dd \bm{z} = y_i \int_{B(0,\eps_n)} J_{\eps_n}(\bm{z}) \, \dd \bm{z} = y_i \]
then $\cEpinfty(\tilde{f}_n;\mathds{1}_{\Omega^\prime}\rho_n) = \cEpinftyCon(\tilde{f}_n;\mathds{1}_{\Omega^\prime}\rho_n)$.
By Proposition~\ref{prop:ConvMini:CompactnpLap}, $\{\tilde{f}_n\}_{n\in \bbN}$ is bounded in $W^{1,p}(\Omega^\prime)$ and precompact in $C^{0,\alpha}(\Omega^\prime)$.
Let $\tilde{f}_{n_k} \to f$ in $C^{0,\alpha}(\Omega^\prime)$ as $k\to+\infty$. 
We claim that $f_{n_k} \to f$ in $L^p(\Omega^\prime)$.

It is enough to show that $\|f_{n_k} - \tilde{f}_{n_k}\|_{L^p(\Omega^\prime)} \to 0$.
By Jensen's inequality we have
\begin{align*}
\| f_{n_k} - \tilde{f}_{n_k} \|^p_{L^p(\Omega^\prime)} & = \int_{\Omega^\prime} \la f_{n_k}(\bm{x}) - \int_{B(0,\eps_{n_k})} J_{\eps_{n_k}}(\bm{x}-\bm{z}) f_{n_k}(\bm{z}) \, \dd \bm{z} \ra^p \, \dd \bm{x} \\
 & \leq \int_{\Omega^\prime} \int_{B(0,\eps_{n_k})} J_{\eps_{n_k}}(\bm{x}-\bm{z}) \la f_{n_k}(\bm{x}) - f_{n_k}(\bm{z}) \ra^p \, \dd \bm{z} \, \dd \bm{x} \\
 & \leq \frac{C \eps_{n_k}^p}{\inf_{\bm{x}\in \Omega} \rho_{n_k}^2(\bm{x})} \cF^{(p)}_{\eps_{n_k}}(f_{n_k};\rho_{n_k}) \to 0
\end{align*}
It follows that $\{f_n\}_{n\in\bbN}$ is compact in $L^p(\Omega^\prime)$.
\end{proof}

As in the previous section the above compactness property can be applied to minimisers.

\begin{corollary}
\label{cor:ConvMini:CompactnonLocalpLapMin}
Assume that $\Omega,\mu,\eta,p,\rho_n$ and $\{\bm{x}_i\}_{i=1}^n$ satisfy Assumptions~\ref{Assum1}-\ref{Assum9}.
Then, minimisers of $\cFpnCon(\cdot;\rho_n)$ are precompact in $L^{p}(\Omega^\prime)$, for any $\Omega^\prime\subset\subset\Omega$.
Furthermore if $f$ is a cluster point of $\{f_n\}_{n\in\bbN}$ in $L^p(\Omega^\prime)$ then $f\in C^{0,\alpha}(\Omega^\prime)$ for any $0<\alpha<1-\frac{d}{p}$ and $f(\bm{x}_i) = y_i$ for all $i=1,\dots, N$.
\end{corollary}

\begin{proof}
Let $R>0$ satisfy $\min_{i\neq j\in \{1,\dots, N\}} |\bm{x}_i - \bm{x}_j| \geq 3R$.
Choose any $f^\dagger\in C^\infty(\Omega')$ that smoothly interpolates between constraints on balls of radius $R$ around each $\bm{x}_i$ (for $i=1,\dots, N$), i.e. $\|f^\dagger\|_{W^{1,p}(\Omega^\prime)}<+\infty$ and $f^\dagger(\bm{x})=y_i$, for $\bm{x}\in B(\bm{x}_i,R)$, $i=1,\dots, N$.
We assume $\eps_n<R$ and let $L$ be the Lipschitz constant for $f^\dagger$.

Let $f_n$ be a sequence of minimisers of $\cFpnCon$.
Then,
\begin{align*}
\cFpnCon(f_n;\rho_n) & \leq \cFpnCon(f^\dagger;\rho_n) \\
 & \leq \frac{L\|\rho_n\|_{L^\infty(\Omega)}^2}{\eps_n^p} \int_\Omega \int_\Omega \eta_{\eps_n}(|\bm{x} - \bm{z}|) |\bm{x} - \bm{z}|^p \, \dd \bm{x} \, \dd \bm{z} \\
 & = L \|\rho_n\|_{L^\infty(\Omega)}^2 \Vol(\Omega) \int_\Omega \eta(|\bm{x}|) |\bm{x}|^p \, \dd \bm{x} \\
 & = d\sigma_\eta L \|\rho_n\|_{L^\infty(\Omega)}^2 \Vol(\Omega).
\end{align*}
So, $\sup_{n\in\bbN}\cFpnCon(f_n;\rho_n) <+\infty$, hence the result follows from Proposition~\ref{prop:ConvMini:midLapCompact}.
\end{proof}

We now prove $\Gamma$-convergence.

\begin{lemma}
\label{lem:ConvMini:midLapGamma}
Assume that $\Omega,\mu,\eta,p,\rho_n$ and $\{\bm{x}_i\}_{i=1}^n$ satisfy Assumptions~\ref{Assum1}-\ref{Assum9}.
Then, $\cFpnCon(\cdot;\rho_n)$ $\Gamma$-converges to $\cEpinftyCon(\cdot;\rho)$.	
\end{lemma}

\begin{proof}
We recall~\eqref{eq:ConvMini:DensityBounds} and therefore for there exists a sequence $\gamma_n\to 0$ such that 
\begin{equation} \label{eq:ConvMini:FEnergyBounds}
(1-\gamma_n)^2 \cFpnCon(f;\rho\lfloor_{\Omega^\prime}) \leq \cFpnCon(f;\rho_n\lfloor_{\Omega^\prime}) \leq (1+\gamma_n)^2 \cFpnCon(f;\rho\lfloor_{\Omega^\prime})
\end{equation}
for all $f\in L^p(\Omega)$.

\paragraph{(Liminf inequality.)} 
Assume $f_n\to f$ in $L^p(\Omega)$ and $\liminf_{n\to\infty} \cFpnCon(f_n;\rho_n) <\infty$ else the result is trivial.
By recourse to a subsequence (relabelled) we assume that
\[ \liminf_{n\to\infty} \cFpnCon(f_n;\rho_n) = \lim_{n\to\infty} \cFpnCon(f_n;\rho_n). \]
By the compactness property (Proposition~\ref{prop:ConvMini:midLapCompact}) we have that $f(x_i) = y_i$ for all $i=1,\dots, N$.
Now,
\[ \liminf_{n\to\infty} \cFpnCon(f_n;\rho_n) \geq \liminf_{n\to\infty} \cFpnCon(f_n;\rho_n\lfloor_{\Omega^\prime}) \geq \liminf_{n\to\infty} \cFpn(f_n;\rho\lfloor_{\Omega^\prime}) \geq \cEpinfty(f;\rho\lfloor_{\Omega^\prime}) \]
by~\cite[Lemma 4.6]{SlepcevThorpe2017} and~\eqref{eq:ConvMini:FEnergyBounds}.
By Fatou's lemma $\liminf_{\Omega^\prime\to\Omega} \cEpinfty(f,\rho\lfloor_{\Omega^\prime}) \geq \cEpinfty(f,\rho)$, hence
\[ \liminf_{n\to\infty} \cFpnCon(f_n;\rho_n) \geq \cEpinfty(f,\rho). \]
Since the constraints are satisfied then $\cEpinfty(f;\rho) = \cEpinftyCon(f;\rho)$.

\paragraph{(Recovery sequence.)}
We prove the recover sequence in three parts.

\emph{Part 1:}
Assume $f\in W^{1,p}(\Omega)$ is Lipschitz continuous with $\cEpinftyCon (f;\rho)<\infty$ and $\eta$ has compact support in $B(0,M)$.
We define
\[ f_n(x) = \lb \begin{array}{ll} y_i & \text{if } |x-x_i|<\eps_n \text{ for } i=1,\dots, N \\ f(x) & \text{else.} \end{array} \rd \]
Then as $f_n$ and $f$ agree away from the constraints,
\begin{align*}
\|f_n-f\|_{L^p(\Omega)}^p & = \int_{\Omega}|f_n(\bm{x})-f(\bm{x})|^p \diff \bm{x},\\
 & = \int_{\cup_{i=1}^N B(\bm{x}_i,\eps_n)}|y_i-f(\bm{x})|^p\diff \bm{x},\\
 & \leq \Lip(f)^p \sum_{i=1}^N\int_{B(\bm{x}_i,\eps_n)} |\bm{x}_i - \bm{x}|^p \diff x \\
 & \leq \Lip(f)^p N \eps_n^p \Vol(B(0,\eps_n)) \\
 & \leq C\eps_n^{p+d}.
\end{align*}
Hence, $f_n\to f$ in $L^p(\Omega)$.

We recall the following: for all $\xi>0$ there exists a constant $C_\xi>0$ such that, for any $a,b\in \bbR^d$,
\[ |a|^p - |b|^p \leq \xi |b|^p + C_\xi |a-b|^p. \]
So, for a fixed $\xi>0$,
\begin{align*}
& \cFpn(f_n;\rho) - \cFpn(f;\rho) \\
& \qquad = \frac{1}{\eps_n^p} \int_\Omega \int_\Omega \eta_{\eps_n}(|\bm{x}-\bm{z}|) \Big( |f_n(\bm{x}) - f_n(\bm{z})|^p - |f(\bm{x}) - f(\bm{z})|^p \Big) \rho(\bm{x}) \rho(\bm{z}) \, \dd \bm{x} \, \dd \bm{z} \\
& \qquad \leq \frac{\xi}{\eps_n^p} \int_\Omega \int_\Omega \eta_{\eps_n}(|\bm{x}-\bm{z}|) |f(\bm{x}) - f(\bm{z})|^p \rho(\bm{x}) \rho(\bm{z}) \, \dd \bm{x} \, \dd \bm{z} \\
& \qquad \qquad + \frac{C_\xi}{\eps_n^p} \int_\Omega \int_\Omega \eta_{\eps_n}(|\bm{x}-\bm{z}|) |f_n(\bm{x}) - f_n(\bm{z}) - f(\bm{x}) + f(\bm{z})|^p \rho(\bm{x}) \rho(\bm{z}) \, \dd \bm{x} \, \dd \bm{z} \\
& \qquad \leq \xi \cFpn(f;\rho) + \frac{2p C_\xi \|\rho\|_{L^\infty(\Omega)}^2}{\eps_n^p} \int_\Omega \int_\Omega \eta_{\eps_n}(|\bm{x}-\bm{z}|) |f_n(\bm{x}) - f(\bm{x})|^p \, \dd \bm{x} \, \dd \bm{z} \\
& \qquad = \xi \cFpn(f;\rho) + \frac{2p C_\xi \|\rho\|_{L^\infty(\Omega)}^2}{\eps_n^p} \int_{\bbR^d} \eta(|\bm{x}|) \, \dd \bm{x} \|f_n - f\|_{L^p(\Omega)}^p \\
& \qquad \leq \xi \cFpn(f;\rho) + \tilde{C}\eps_n^d.
\end{align*}
Hence, $\cFpn(f_n;\rho) \leq (1+\xi)\cFpn(f;\rho) + O(\eps_n^{d})$.

Applying the above to~\eqref{eq:ConvMini:FEnergyBounds} we have,
\begin{align*}
\cFpnCon(f_n;\rho_n) & \leq \cFpn(f_n;\rho) + \frac{2}{\eps_n^p} \int_\Omega \int_{\Omega\setminus \Omega^\prime} \eta_{\eps_n}(|\bm{x} - \bm{z}|) |f(\bm{x}) - f(\bm{z})|^p \rho_n(\bm{x}) \rho_n(\bm{z}) \, \dd \bm{x} \, \dd \bm{z} \\
 & \leq (1+\xi) \cFpn(f;\rho) + O(\eps_n^{d}) \\
 & \quad + \frac{2\|\rho_n\|_{L^\infty}^2\Lip(f)}{\eps_n^p} \int_{\dist(z,\partial \Omega) \leq \dH(\Omega,\Omega^\prime) + M\eps_n} \int_{\Omega\setminus\Omega^\prime} \eta_{\eps_n}(|\bm{x} - \bm{z}|) |\bm{x} - \bm{z}|^p \, \dd \bm{x} \, \dd \bm{z} \\
 & \leq (1+\xi) \cFpn(f;\rho) + O(\eps_n^{d}) \\
 & \quad + 2d\sigma_\eta \Lip(f) \Vol\l\lb z\,:\, \dist(z,\partial \Omega)\leq \dH(\Omega,\Omega^\prime) + M\eps_n \rb\r \sup_{n\in\bbN} \|\rho_n\|_{L^\infty}^2.
\end{align*}
By~\cite[Lemma 4.6]{SlepcevThorpe2017} we have $\limsup_{n\to\infty}\cFpn(f;\rho)\leq \cEpinfty(f;\rho) = \cEpinftyCon(f;\rho)$.
So taking $n\to \infty$ followed by $\xi\to 0$ and $\Omega^\prime \to \Omega$ we have,
\[ \limsup_{n\to\infty} \cFpnCon(f_n;\rho_n) \leq \cEpinftyCon(f;\rho). \]

\emph{Part 2:}
We still assume that $f$ is Lipschitz continuous but relax the compact support assumption on $\eta$.
Assume $\eta$ satisfies the integrability condition in (A7).
We define $\cFpn(\cdot;\rho,\eta)$ to be the functional $\cFpn(\cdot;\rho)$ with weight function $\eta$.
Then, we let $\eta^M$ be the truncated weight function $\eta^M(t) = \eta(t) \mathds{1}_{t\leq M}$.
Now,
\[ \cFpn(f;\rho_n,\eta) = \cFpn(f;\rho_n,\eta^M) + \frac{1}{\eps_n^p} \int\int_{|\bm{x}-\bm{z}|> M\eps_n} \eta_{\eps_n}(|\bm{x}-\bm{z}|) |f(\bm{x}) - f(\bm{z})|^p \rho_n(\bm{x}) \rho_n(\bm{z}) \, \dd \bm{x} \, \dd \bm{z}. \]
We can apply part 1 to the first term on the right hand side.
For the second term, for each $\bm{z}\in\Omega$,
\[ \frac{1}{\eps_n^p} \int_{|\bm{x}-\bm{z}|> M\eps_n} \eta_{\eps_n}(|\bm{x}-\bm{z}|) |f(\bm{x}) - f(\bm{z})|^p \rho_n(\bm{x}) \, \dd \bm{x} \leq \Lip(f) \|\rho_n\|_{L^\infty} \int_{|\bm{w}|\geq M} \eta(|\bm{w}|) |\bm{w}|^p \, \dd \bm{w}. \]
Hence, (using $\eta^M\leq \eta$)
\[ \limsup_{n\to \infty} \cFpn(f;\rho_n,\eta) \leq \cEpinfty(f;\rho,\eta) + \Lip(f) \sup_{n\in\bbN}\|\rho_n\|_{L^\infty}^2 \Vol(\Omega) \int_{|\bm{w}|\geq M} \eta(|\bm{w}|) |\bm{w}|^p \, \dd \bm{w}. \]
By the monotone convergence theorem, taking $M\to\infty$ we have
\[ \limsup_{n\to \infty} \cFpn(f;\rho_n,\eta) \leq \cEpinfty(f;\rho,\eta) \]
as required.

\emph{Part 3:}
Since Lipschitz functions are dense in $W^{1,p}$ we can, as is usual in $\Gamma$-convergence arguments, conclude by a diagonalisation argument.
\end{proof}

\section{Convergence of Density Estimates} \label{sec:DensityEstimates}

In the following two subsections we prove that the kernel density estimate, and the spline kernel density estimate, satisfy Assumption (A9).

\subsection{Convergence of the Kernel Density Estimate} \label{subsec:DensityEstimates:KDE}

Our result is an easy consequence of the following theorem due to~\cite[Theorem 2.3]{GINE2002}.

\begin{theorem}
\label{thm:DensityEstimates:KDE:Gine02}
Let $K = \phi\circ \xi$ where $\phi$ is a bounded function of bounded variation and $\xi$ is a polynomial.
Assume $x_i\iid \mu$ where $\mu\in \cP(\bbR^d)$ has a bounded density $\rho$.
Assume $h_n$ satisfies
\[	h_n \to 0^+, \quad \frac{nh^d_n}{|\log(h_n)|} \to \infty, \quad \frac{|\log(h_n)|}{\log\log(n)} \to \infty,\quad \text{and} \quad h_n \leq c h_{2n} \]
for some $c>0$.
Define $\rho_{n,h}$ as in~\eqref{eq:MainResults:KDE} and $\bar{\rho}_{h}$ by
\[ \bar{\rho}_{h}(\bm{x}) = \bbE \rho_{n,h}(x) = \int_{\bbR^d} K_h(\bm{x} - \bm{z}) \, \dd \mu(\bm{z}). \]
Then there exists $C>0$ such that, with probability one,
\[ \limsup_{n\to\infty} \sqrt{\frac{n h_n^d}{|\log h_n|}} \| \rho_{n,h_n} - \bar{\rho}_{h_n} \|_{L^\infty(\bbR^d)} = C. \]
\end{theorem}

As remarked in Section~\ref{subsec:MainResults:DensResults}~\cite{GINE2002} treats a more general class of kernels $K$ and for example one could also include kernels of the form $K=\mathds{1}_{[-1,1]^d}$.
We now prove Theorem~\ref{thm:MainResults:DensResults:KDE}.

\begin{proof}[Proof of Theorem~\ref{thm:MainResults:DensResults:KDE}.]
We extend $\rho$ to the whole of $\bbR^d$ by setting $\rho(\bm{x}) = 0$ for all $\bm{x} \in \bbR^d\setminus \Omega$.
Now, for any $\Omega^\prime \subset\subset \Omega$,
\[ \| \rho_{n,h_n} - \rho \|_{L^\infty(\Omega^\prime)} \leq \| \rho_{n,h_n} - \bar{\rho}_{h_n} \|_{L^\infty(\Omega^\prime)} + \| \bar{\rho}_{h_n} - \rho \|_{L^\infty(\Omega^\prime)} \]
the first term on the RHS goes to zero by Theorem~\ref{thm:DensityEstimates:KDE:Gine02}.
For the second term we define
\[ \tilde{\Omega} = \lb \bm{x} \in \Omega \,:\, \inf_{\bm{z}\in\Omega^\prime} |\bm{x} - \bm{z} | \leq \frac12 \dH(\Omega^\prime,\Omega) \rb \]
where $\dH$ is the Hausdorff distance, and choose $\delta>0$.
Since $\rho$ is uniformly continuous on $\tilde{\Omega}$ there exists $R_\delta>0$ such that for all $\bm{x},\bm{z}\in \tilde{\Omega}$ with $|\bm{x} - \bm{z}|<R_\delta$ we have $|\rho(\bm{x}) - \rho(\bm{z}) |\leq \delta$.
Let $\spt(K) \subset B(0,M)$, and assume $n$ is large enough so that
\[ h_n \leq \frac{1}{M} \min\lb \frac12 \dH(\Omega^\prime,\Omega), R_\delta \rb. \]
Then, for any $x\in \Omega^\prime$,
\begin{align*}
| \bar{\rho}_{h_n}(\bm{x}) - \rho(\bm{x}) | & = \la \int_{\bbR^d} K_{h_n}(\bm{x} - \bm{z}) \l \rho(\bm{z}) - \rho(\bm{x}) \r \, \dd \bm{z} \ra \\
 & \leq \int_{\bbR^d} \la K_{h_n}(\bm{x} - \bm{z}) \ra |\rho(\bm{z}) - \rho(\bm{x})| \, \dd \bm{z} \\
 & \leq \delta \| K\|_{L^1(\bbR^d)}.
\end{align*}
Hence $\lim_{n\to\infty}\|\bar{\rho}_{h_n} - \rho \|_{L^\infty(\Omega^\prime)} \leq \delta \| K\|_{L^1(\bbR^d)}$.
Since $\delta>0$ is arbitrary we have shown $\lim_{n\to \infty} \|\bar{\rho}_{h_n} - \rho \|_{L^\infty(\Omega^\prime)} = 0$ as required.
\end{proof}

The following result allows us to extend the convergence to sets $\Omega_h$ where
\begin{equation} \label{eq:DensityEstimates:KDE:Omegah}
\Omega_h = \lb x\in \Omega \,:\, \dist(x,\partial\Omega)\geq h \rb.
\end{equation}

\begin{lemma}
\label{lem:DensityEstimates:UniformKDEOmegah}
If in addition to the assumptions in Theorem~\ref{thm:MainResults:DensResults:KDE} we assume that $\rho$ is Lipschitz continuous on $\Omega$, then
\[ \lim_{n\to \infty} \| \rho_{n,h_n} - \rho \|_{L^\infty(\Omega_{Mh_n})} = 0 \]
where $\Omega_h$ is defined by~\eqref{eq:DensityEstimates:KDE:Omegah}.
\end{lemma}

\begin{proof}
Analogously to the proof of Theorem~\ref{thm:MainResults:DensResults:KDE} we have
\[ \| \rho_{n,h_n} - \rho \|_{L^\infty(\Omega_{Mh_n})} \leq \| \rho_{n,h} - \bar{\rho}_{h_n} \|_{L^\infty(\bbR^d)} + \| \bar{\rho}_{h_n} - \rho\|_{L^\infty(\Omega_{h_n})} \]
where the first term goes to zero by Theorem~\ref{thm:DensityEstimates:KDE:Gine02}.
The second term goes to zero uniformly by, for all $x\in \Omega_{Mh_n}$,
\begin{align*}
|\bar{\rho}_{h_n}(\bm{x}) - \rho(\bm{x})| & \leq \int_{\bbR^d} |K_{h_n}(\bm{x} - \bm{z})| |\rho(\bm{z}) - \rho(\bm{x})| \, \dd \bm{z} \\
 & \leq \Lip(\rho) h_n \int_{\bbR^d} |K(\bm{w})| \|\bm{w}\| \, \dd \bm{w}.
\end{align*}
Hence $\| \bar{\rho}_{h_n} - \rho\|_{L^\infty(\Omega_{h_n})}\to 0$ as required.
\end{proof}

\subsection{Convergence of the Spline Kernel Density Estimate} \label{subsec:DensityEstimates:SKDE}

Our method relies on the result of \cite{arcangeli93} (given below), where almost sure $H^m$ error estimates were constructed for multivariate spline functions from data with uncorrelated, centred noise with results from~\cite{UTRERAS1988}. Here we adapt the results to suit SKDE. By the linearity of the smoothing spline functional we can write
\[ \rho_{n,h,\lambda,T} = S_{\lambda,T}(\rho + v_{n,h} - \bar{v}_{n,h}) + S_{\lambda,T}(\bar{v}_{n,h}) \]
where
\begin{align*}
v_{n,h} & = \rho_{n,h} - \rho \\
\bar{v}_{n,h} & = \bbE v_n = \int_\Omega K(\bm{x}) \l \rho(\cdot-h\bm{x}) - \rho(\cdot) \r \, \dd \bm{x}.
\end{align*}
By the triangle inequality, for $\Omega^\prime\subset\subset\Omega$,
\[ \|\rho_{n,h,\lambda,T} - \rho\|_{H^m(\Omega^\prime)} \leq \| S_{\lambda,T}(\rho+v_{n,h} - \bar{v}_{n,h}) - \rho\|_{H^m(\Omega^\prime)} + \|S_{\lambda,T}(\bar{v}_{n,h})\|_{H^m(\Omega^\prime)}. \]
The first term on the right hand side can be bounded by following theorem found in~\cite[Theorem 4.1]{arcangeli93}.

\begin{theorem} \label{Theorem:SplineExactConv}
In addition to Assumptions~\ref{AssumB1},\ref{AssumB8}-\ref{AssumB11} on $\Omega,T_n,m$ and $\{\bm{t}_i\}_{i=1}^{T_n}$, assume $\delta_n \in \bbR^{T_n}$ are random variables satisfying $\bbE\delta_{n,i} = 0$, for all $i$, $\delta_{n,i}$ is independent of $\delta_{n,j}$ for all $i\neq j$, and
\[ \forall r\in \bbN, \quad \exists C = C(r) \quad \text{such that} \quad \forall n\in\bbN, \,\, \text{and} \,\, \forall i=1,\dots, T_n \quad \text{we have} \quad \bbE|\delta_{n,i}|^{2r} \leq C. \]
Define $S_{\lambda,T}$ by~\eqref{eq:SKDE} and $P_T$ by~\eqref{eq:MainResults:PT}.
Then, for $f\in H^m(\Omega)$,
\[ \lim_{n\to \infty} \| S_{\lambda_n,T_n}(P_{T_n}(f)+\delta_n) - f \|_{H^m(\Omega)} = 0 \]
with probability one.
\end{theorem}

It is easy to check that for $\delta_n = P_{T_n}(v_{n,h_n} - \bar{v}_{n,h_n})$ that $\delta_{n,i}$ are independent whenever $\Sep(T_n)\geq 2M h_n$ where $\spt(K)\subset B(0,M)$.
Moreover,
\[ \bbE|\delta_{n,i}|^{2r} \leq 2^{2r} \bbE \|\rho_{n,h_n} - \rho\|_{L^\infty(\Omega_{h_n})}^{2r} \]
for $\{\bm{t}_i\}_{i=1}^T \subset \Omega_{Mh}$ and therefore, by Lemma~\ref{lem:DensityEstimates:UniformKDEOmegah}, the RHS converges to zero.
Hence, we can apply the above theorem to infer that $\| S_{\lambda_n,T_n}(\rho+v_{n,h_n} - \bar{v}_{n,h_n}) - \rho\|_{H^m(\Omega^\prime)} \to 0$ with probability one.

\begin{theorem}
\label{thm:DensityEstimates:SKDE:HmConv}
Under the conditions of Theorem~\ref{thm:MainResults:DensResults:SKDE} with probability one, for all $\Omega^\prime\subset\subset \Omega$, we have
\[ \|\rho_{n,h_n,\lambda_n,T_n} - \rho\|_{H^m(\Omega^\prime)} \to 0. \]
\end{theorem}

\begin{proof}
By the preceding argument and Theorem~\ref{Theorem:SplineExactConv} it is enough to show that $\|S_{\lambda_n,T_n}(\bar{v}_{n,h_n})\|_{H^m(\Omega^\prime)}\to 0$.
Let $\gamma_n = S_{\lambda_n,T_n}(\bar{v}_{n,h_n})$ and $\Gamma_n$ be the $T_n\times T_n$ matrix satisfying
\[ \bm{z}^\top \Gamma_n \bm{z} = \min\lb \|\nabla^m u\|_{L^2(\Omega)}^2 \,:\, u\in H^m(\Omega), u(\bm{t}_i) = z_i \, \forall i=1,\dots, T_n \rb. \]
One has, see~\cite[Section 5]{UTRERAS1988},
\begin{align*}
\frac{1}{T_n} \sum_{i=1}^{T_n} |\gamma_n(\bm{t}_i)|^2 & = \frac{1}{T_n} \bar{v}_{n,h_n}^\top A_n^2 \bar{v}_{n,h} \leq \frac{1}{T_n} \|\bar{v}_{n,h_n}\|_{L^\infty(\Omega_{Mh_n})}^2 \tr(A_n^2) \\
\|\nabla^m \gamma_n\|_{L^2(\Omega)} & = \frac{1}{T_n\lambda_n} \bar{v}_{n,h_n}^\top (A_n - A_n^2) \bar{v}_{n,h_n} \leq \frac{1}{T_n\lambda_n} \| \bar{v}_{n,h_n}\|_{L^\infty(\Omega_{n,h_n})}^2 \tr(A_n) 
\end{align*}
where $A_n = (\Id + \lambda_n T_n \Gamma_n)^{-1}$.
By~\cite[Theorem 5.3]{UTRERAS1988} there exists $C_1>0$, $C_2>0$ and $M>0$ such that
\[ \alpha_i^{(n)} = 0 \quad \forall i=1,\dots, M \qquad \text{and} \qquad C_1 i^{\frac{2m}{d}} \leq \alpha_i^{(n)} \leq C_2 i^{\frac{2m}{d}} \quad \forall i=M+1,\dots, T_n \]
where $\alpha_i^{(n)}$ are the ordered eigenvalues of $T_n\Gamma_n$.
Hence,
\begin{align*}
\tr(A_n^2) & = \sum_{i=1}^{T_n} \frac{1}{(\alpha_i^{(n)}\lambda_n + 1)^2} \\
 &\leq M + \sum_{i=M+1}^{T_n} \frac{1}{(C_1 i^{\frac{2m}{d}}\lambda_n +1)^2} \\
 & \leq M + \int_M^{T_n} \frac{1}{(C_1 t^{\frac{2m}{d}}+1)^2} \, \dd t \\
 & \leq M + \frac{1}{(C_1\lambda_n)^{\frac{d}{2m}}} \int_0^\infty \frac{1}{(s^{\frac{2m}{d}}+1)^2} \, \dd s
\end{align*}
and similarly,
\[ \tr(A_n) \leq M + \frac{1}{(C_1\lambda_n)^{\frac{d}{2m}}} \int_0^\infty \frac{1}{s^{\frac{2m}{d}}+1} \, \dd s. \]
So,
\begin{align*}
\frac{1}{T_n} \sum_{i=1}^{T_n} |\gamma_n(\bm{t}_i)|^2 & \leq \frac{\tilde{C}}{T_n\lambda_n^{\frac{d}{2m}}} \|\bar{v}_{n,h_n}\|_{L^\infty(\Omega_{Mh_n})}^2 \\
\|\nabla^m \gamma_n\|_{L^2(\Omega)} & \leq \frac{\tilde{C}}{T_n\lambda_n^{1+\frac{d}{2m}}} \| \bar{v}_{n,h_n}\|_{L^\infty(\Omega_{n,h_n})}^2. 
\end{align*}
By~\cite[Theorem 3.4]{UTRERAS1988} we can bound
\[ \|\gamma_n\|_{H^m(\Omega^\prime)}^2 \leq C^\prime \l \frac{1}{T_n} \sum_{i=1}^{T_n} |\gamma_n(\bm{t}_i)|^2 + \|\nabla^m\gamma_n\|_{L^2(\Omega^\prime)}^2 \r \leq \frac{\hat{C} \|\bar{v}_{n,h_n}\|_{L^\infty(\Omega_{Mh_n})}^2}{T_n\lambda_n^{\frac{2m+d}{2m}}}. \]
Hence, $\gamma_n\to 0$ in $H^m(\Omega^\prime)$ as required.
\end{proof}

The proof of Theorem~\ref{thm:MainResults:DensResults:SKDE} is now, due to Sobolev embeddings (in particular Morrey's inequality), just a corollary of the above theorem since $m>d/2$.

\section{Numerical Experiments} \label{sec:Numerics}

To provide evidence of the convergence results stated in this paper, we consider numerical examples of the KDE and SKDE, construct methods to determine the minimisers of different $p$-Dirichlet energies, and provide example computations and error estimates.
We compare our results in terms of computation time with~\cite{floresrios19AAA}.
By approximating the densities and discretising on a coarser grid we introduce another source of error which could be significant for small data sizes, hence our method is not state-of-the-art in the small data regime.
On the other hand, discrete based methods fairly quickly become computationally infeasible whereas the continuum limit based numerical method controls the computational cost allowing one to apply the method to very large datasets; this is the regime where our approach is state-of-the-art.

\subsection{Setup} \label{subsec:Numerics:Setup}
We consider the domain $\Omega = [0,1]\times[0,1]$, and sample from three different densities:
\begin{align*}
\rho_1(x,y) &= 1,\\
\rho_2(x,y) &= \frac{1}{\mathcal{N}_2}(xy+0.2),\\
\rho_3(x,y) &= \frac{1}{\mathcal{N}_3}(\cos\big(6\pi((x-0.5)^2+(y-0.2)^2)\big)/3 + 0.5),
\end{align*}
where $\mathcal{N}_2,\mathcal{N}_3$ are normalisation constants.
The densities are plotted in Figure~\ref{fig:Numerics:Densities}.

\begin{figure}[ht!]
\centering
\subfloat[][]{\includegraphics[width=0.30\textwidth]{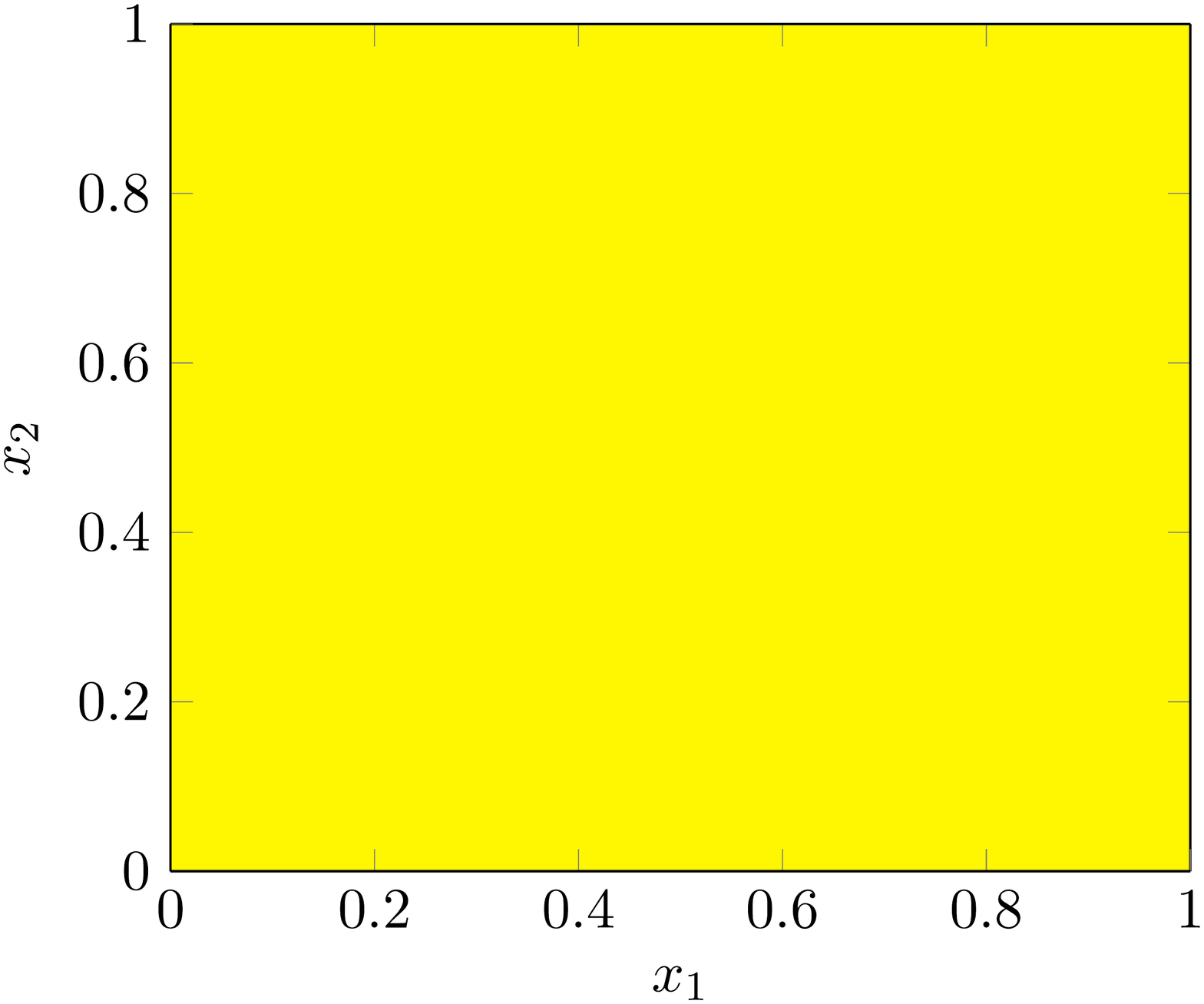}\label{Figure:UniformDensity}}
 \subfloat[][]{\includegraphics[width=0.30\textwidth]{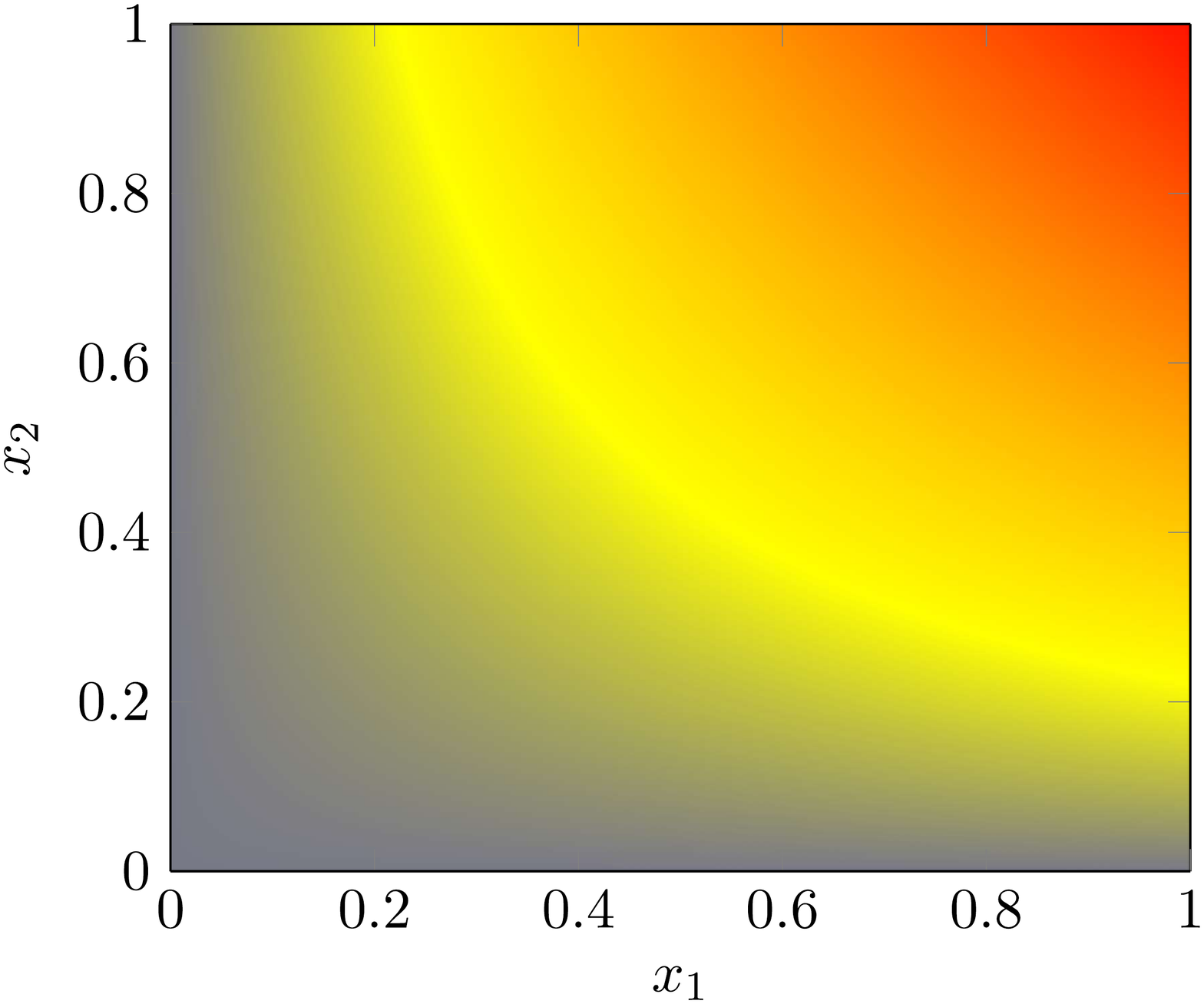}\label{Figure:xyDensity}}	
\subfloat[][]{\includegraphics[width=0.30\textwidth]{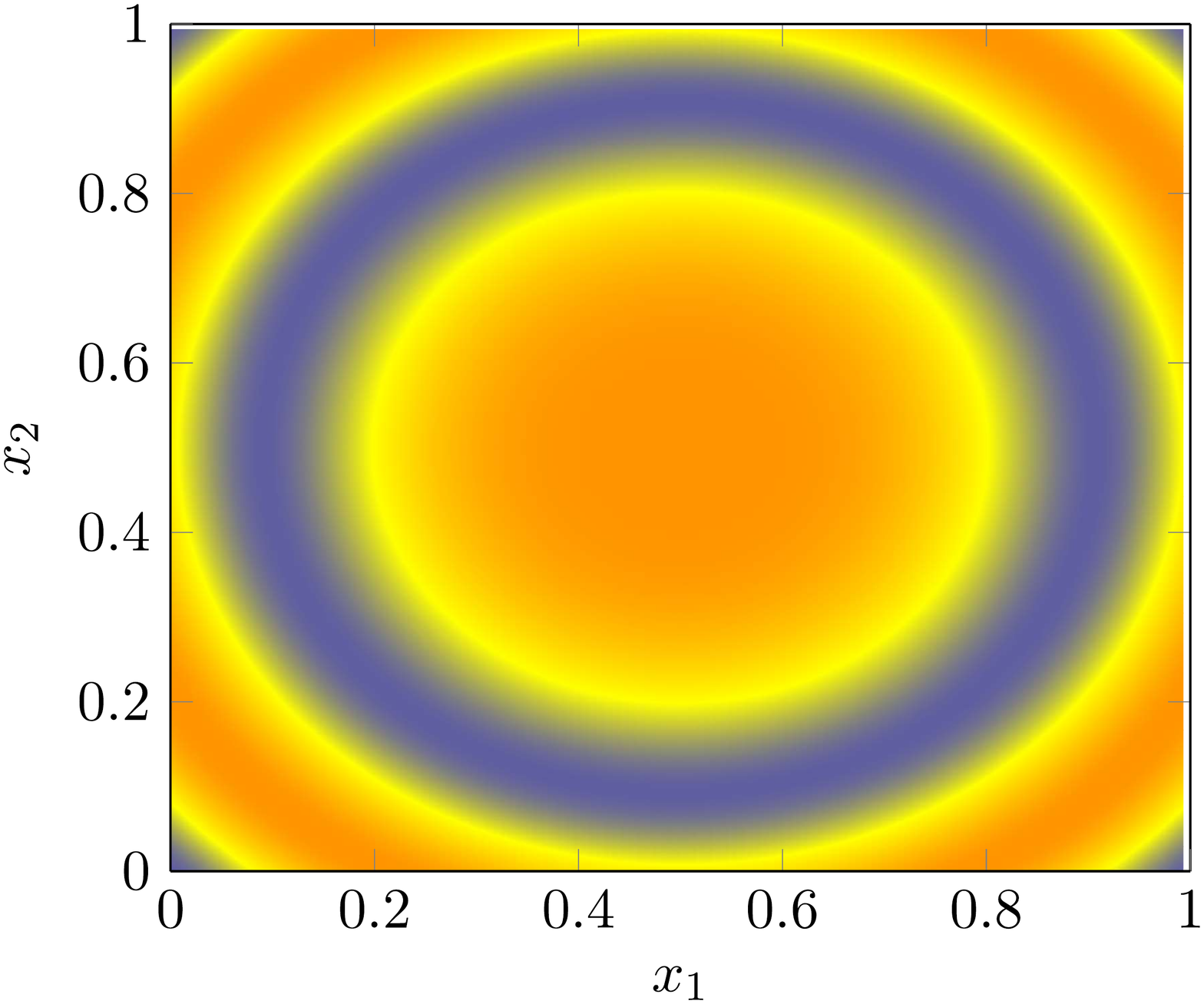}\label{Figure:TrigDensity}}
\hspace{0.02\textwidth}
\subfloat{\includegraphics[width=0.07\textwidth]{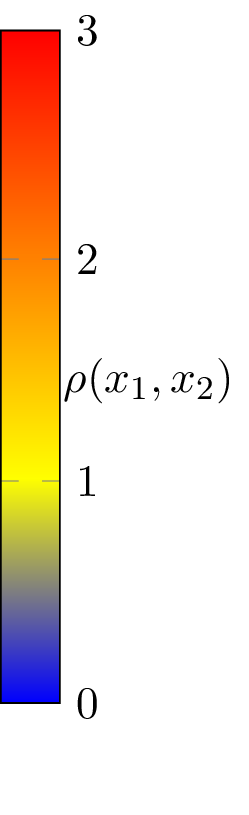}\label{Figure:DensityColorbar}}
\addtocounter{subfigure}{-1}
\\
\caption{Three densities considered in the examples. Left: $\rho_1$, centre: $\rho_2$, right: $\rho_3$.
}\label{fig:Numerics:Densities}
\end{figure}

For simplicity, we position 16 constraints uniformly across the domain, and labels are given using the formula:
\[	C(x,y) = 4\left(x-\frac{1}{2}\right)^2 + \left(y-\frac{1}{2}\right)^2. \]
The constraints are presented graphically in Figure~\ref{Figure:Constraints}.

\begin{figure}[ht!]
	\centering
	\includegraphics[width=0.5\textwidth]{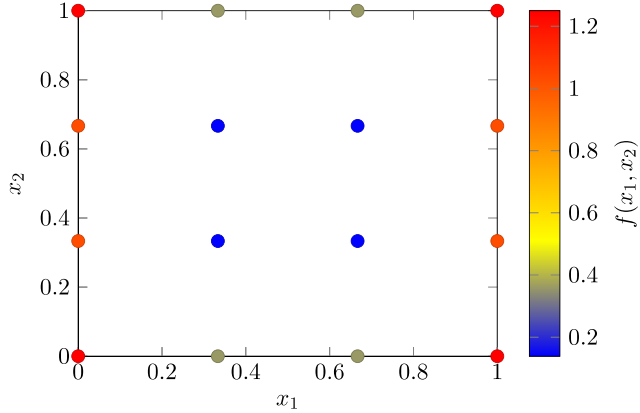}
	\caption{The position and value of the constraints used in each example.}\label{Figure:Constraints}	
\end{figure}

Code is based on the pseudo-spectral code base \texttt{2DChebClass}, \cite{DDFTCode}. The numerical methods used in this paper also rely on boundary patching methods. A version of \texttt{2DChebClass} which includes boundary patching and $p$-Dirichlet minimization is available upon request.

\subsection{Density Estimation}\label{subsec:Numerics:DensityEstimates}
\subsubsection{Numerical Method}
For the numerical results of the density estimate, we discretise the domain $\Omega$ uniformly with $D=2^{10}$ evenly spaced grid points in each dimension and consider
\[	\tilde{\Omega}_D = \{(x_i,y_j), x_i = \frac{i}{D-1}, y_j = \frac{j}{D-1}, i,j=0,\dots,(D-1)\}. \]
We choose $D$ large, so that discretization errors are small.
We sample from non-uniform densities using the \texttt{MATLAB} function \texttt{pinky} \cite{PinkyFile}. For the density estimates we consider two measures of error, given a density estimate $\rho_n$, we consider the $L^2$ error:
\[ \|\rho_n - \rho\|_{L^2(\Omega^\prime)}^2 = \int_{\Omega^\prime}|\rho_n(\bm{x}) - \rho(\bm{x})|^2\diff\bm{x}
\approx \frac{1}{D^2}\sum_{i,j=0}^{D-1} \mathds{1}_{(x_i,y_j)\in \Omega^\prime} |\rho_n(x_i,y_j) - \rho(x_i,y_j)|^2, \]
and the $L^\infty$ error:
\[	\|\rho_n - \rho\|_{L^\infty(\Omega^\prime)} = \sup_{\bm{x}\in\Omega^\prime}|\rho_n(\bm{x}) - \rho(\bm{x})|
	\approx \sup_{(x,y)\in\tilde{\Omega}_D\cap \Omega^\prime}|\rho_n(x,y) - \rho(x,y)|. \]
As we are only interested in the local approximation to the density, we construct the $L^{\infty}$ and $L^2$ errors on $\Omega^\prime = [0.01,0.99]\times[0.01,0.99]$. Kernel density estimates and smoothing splines are well studied, so we can utilise built-in functions in \texttt{MATLAB} in our calculations. We construct the KDE using the built-in function \texttt{mvksdensity} using Gaussian kernels (which are kernel functions of order 2). To construct the SKDE we use the built-in function \texttt{spaps}. For simplicity, we take the knots $\{t_i\}_{i=1,\dots,T}$ to be evenly spaced across the domain. An example of samples from $\rho_2$ and the associated KDE and SKDE is given in Figure~\ref{Figures:DensityEstimateExamples}.
\begin{figure}[ht!]
	\centering
	\subfloat[][]{\includegraphics[width=0.3\textwidth]{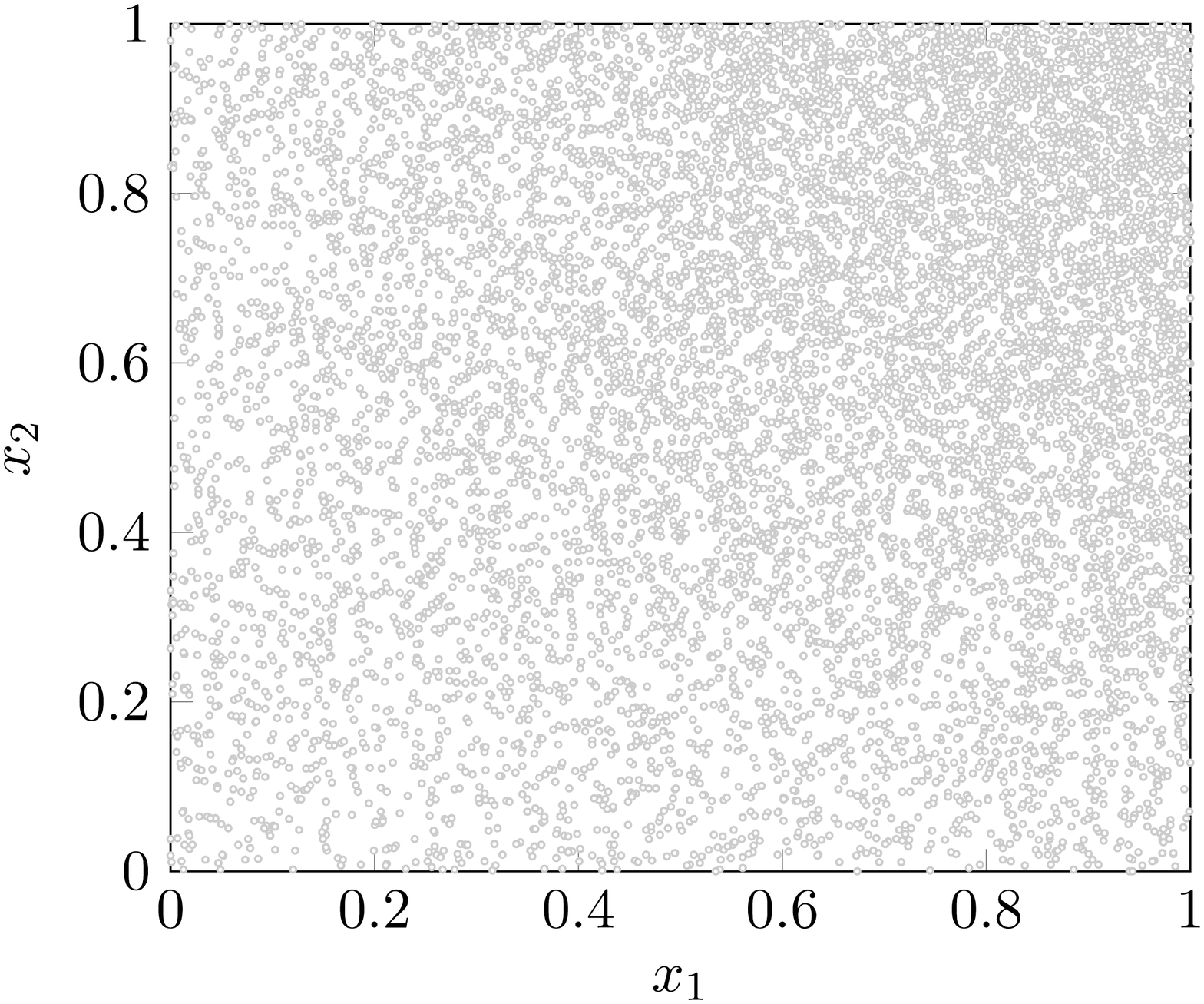}\label{Figure:Samples}}
	\subfloat[][]{\includegraphics[width=0.3\textwidth]{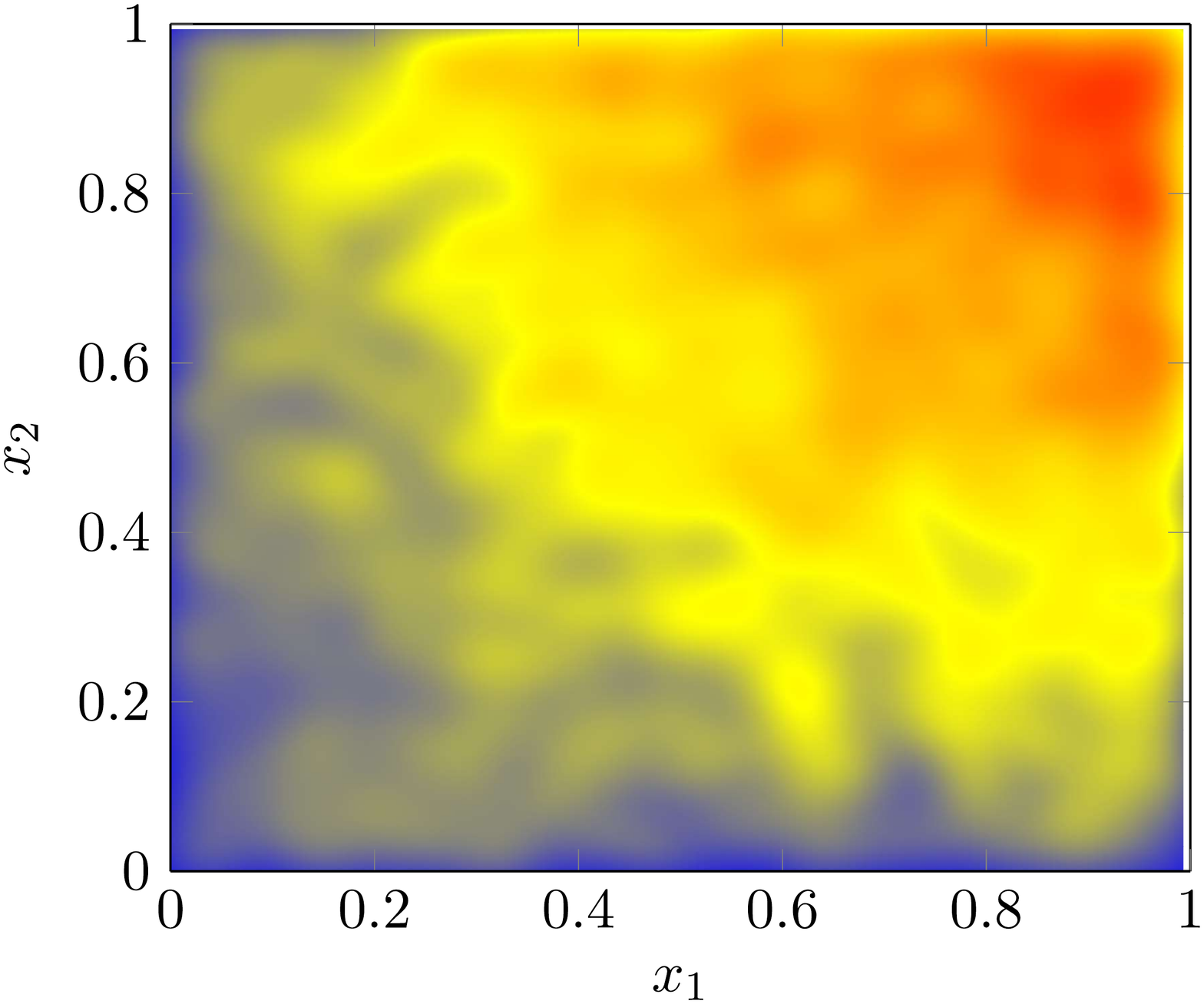}\label{Figure:KDEExample}}
	\subfloat[][]{\includegraphics[width=0.3\textwidth]{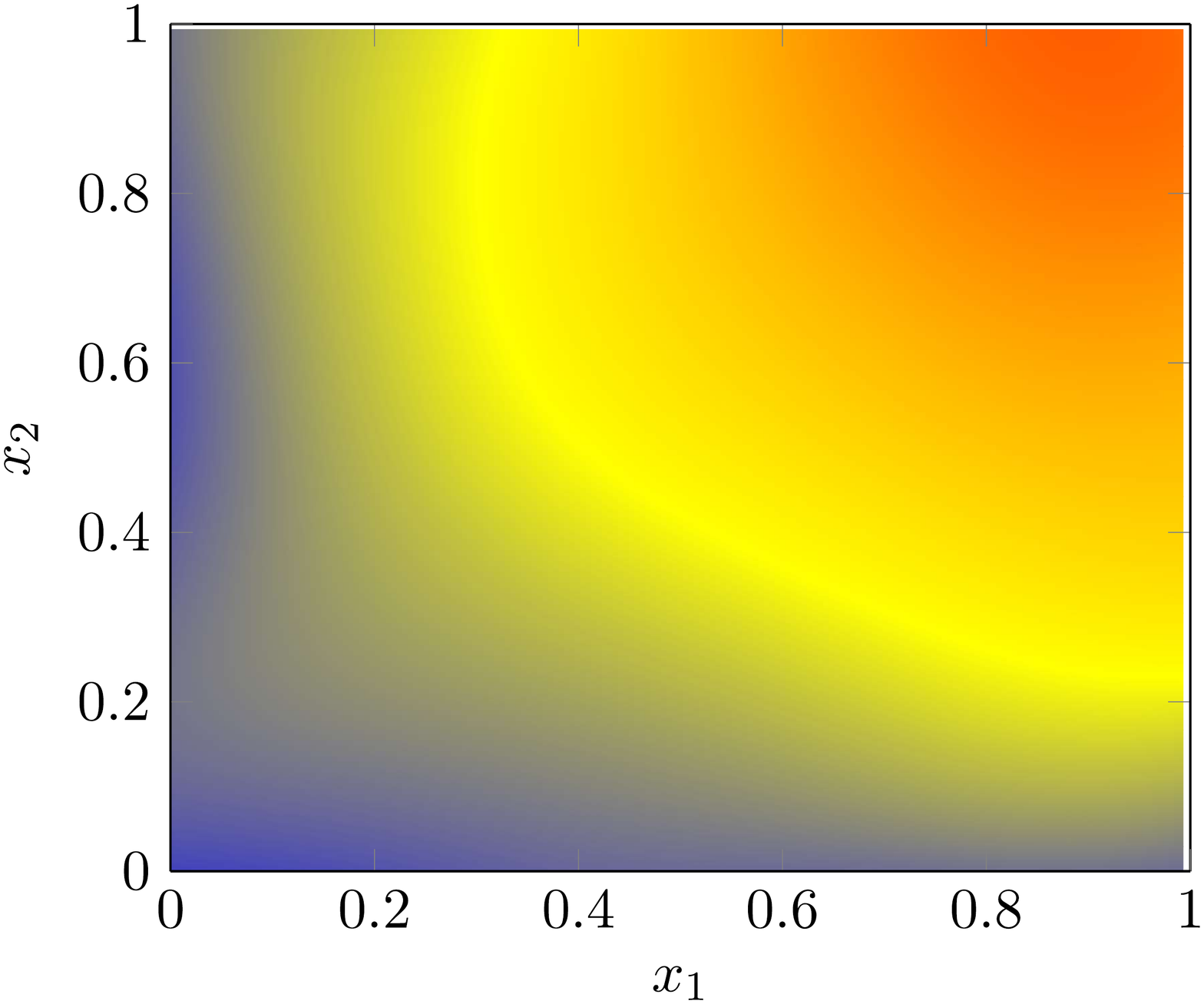}\label{Figure:SKDEExample}}
	\hspace{0.02\textwidth}
	\subfloat{\includegraphics[width=0.07\textwidth]{Figs/GenerateFigures-figure3.png}\label{Figure:DensityColors}}
	\addtocounter{subfigure}{-1}
	\caption{Left:10000 samples from $\rho_2$ using \texttt{pinky}, centre: the KDE with $h=0.03$ using \texttt{mvksdensity}, right: the SKDE with $h=0.03,\lambda=10^{-6},T=2^{12}$ using \texttt{spaps}. Density estimates are calculated on a mesh of $2^{10}\times 2^{10}$ points.}\label{Figures:DensityEstimateExamples}			
\end{figure}

\subsubsection{Results and Discussion}

We present the $L^{\infty}$ errors for the density and the derivative in Figure~\ref{Figures:LInfErrors} and Figure~\ref{Figures:dLInfErrors}.
We take $T=2^{12}$ and $\lambda=10^{-6}$.
For simplicity $T$ and $\lambda$ remain fixed in the experiments.

Figures~\ref{Figures:LInfErrors} and~\ref{Figures:dLInfErrors} show that the SKDE does no worse that the KDE, and often performs better, in terms of $L^\infty$ error. Occasionally, we see that the $L^{\infty}$ error for the SKDE is greater than the KDE.
This may be because of larger fluctuations in the derivatives, where the SKDE over-smooths the KDE. For all three densities, the KDE density error has an optimal choice of bandwidth $h$, as predicted by the theory.

Although in some cases the improved approximation due to smoothing splines is small, they also provide additional robustness in the choice of bandwidth $h$, with very little additional computation cost. We present the computation time for the KDE and SKDE in Figure~\ref{Figures:DensityEstimateTime}. The computation time for different densities is almost identical, and the inclusion of a smoothing spline approximation is negligible in cost. We note that the computational cost of kernel density estimation can also be significantly reduced using parallelisation.

\begin{figure}[ht!]
	\centering
	\subfloat[][]{\includegraphics[width=0.3\textwidth]{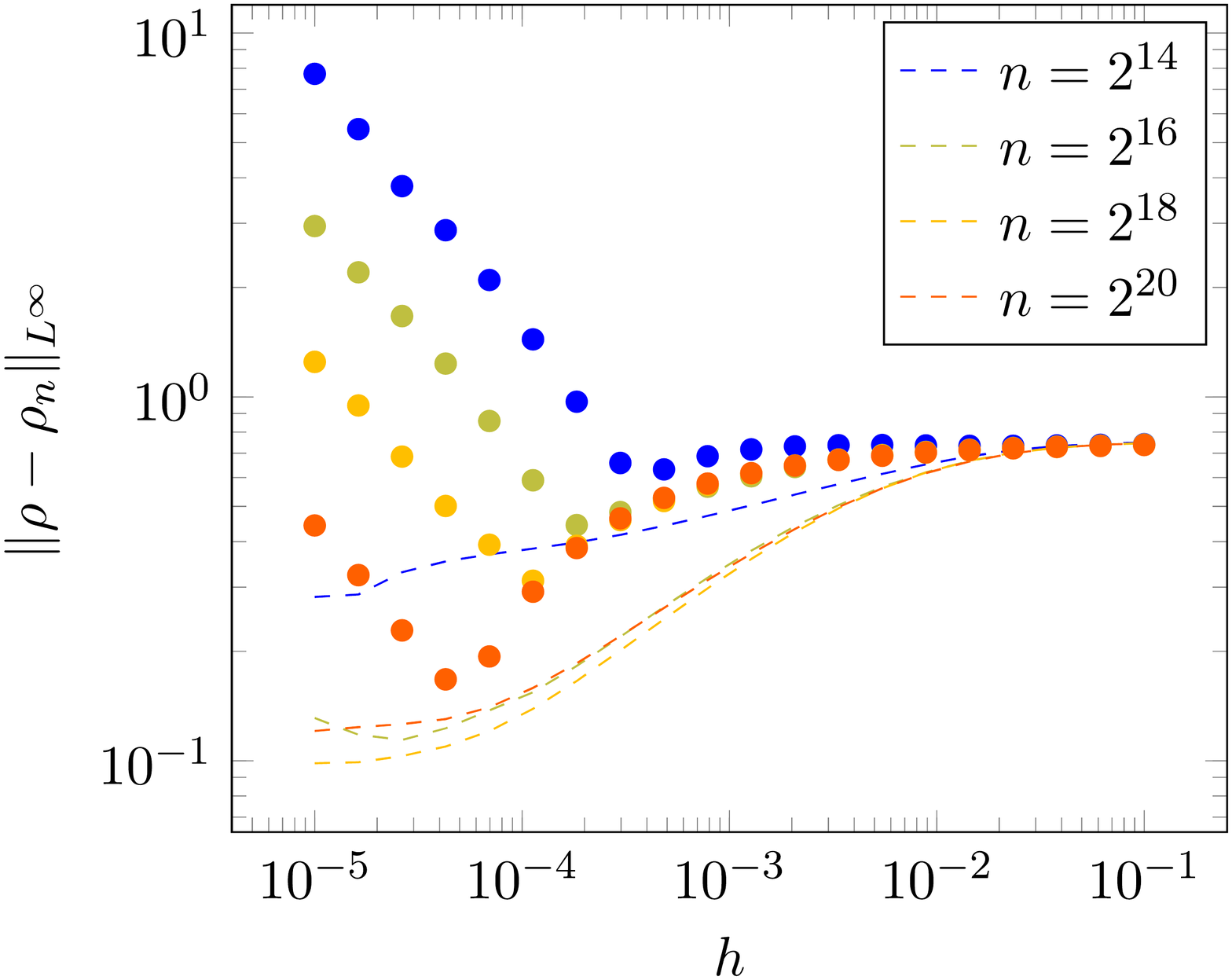}\label{Figure:UniformLInfError}}
	\subfloat[][]{\includegraphics[width=0.3\textwidth]{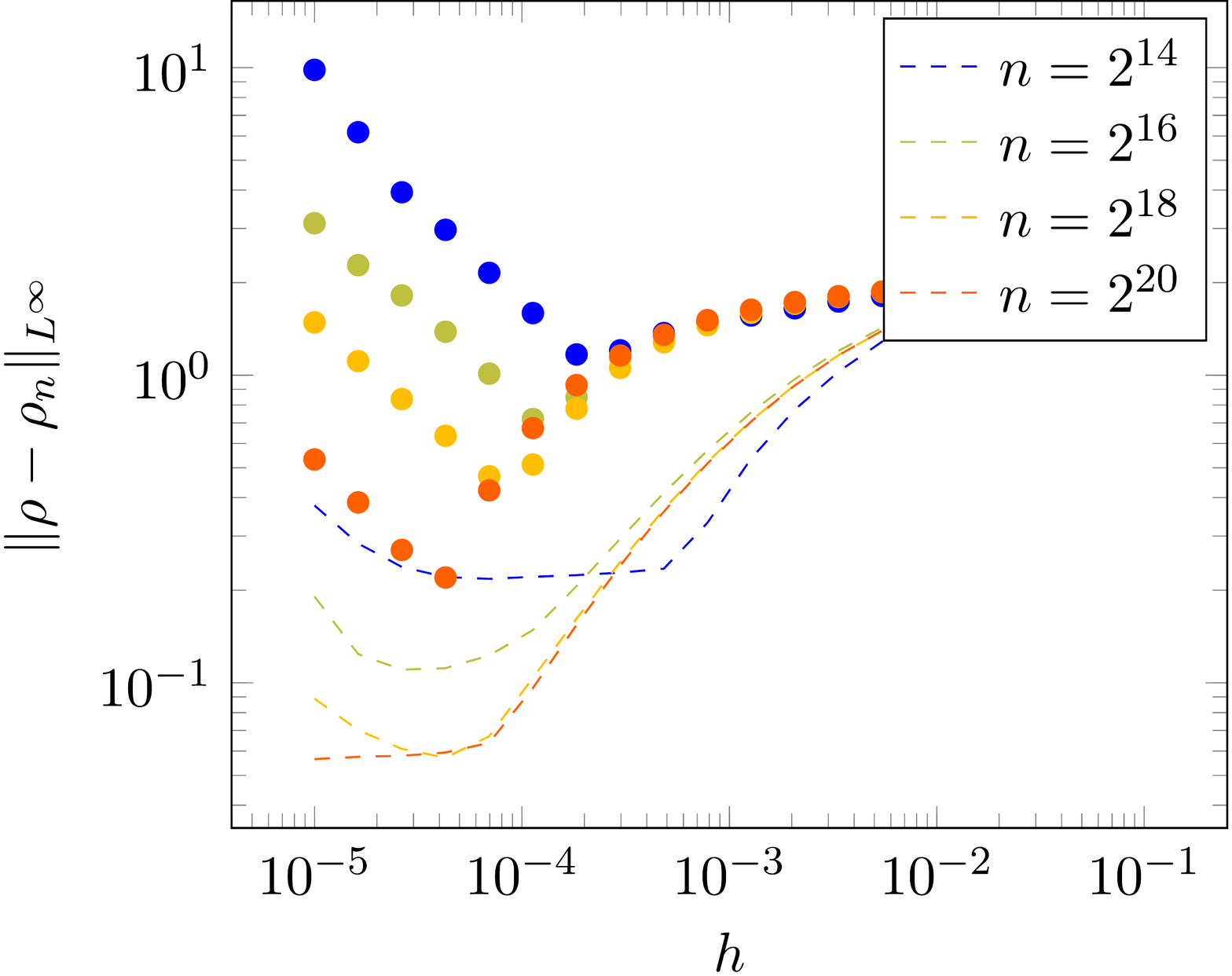}\label{Figure:xyLInfError}}
	\subfloat[][]{\includegraphics[width=0.3\textwidth]{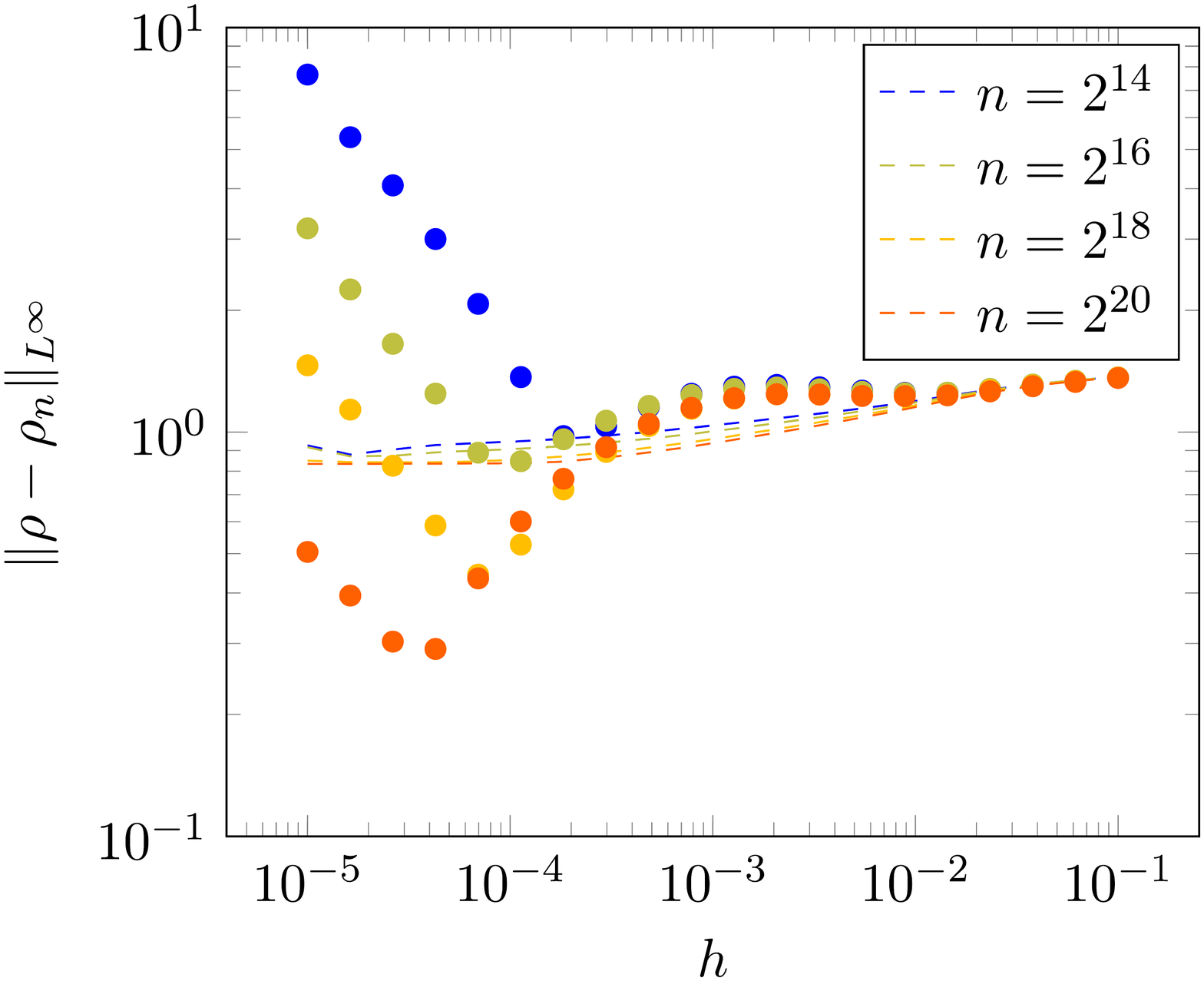}\label{Figure:TrigLInfError}}
	\caption{$L^{\infty}$ errors for the three densities using the two density estimates. Dotted lines represent the KDE, while dashed lines represent the SKDE. Left: $\rho_1$, centre: $\rho_2$, right: $\rho_3$. }\label{Figures:LInfErrors}	
\end{figure}	
	
\begin{figure}
	\centering
	\subfloat[][]{\includegraphics[width=0.3\textwidth]{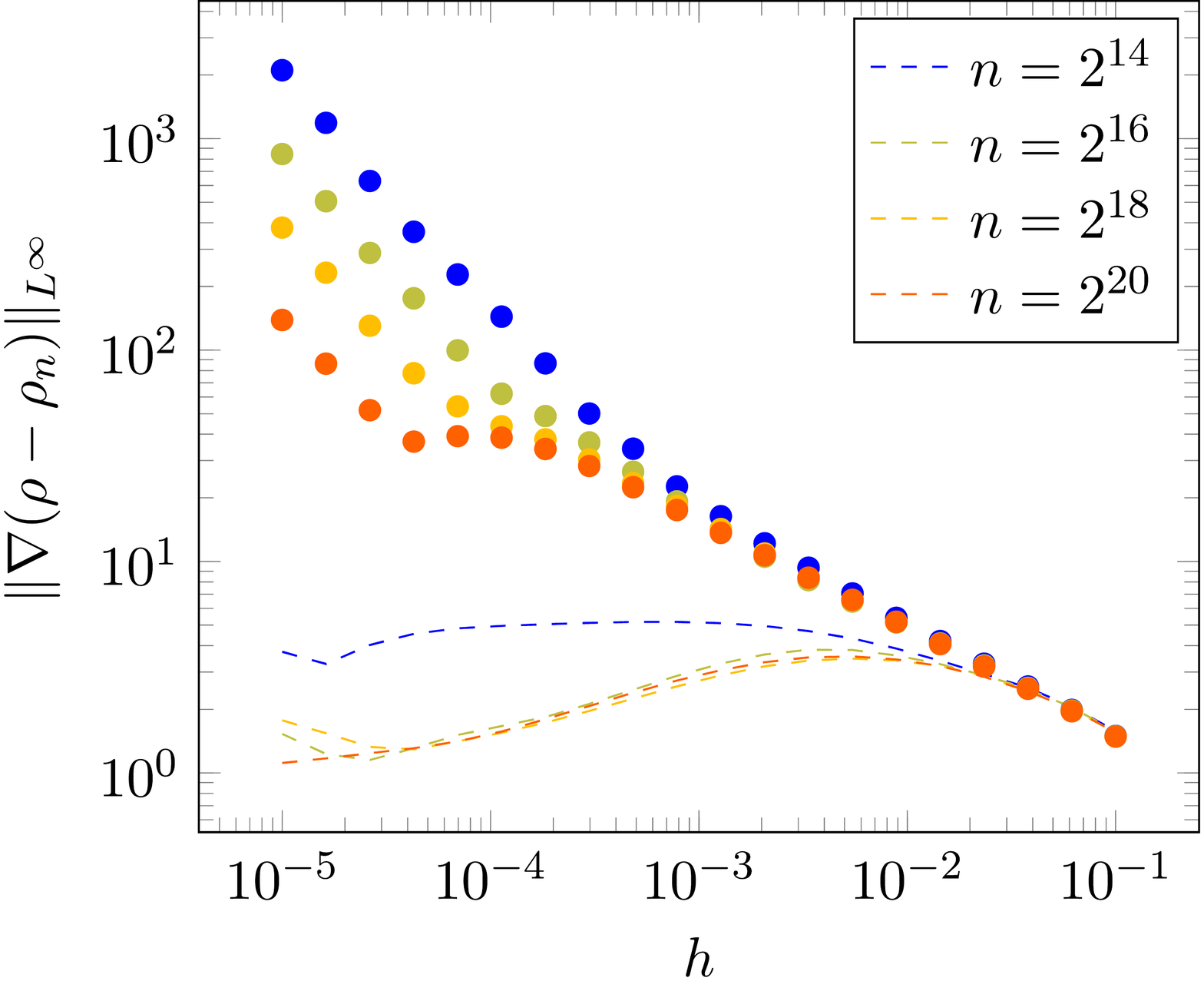}\label{Figure:UniformdLInfError}}
	\subfloat[][]{\includegraphics[width=0.3\textwidth]{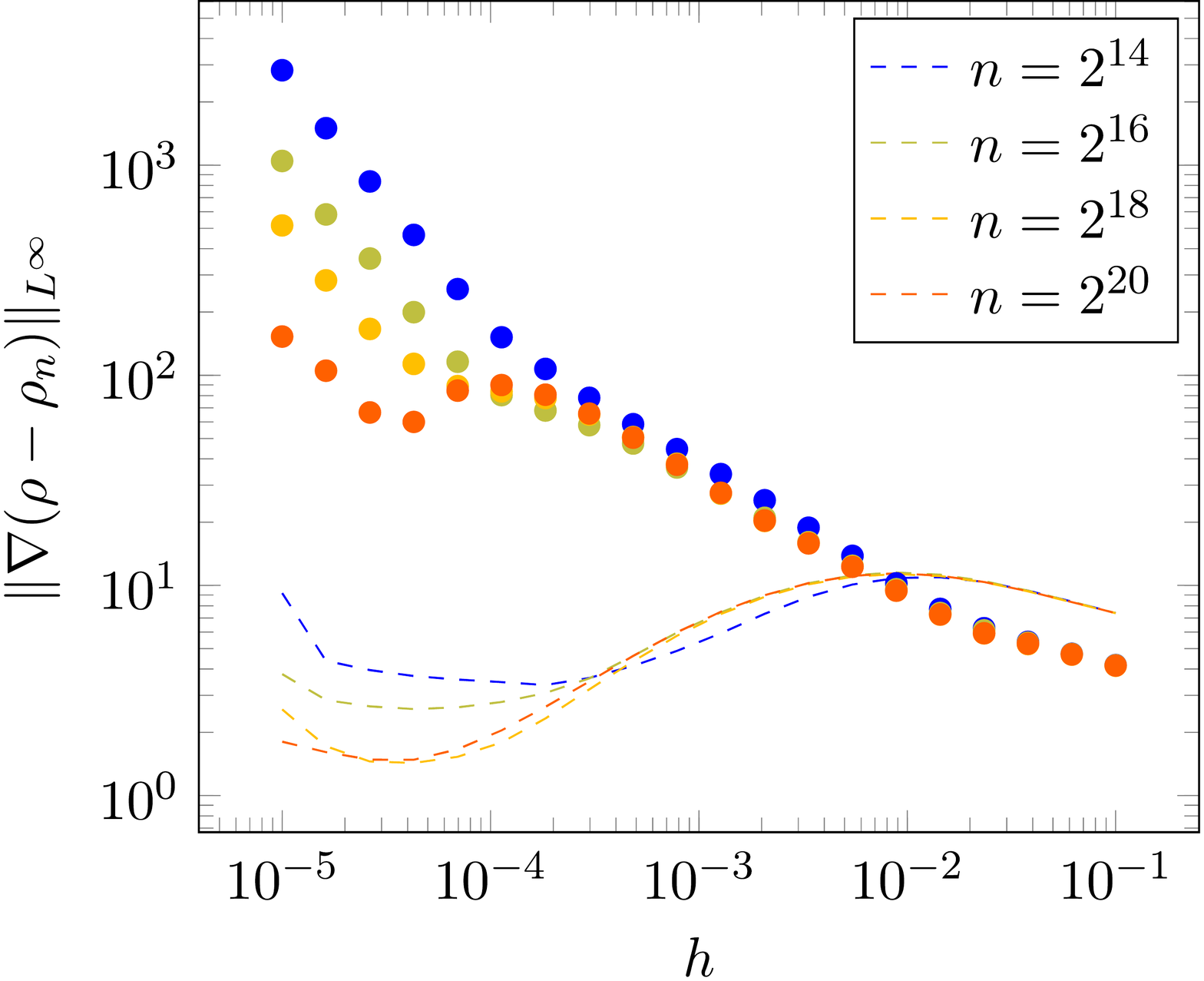}\label{Figure:xydLInfError}}
	\subfloat[][]{\includegraphics[width=0.3\textwidth]{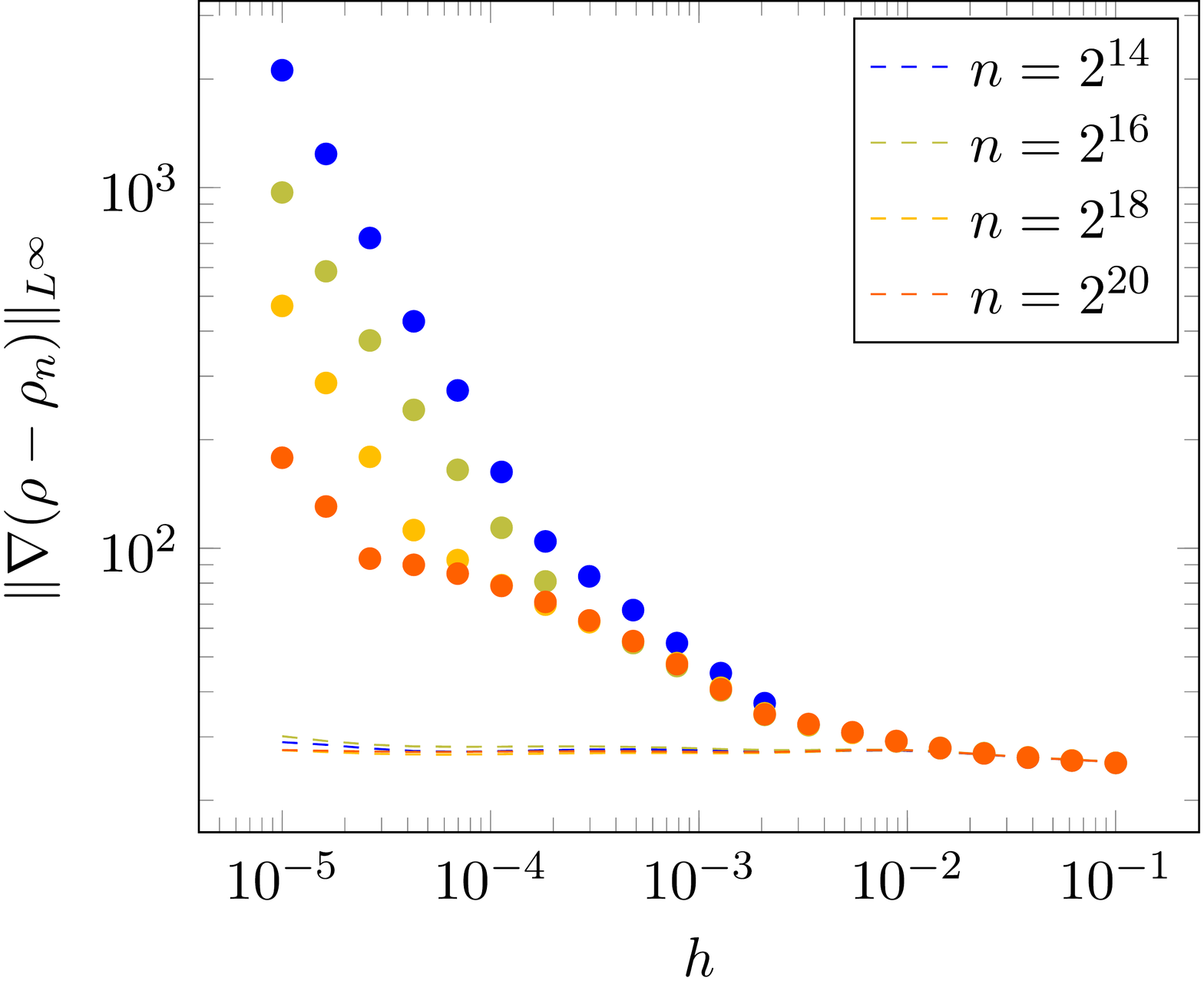}\label{Figure:TrigdLInfError}}
	\caption{$L^{\infty}$ errors for the derivatives of the three densities using the two density estimates. Dotted lines represent the KDE, while dashed lines represent the SKDE. Left: $\rho_1$, centre: $\rho_2$, right: $\rho_3$.
	}
	\label{Figures:dLInfErrors}	
\end{figure}

\begin{figure}[ht!]
	\centering
	\includegraphics[width=0.45\textwidth]{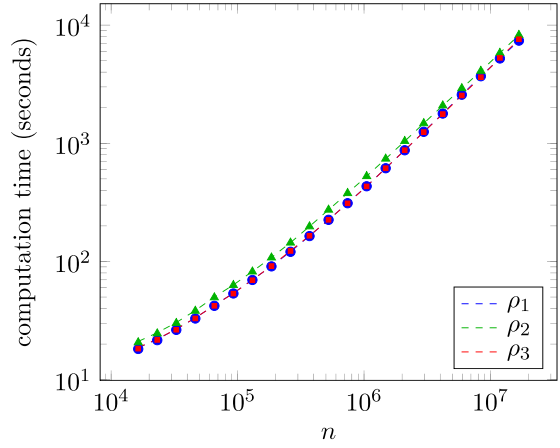}\label{Figure:KDETime}
	\caption{Computation time for KDE and SKDE density estimates. 
	Dashed lines represent KDE computation times, the larger circles, squares and triangles of the same color are SKDE computation times for each density. 
	}\label{Figures:DensityEstimateTime}			
\end{figure}

\subsection{\texorpdfstring{$p$}{p}-Dirichlet energy minimisation}

We wish to compare the accuracy and efficiency of different Dirichlet energies on different densities.
In contrast to the discrete $p$-Dirichlet energies, the continuum $p$-Dirichlet energies are not prohibitively expensive when $n$ is large.
However, when $d$ is large, standard numerical methods become computationally intractable as the number of discretization points increases exponentially in dimension.
It is an area of future work to construct a numerical scheme which can find the minimiser of \eqref{eq:Intro:conpDirichlet} when $d$ is large, perhaps under additional assumptions on the underlying probability density of the data.
For now, we restrict our numerical investigation to problems with $d=2$, and focus on the large data problem, rather than the large dimension problem.
\subsubsection{Numerical Methods}
When $p\neq2$, the minimization problem \eqref{eq:minimizationProblem} we must consider is nonlinear. Therefore, to construct minimisers for the different $p$-Dirichlet energies we use gradient descent.

\paragraph{Discrete $p$-Dirichlet energies}
For~\eqref{eq:Intro:conGraphpDirichlet} we consider the gradient flow
\begin{align*}
\frac{\partial f(\bm{x}_i)}{\partial t} & = \begin{cases} 0 \quad & i=1,\dots,N, \\ -\frac{p}{\eps^p_nn^2}\sum_{j=1}^n W_{ij}(f(\bm{x}_j)-f(\bm{x}_i))|f(\bm{x}_j)-f(\bm{x}_i)|^{p-2}\quad & \text{else}. \end{cases} \\
 & =: \nabla\cEpnCon(f)(\bm{x}_i).
\end{align*}
By construction, the solution to $\nabla\cEpnCon(f) = 0$ is the minimiser of $\cEpnCon$.
To find the minimiser we start with an initial guess $f_0$ (where $f_0$ agrees with the constraints), discretise time and advance via a timestep $\tau$.
Thus, at the $k^{th}$ step:
\[ f_{k+1}(\bm{x}_i) = f_k(\bm{x}_i) - \tau\nabla\cEpnCon(f_k)(\bm{x}_i) \qquad i=N+1,N+2,\dots, n. \]

Gradient descent is an important and well studied method in optimisation. 
Nesterov accelarated gradient descent \cite{Su2016} improves convergence to $\mathcal{O}(\frac{1}{k^2})$, compared to the result for standard gradient descent, which has a convergence rate of $\mathcal{O}(\frac{1}{k})$. Adaptive gradient descent methods such as ADAM \cite{Kingma2014} can further speed up convergence. Proof of convergence of ADAM for convex functions was originally provided in \cite{Kingma2014}, and improvements in the proof were later provided by \cite{Bock2018}, although there is still some contention in the literature of this result \cite{Reddi2018}.

However, independent of the gradient descent algorithm applied, the computation time of the minimization problem scales as $\mathcal{O}(n^3)$. Difficulties also arise when trying to choose the correct value for $\eps_n$.
The asymptotic bounds~\eqref{eq:MainResults:epsilonBounds} provide some reference for a good value to take, but it is uncertain what value will provide a good result for a particular number of samples $n$. 

For state-of-the art calculation of the discrete $p$-Dirichlet minimizer, we use the Newton iteration method and homotopy discussed in \cite{floresrios19AAA}, which reduces the computational cost to $\mathcal{O}(n^2)$. In this case the graph is connected using $k$ nearest neighbours calculations, to avoid complications in choosing $\varepsilon_n$. We present the computation times for this method on our examples in \cref{Figure:ComputationTimesComparison}.

\paragraph{Continuum $p$-Dirichlet energies}

For the continuum $p$-Dirichlets, we require gradient descent on a continuum rather than on discrete data points. The associated gradient flow is found by calculating the Gateaux derivative of $\cEpinfty$. For any $v\in W^{1,p}(\Omega)$,
\begin{align*}
	\partial \cEpinfty(u;v)  :=& \lim_{\delta\to 0^+} \frac{1}{\delta} \left( \cEpinfty(u+\delta v) - \cEpinfty(u) \right), \nonumber\\
	=& p \sigma_\eta \int_{\partial \Omega} v |\nabla u|^{p-2} \rho^2 \nabla u \cdot \dd S - p \sigma_\eta \int_\Omega v \, \mathrm{div}(\nabla u |\nabla u|^{p-2} \rho^2 ) \, \dd x.
\end{align*}
The minimiser will therefore satisfy
\[	 |\nabla u|^{p-2} \nabla u\cdot \mathrm{n} \rho^2 \mathds{1}_{\partial\Omega} - \mathrm{div}(\nabla u |\nabla u|^{p-2} \rho^2) = 0, \]
where $\mathds{1}$ is an indicator function and $\mathrm{n}$ is the outward unit normal to the surface $\partial\Omega$. To find the minimiser we can perform gradient descent using
\begin{align}
	\frac{\partial u}{\partial t} =\begin{cases} \mathrm{div}(\nabla u |\nabla u|^{p-2} \rho^2) \quad\text{ on }\Omega\setminus\partial\Omega,\\
	-|\nabla u|^{p-2} \nabla u \cdot \mathrm{n} \rho^2 + \beta \mathrm{div}(\nabla u |\nabla u|^{p-2} \rho^2)\quad\text{ on }\partial\Omega,\\
	0,\quad\text{ at }\bm{x}_i,i=1,\dots,N
\end{cases}\label{eq:GradDescentpLap}
\end{align}
where $\beta$ is a parameter which allows flux through the boundary $\partial\Omega$. In our simulations we take $\beta=0.01$.

For each gradient step, we need to approximate spacial derivatives of $\rho$ and $f$.
We use pseudo-spectral methods~\cite{Trefethen2000} and domain decomposition~\cite{Boyd2001} to accurately and efficiently apply gradient descent.
Pseudo-spectral (or collocation) methods are popular methods in the construction of numerical solutions to PDEs.
Provided the function in consideration is suitably smooth, these methods produce high precision on coarse meshes.

To gain some intuition of the pseudo-spectral methodology, we consider the one dimensional case on a periodic domain. For a function $f:\bbR\to\bbR$ discretised on a uniform grid $\{x_1,\dots,x_D\}$, where $|x_i-x_{i+1}|=h$, the finite difference approximation $g:\bbR\to\bbR$ of the derivative $f^\prime$ at points $x_i$ is determined by
\[ f^\prime(x_i) \approx g(x_i) = \frac{f(x_{i+1})-f(x_{i-1})}{2h} \]
We can write this as a matrix multiplication, involving a sparse matrix,
\[ \begin{pmatrix} g(x_1) \\ g(x_2) \\  \\	\vdots \\  \\ g(x_D) \end{pmatrix} = h^{-1}
\begin{pmatrix}	0 & \frac{1}{2} &  &  &  & -\frac{1}{2}\\ -\frac{1}{2} & 0 &  & \ddots &  & \\ & 0 &  & \ddots &              & \\ & 0 &  & \ddots & 0 & \frac{1}{2} \\ \frac{1}{2} &  &  &  & -\frac{1}{2} & 0 \end{pmatrix}	\begin{pmatrix} f(x_1) \\ f(x_2) \\ 	\\ \vdots \\  \\ f(x_D)	\end{pmatrix}. \]
This approximation is then accurate to $O(h^2)$, where $h$ is the distance between two grid points.
We may also consider a finite difference approximation which includes the four nearest points,
\[ 
f^\prime(x_i) \approx g(x_i) = 
\frac{-f(x_{i+2}) +8 f(x_{i+1}) - 8f(x_{i-1})+f(x_{i-2})}{12h}
\]
which can be written as another matrix multiplication with matrix,
\[ \begin{pmatrix}
	g(x_1) \\
	g(x_2) \\
	\\
	\vdots \\
	\\
	g(x_D)
	\end{pmatrix}
	=h^{-1}
	\begin{pmatrix}
	 &              & \ddots        &             &        & \frac{1}{12}   & -\frac{2}{3}      &\\
	 &              & \ddots        &-\frac{1}{12}&        &                &  \frac{1}{12}     &\\
	 &              & \ddots        & \frac{2}{3} & \ddots &                &       &\\
	 &              & \ddots        & 0           & \ddots &                &       &\\
	 &              & \ddots        &-\frac{2}{3} & \ddots &  			    &       &\\
	 &-\frac{1}{12} &               & \frac{1}{12}& \ddots &    			&       &\\
	 &\frac{2}{3}   &-\frac{1}{12}  &             & \ddots &    			&       &\\
	\end{pmatrix}
	\begin{pmatrix}
	f(x_1) \\
	f(x_2) \\
	\\
	\vdots \\
	\\
	f(x_D)
	\end{pmatrix} \]
this then increases the order of convergence to $\mathcal{O}(h^4)$, but computation is more expensive, as the matrix representation of differentiation is less sparse.
Spectral methods can be thought of as a limiting finite difference approximation. Under the assumption that the solution is periodic and infinitely differentiable, by considering a discretization in the Fourier domain, we can construct dense differentiation matrices in the spatial domain which can provide $\mathcal{O}(h^m)$ convergence for every $m$. These results can be extended to non-periodic domains by using Chebyshev polynomials instead of a Fourier basis, which for example avoids the {\em Runge phenomenom} (where interpolation errors increase exponentially as the number of gridpoints increases).
For example, on the grid $[-1,1]$ we use the points
\[ x_i = \cos\left(\frac{i\pi}{D}\right),\quad i=1,\dots,D \]
Higher order derivatives and boundary conditions are also included in an intuitive manner.
In higher dimensions, the differential matrices are constructed by taking Kronecker products of one-dimensional pseudo-spectral matrices.
For more details and a concise introduction to pseudo-spectral methods, see~\cite{Trefethen2000}. For some convergence results related to non-linear PDEs and spectral methods, see \cite{Boyd2001}. In the two dimensional problems we are considering, we construct pseudospectral differential matrices $\mathcal{D}_x$ and $\mathcal{D}_y$ using the methods provided in \cite{Trefethen2000,Boyd2001}. Heuristically, once suitable pseudospectral matrices are constructed, the non-linear partial differential equation is discretised in space by replacing gradients $\partial_x,\partial_y$ with the matrices $\mathcal{D}_x,\mathcal{D}_y$ respectively.

For a system where the pointwise constraints $\{\bm{x}_i,i=1,\dots N\}$ are located on the boundary $\partial\Omega$ at Chebyshev gridpoints, \eqref{eq:GradDescentpLap} can be discretised using an explicit Euler method with timestep $\tau$. Given a matrix of points defined by a Kronecker product grid of Chebyshev points $(x_i)_{i=1,\dots D},(y_i)_{i=1,\dots D}$, we define $\bm{u}^n=(u(x_i,y_j),t^n)_{i,j=1,\dots D}$ as the discretised approximation of $u(\bm{x},n\tau)$ at time $n\tau$ for $n\in\bbN$. Given initial condition $\bm{u}^0=\bm{u}_0$, we then have:
\begin{align}
\frac{\bm{u}^{n+1} - \bm{u}^n}{\tau}
=
\begin{cases}
\begin{split}
&\mathcal{D}_x(\mathcal{D}_x \bm{u} ((\mathcal{D}_x\bm{u}^n)^2+(\mathcal{D}_y\bm{u}^n)^2)^{\frac{p-2}{2}} \cdot\bm{\rho}^2)
\\&
\qquad + \mathcal{D}_y(\mathcal{D}_y \bm{u}^n ((\mathcal{D}_x\bm{u}^n)^2+(\mathcal{D}_y\bm{u}^n)^2)^{\frac{p-2}{2}}\cdot \bm{\rho}^2),
\end{split}	
 \qquad\qquad\text{ on }\Omega\setminus\partial\Omega,
\\
\\
\begin{split}
&-\bm{\rho}^2\cdot((\mathcal{D}_x\bm{u}^n)^2+(\mathcal{D}_y\bm{u}^n)^2)^{\frac{p-2}{2}}
((\mathcal{D}_x\bm{u}^n)\bm{n}_x + (\mathcal{D}_x\bm{u}^n)\bm{n}_y)
\\&\qquad+
\beta\Bigg[
\mathcal{D}_x(\mathcal{D}_x \bm{u}^n ((\mathcal{D}_x\bm{u}^n)^2+(\mathcal{D}_y\bm{u}^n)^2)^{\frac{p-2}{2}}\cdot\bm{\rho}^2)
\\&\qquad\qquad+
\mathcal{D}_y(\mathcal{D}_y \bm{u}^n ((\mathcal{D}_x\bm{u}^n)^2+(\mathcal{D}_y\bm{u}^n)^2)^{\frac{p-2}{2}} \cdot\bm{\rho}^2)
\Bigg],
\end{split}
\!\begin{aligned}
&\text{ on }(x_i,y_j) \\
&\in\partial\Omega\backslash\{\bm{x}_i,i=1,\dots N\}
\end{aligned}
\\
\\
\begin{split}0,\end{split}
\qquad\qquad\qquad\qquad\qquad\qquad\qquad\qquad\qquad\qquad\qquad\qquad
\text{ on }\bm{x}_i,i=1,\dots N.
\end{cases} \label{eq:DiscretisedEuler}
\end{align}
where $\bm{n}_x,\bm{n}_y$ are matrices which are zero on $\Omega\backslash\partial\Omega$, and are the $x$ and $y$ components of the outward unit normal of the domain on $\partial\Omega$ respectively, $\bm{\rho} = (\rho(x_i,y_j))_{i,j=1,\dots,D}$, and $\cdot$ represents the pointwise product between two matrices.

However, sharp peaks are generally observed around constrained points in the interior of the domain, which is particularly evident for small values of $p$.
Between each constraint we expect the function to be smooth.
We therefore decompose the domain of interest, such that constraints lie on boundaries between patches of the domain, and match boundary conditions between each patch.

Inside each patch, and on the boundary of the entire domain, the system obeys~\eqref{eq:GradDescentpLap}. If two patches share a boundary, we need to ensure that their values match, and that the flux also matches.
Thus for patches $i$ and $j$ which share the boundary $\partial\Omega_{ij}$, at shared points
\begin{gather}
u_i = u_j,\nonumber\\
\rho^2|\nabla u_i|^{p-2}\nabla u_i\cdot \mathrm{n}_i = -\rho^2|\nabla u_j|^{p-2}\nabla u_j\cdot \mathrm{n}_j,\label{eq:FluxBCs}
\end{gather}
where $\mathrm{n}_i$ is the normal from patch $i$ to $\partial\Omega_{ij}$, and $u_i$ is the value of the function in patch $i$, and similarly for $\mathrm{n}_j$ and patch $j$.
Due to the shape of the grid, for the best accuracy, constraints should be placed between shared corners, however in practice this causes degeneracies due to discretization.
To solve this the boundary conditions between patches that share corners can be altered to account for this, but for simplicity we place constraints shared by more than two patches near, not on, shared corners.

\begin{figure}[ht!]
\centering
\includegraphics[width=0.45\textwidth]{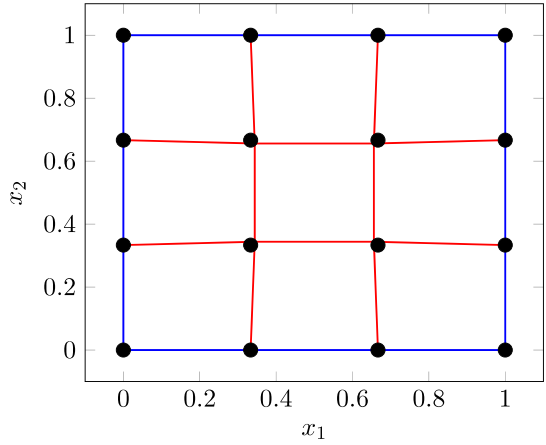}
\caption{Boundary patching method. The blue lines indicate the boundary of the domain, where the original boundary condition in~\eqref{eq:GradDescentpLap} is used. The interior red lines are where~\eqref{eq:FluxBCs} is used. Constraints (black points) are placed on the corners of patches, with the exception of interior corners, which are offset.
}\label{fig:Numerics:BoundaryPatching}		
\end{figure}

The additional matching conditions between patches converts the problem~\ref{eq:GradDescentpLap} to a set of {\em differential algebraic equations} (DAEs)~\cite{kunkel2006differential}.
Our system is semi-explicit, in that defining $\Gamma = \l\cup_{i,j} \partial\Omega_{ij}\r$ to be the boundary between patches, $g=u\lfloor_\Gamma$ as the value of $u$ on $\Gamma$, and $f=u\lfloor_{\Omega\setminus \Gamma}$ as the value of $u$ away from the boundaries, we can write 
\begin{align*}
\frac{\dd \bm{f}}{\dd t} = & F(\bm{f},\bm{g}),\\
0 = & G(\bm{f},\bm{g}).
\end{align*}
where $F$ is given by the equations \eqref{eq:GradDescentpLap} restricted to $\Omega\setminus\Gamma$, and $G$ is given by \eqref{eq:FluxBCs}, and $\bm{f},\bm{g}$ are the spatially discretized $f$ and $g$. It is necessary to use (semi-)implicit methods on DAEs, so that variables determined by algebraic equations can be updated for each timestep. Many implicit methods are available and well studied in the literature~\cite{hairer1991solving}, we consider the semi-implicit Euler method~\cite{Deuflhard1987}.
At step $n+1$, we solve the linear system:
\begin{align*}
\bm{f}^{n+1} - \bm{f}^n & = \tau F(\bm{f}^{n+1},\bm{g}^{n+1}), \\
0 & = G(\bm{f}^{n+1},\bm{g}^{n+1}).
\end{align*}
After discretising $F$ and $G$ using the pseudospectral matrices $\mathcal{D}_x,\mathcal{D}_y$ in an analogous way to \eqref{eq:DiscretisedEuler}, numerically this involves computing a Jacobian in $F$ and $G$, and solving a linear system of equations at each timestep.

\subsubsection{Results}

Figure~\ref{Figure:pLapExamples} shows an example of the minimiser of $\cEpnCon$ (Figures~\ref{Figure:DiscretepLapUniform}, \ref{Figure:DiscretepLapxy} and \ref{Figure:DiscretepLapTrig}) with $n=1500$, the minimiser of $\cEpinftyCon(\cdot;\rho)$ (Figures~\ref{Figure:pLapUniform}, \ref{Figure:pLapxy} and \ref{Figure:pLapTrig}), the minimiser of $\cEpinftyCon(\cdot;\rho_{n,h})$ (Figure~\ref{Figure:npLapUniformKDE}, \ref{Figure:npLapxyKDE} and \ref{Figure:npLapTrigKDE}) and the minimiser of the minimiser of $\cEpinftyCon(\cdot;S_{T,\lambda}(\rho_{n,h}))$ (Figures~\ref{Figure:npLapUniformSKDE}, \ref{Figure:npLapxySKDE} and~\ref{Figure:npLapTrigSKDE}), with $n=2^{17}$ samples for each density, and $p=3$, where the minimisers are achieved using gradient descent methods explained above with a tolerance of $10^{-5}$. We construct the KDE using $h=0.01$, and for the SKDE we choose $T=2^{12}$ and $\lambda=10^{-6}$.
Each patch is discretised with $100$ Chebyshev points. For Figures~\ref{Figure:DiscretepLapUniform}, \ref{Figure:DiscretepLapxy} and \ref{Figure:DiscretepLapTrig} we chose $\eps$ via the tuning procedure described above.

\begin{figure}
\centering
\subfloat[][]{\includegraphics[width=0.3\textwidth]{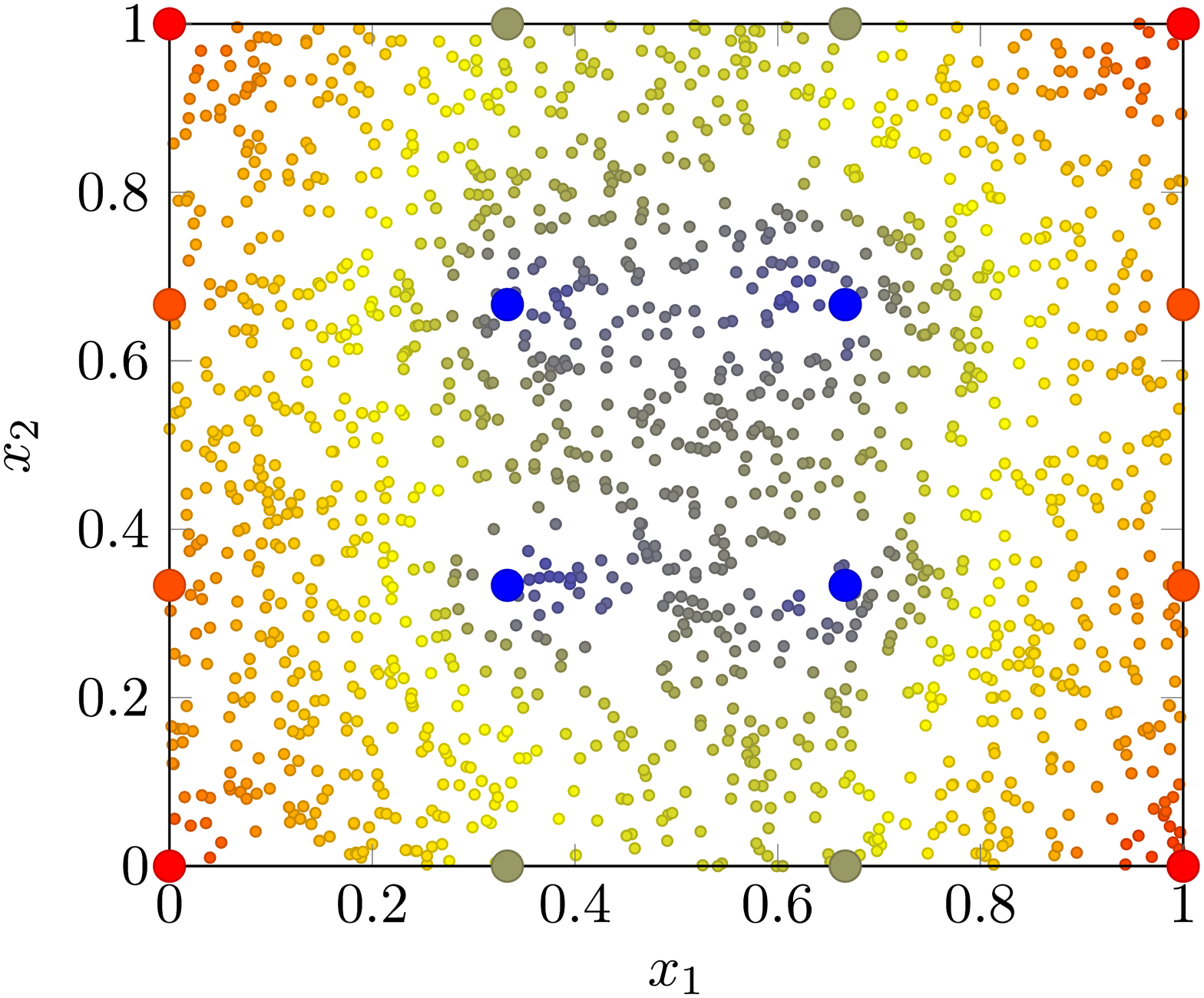}\label{Figure:DiscretepLapUniform}}
\subfloat[][]{\includegraphics[width=0.3\textwidth]{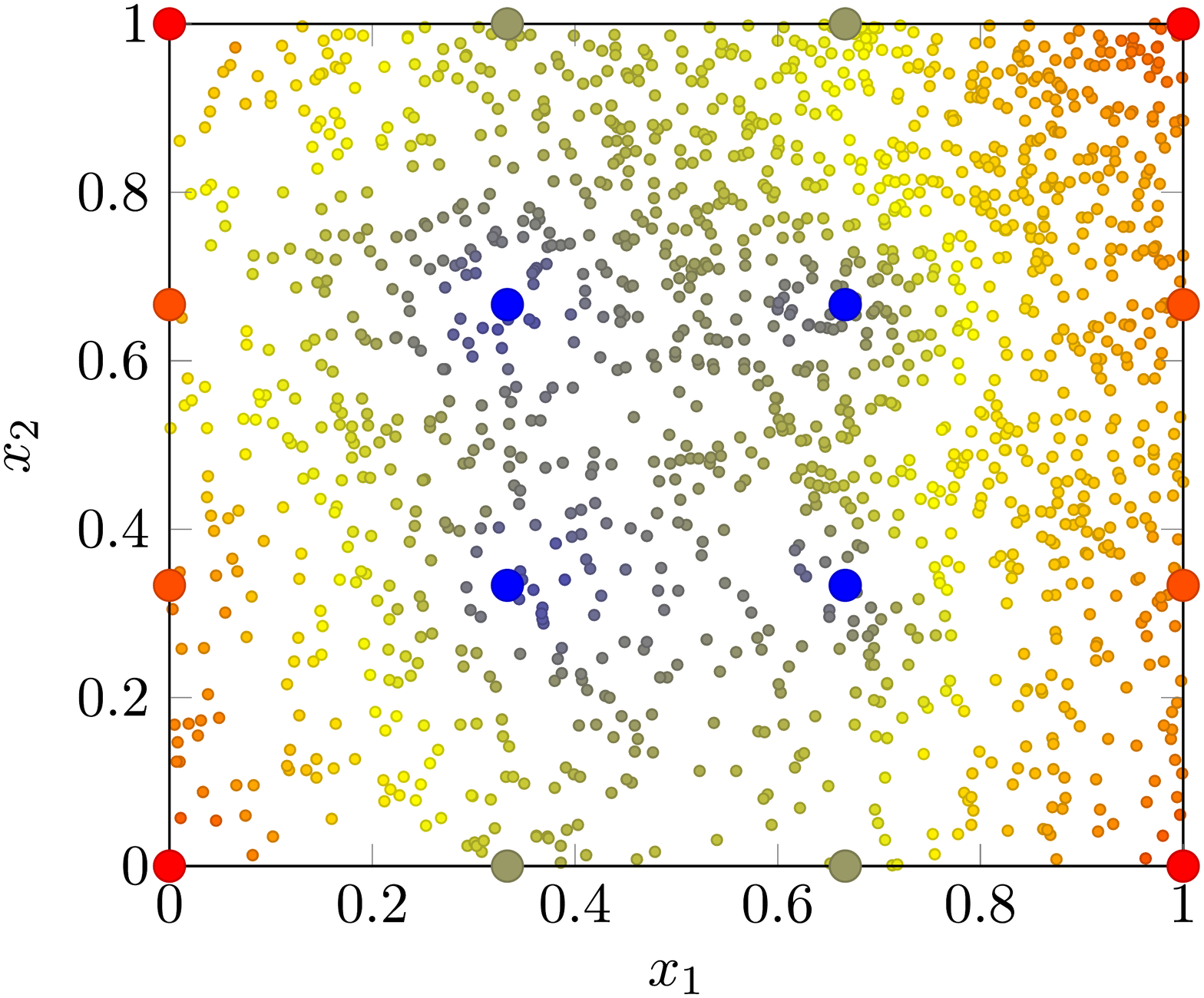}\label{Figure:DiscretepLapxy}}
\subfloat[][]{\includegraphics[width=0.3\textwidth]{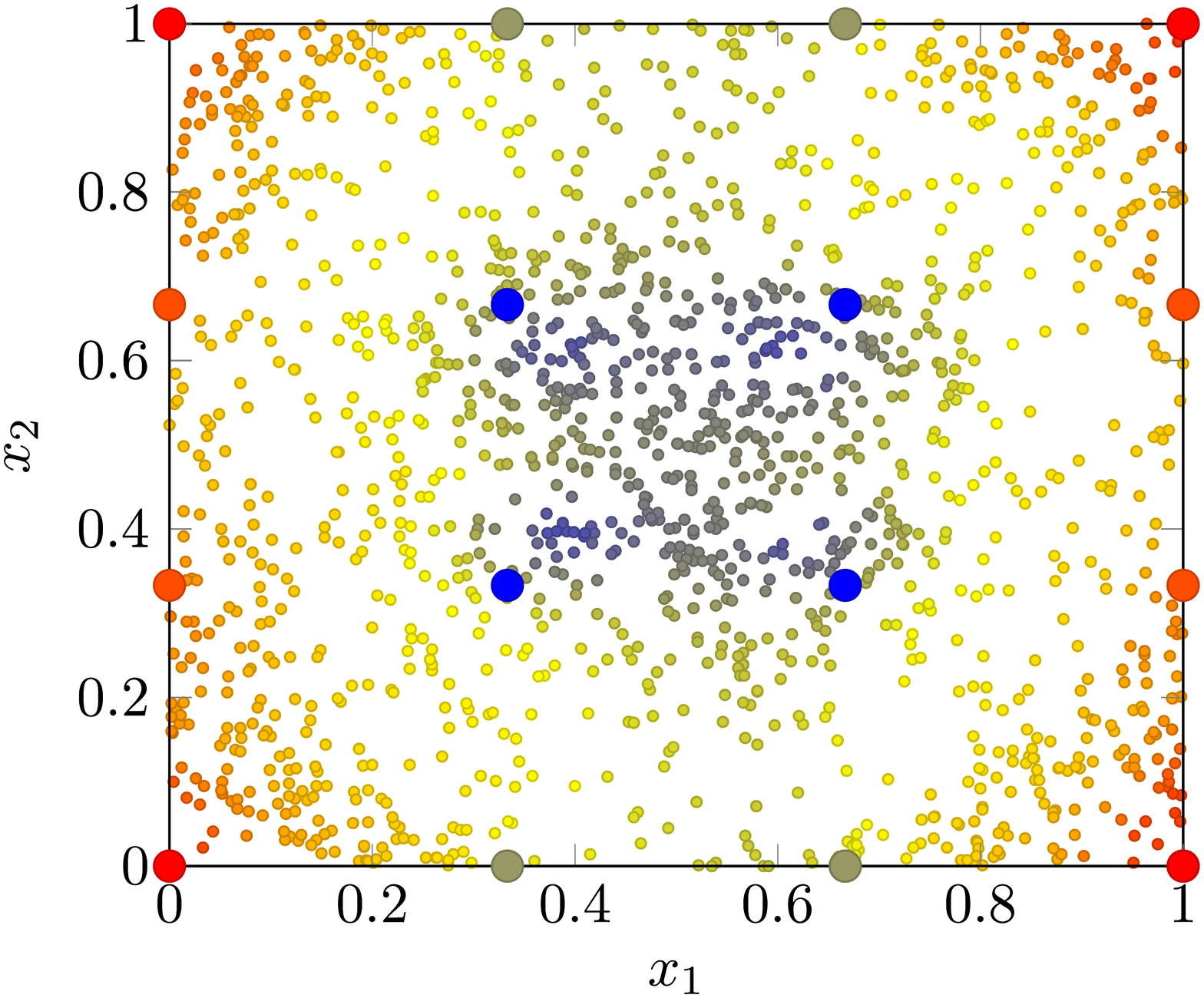}\label{Figure:DiscretepLapTrig}}
\hspace{0.02\textwidth}
\subfloat{\includegraphics[width=0.07\textwidth]{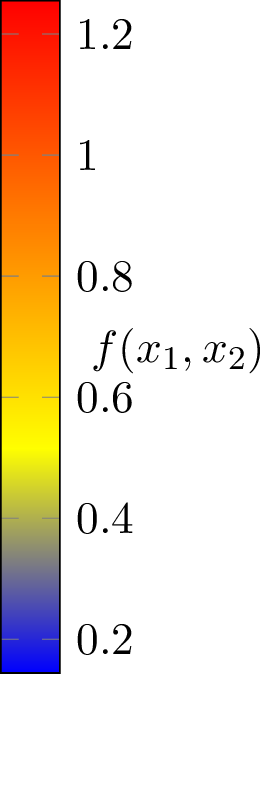}}
\addtocounter{subfigure}{-1}
\\
\subfloat[][]{\includegraphics[width=0.30\textwidth]{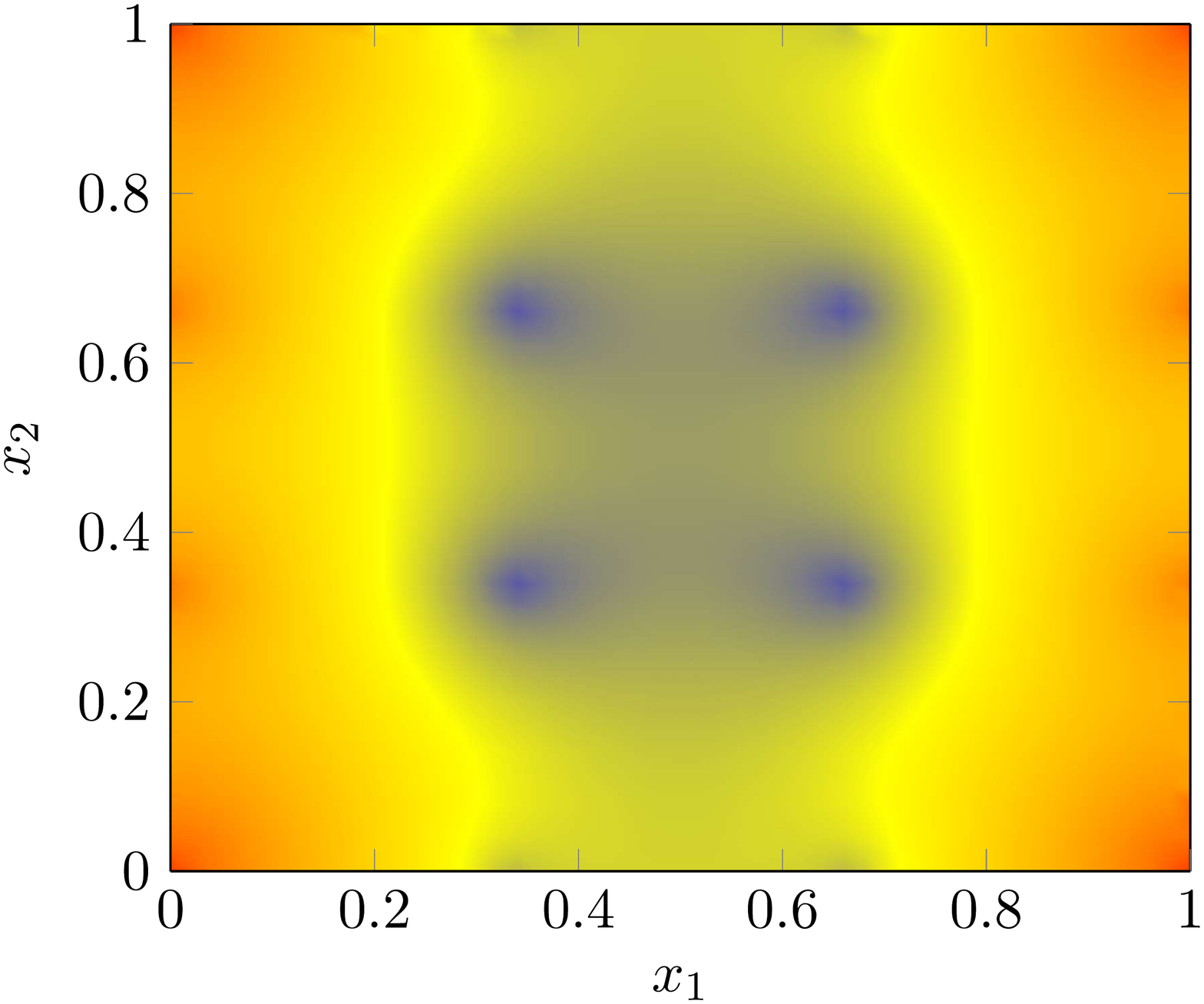}\label{Figure:pLapUniform}}	
\subfloat[][]{\includegraphics[width=0.30\textwidth]{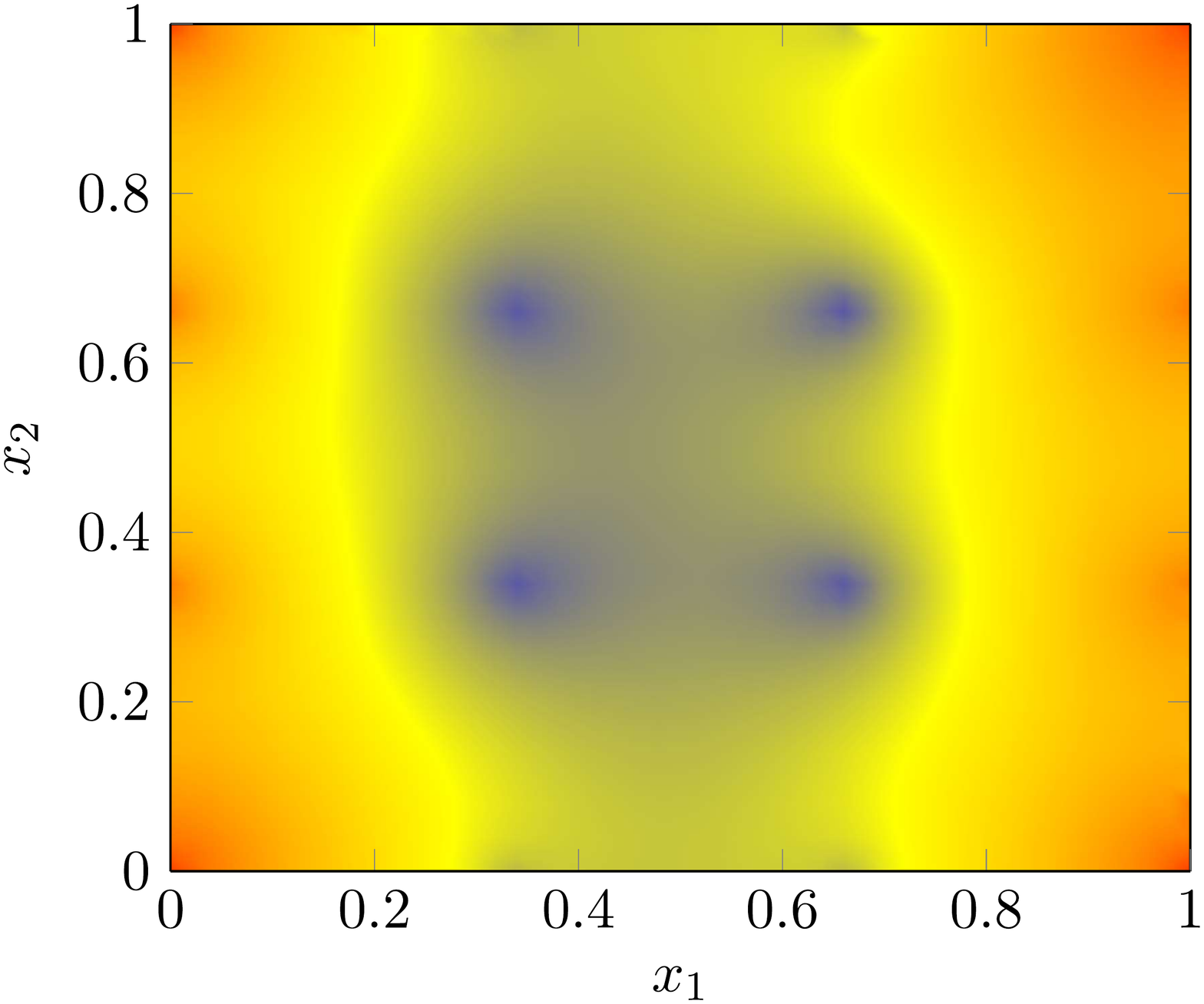}\label{Figure:pLapxy}}	
\subfloat[][]{\includegraphics[width=0.30\textwidth]{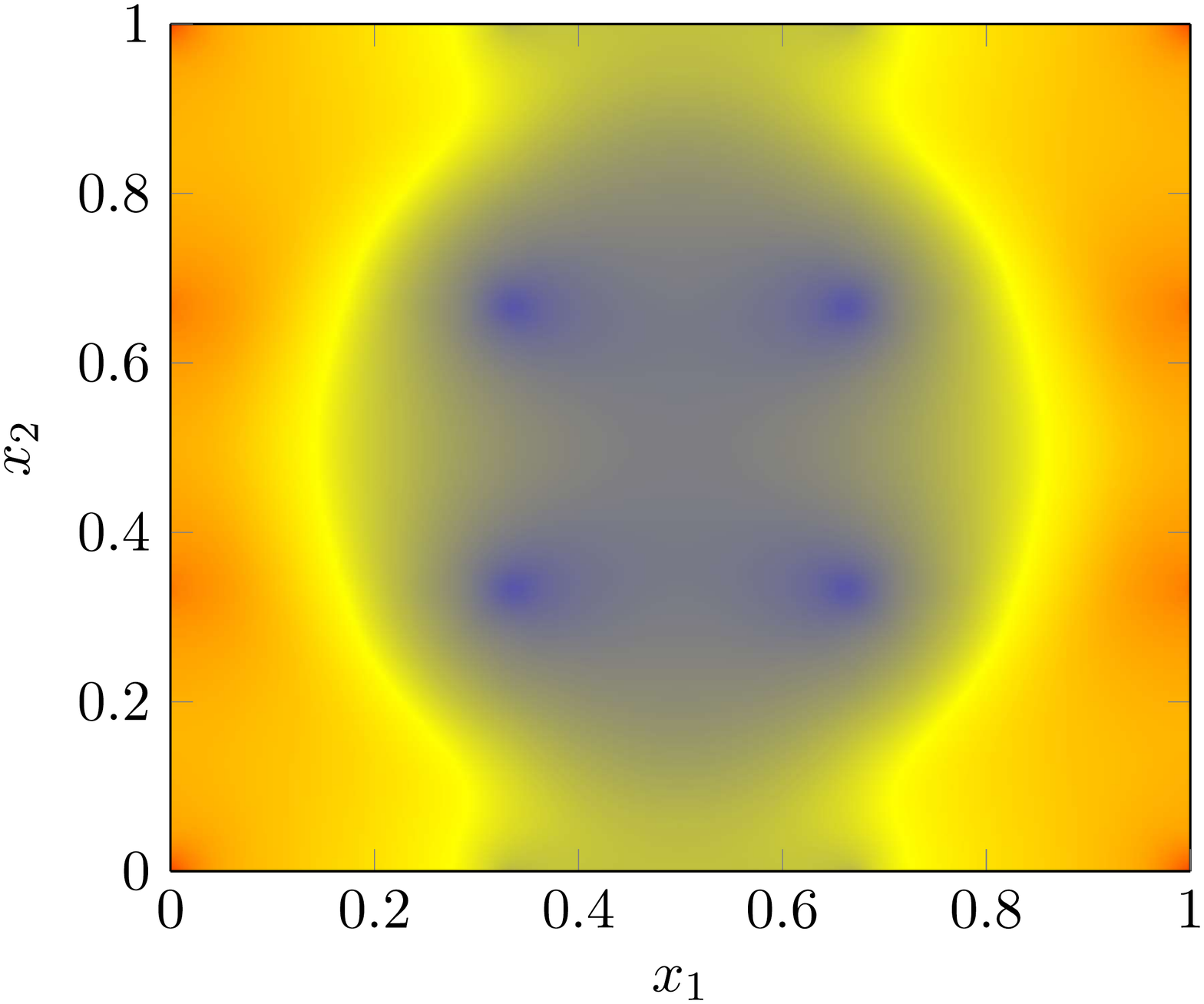}\label{Figure:pLapTrig}}
\hspace{0.02\textwidth}
\subfloat{\includegraphics[width=0.07\textwidth]{Figs/GenerateFigures-figure30.png}}
\addtocounter{subfigure}{-1}
\\
\subfloat[][]{\includegraphics[width=0.30\textwidth]{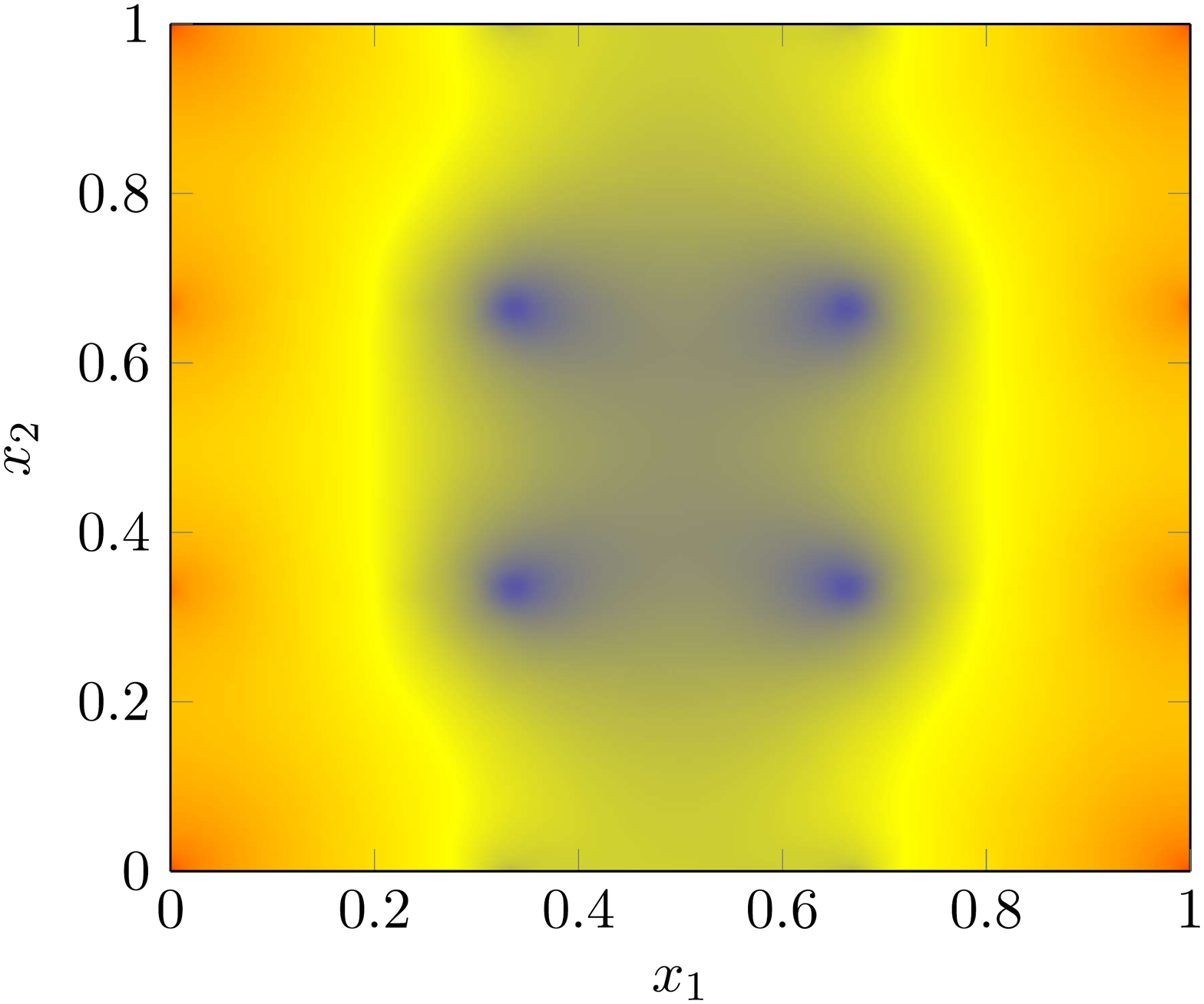}\label{Figure:npLapUniformKDE}}
\subfloat[][]{\includegraphics[width=0.30\textwidth]{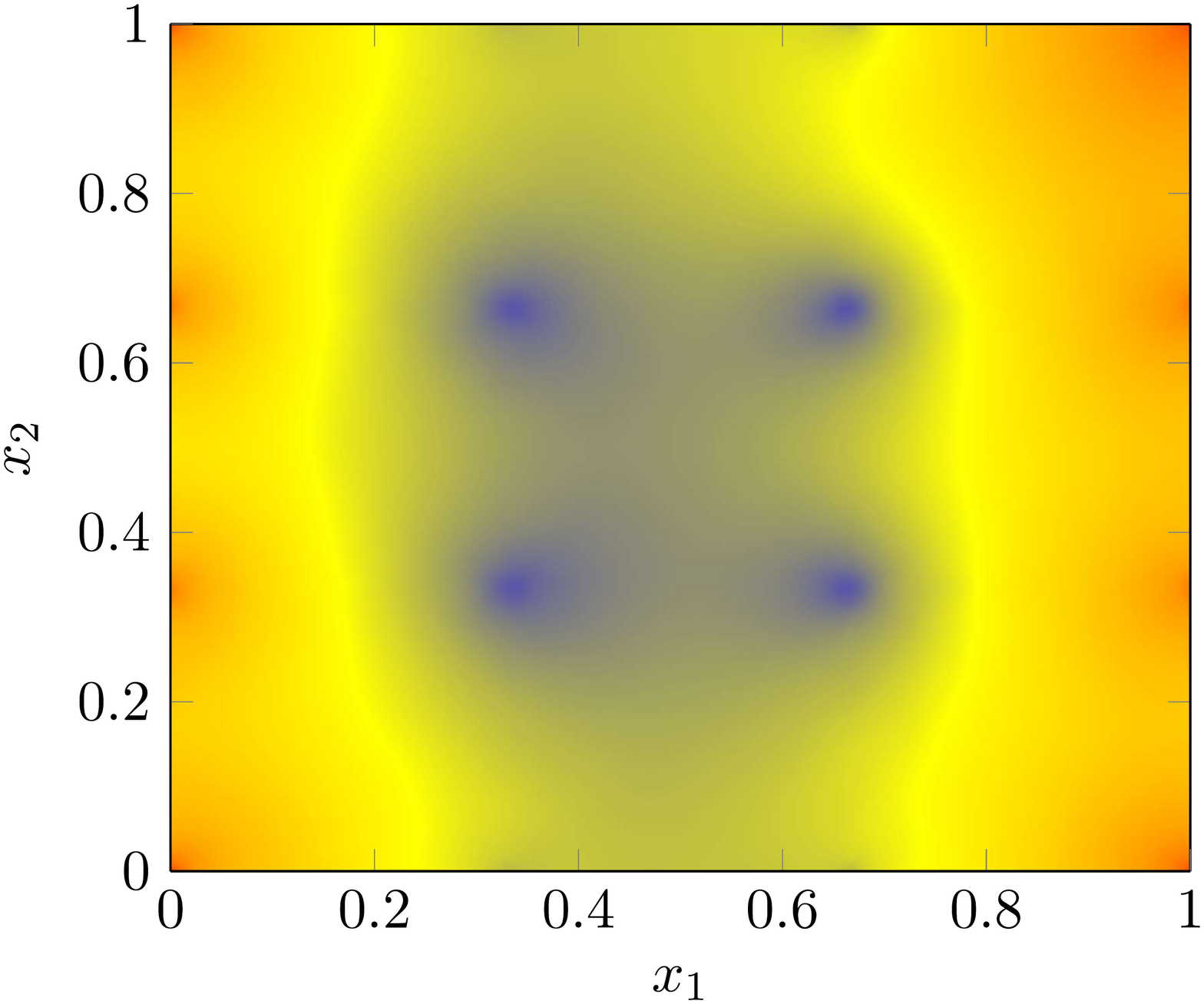}\label{Figure:npLapxyKDE}}	
\subfloat[][]{\includegraphics[width=0.30\textwidth]{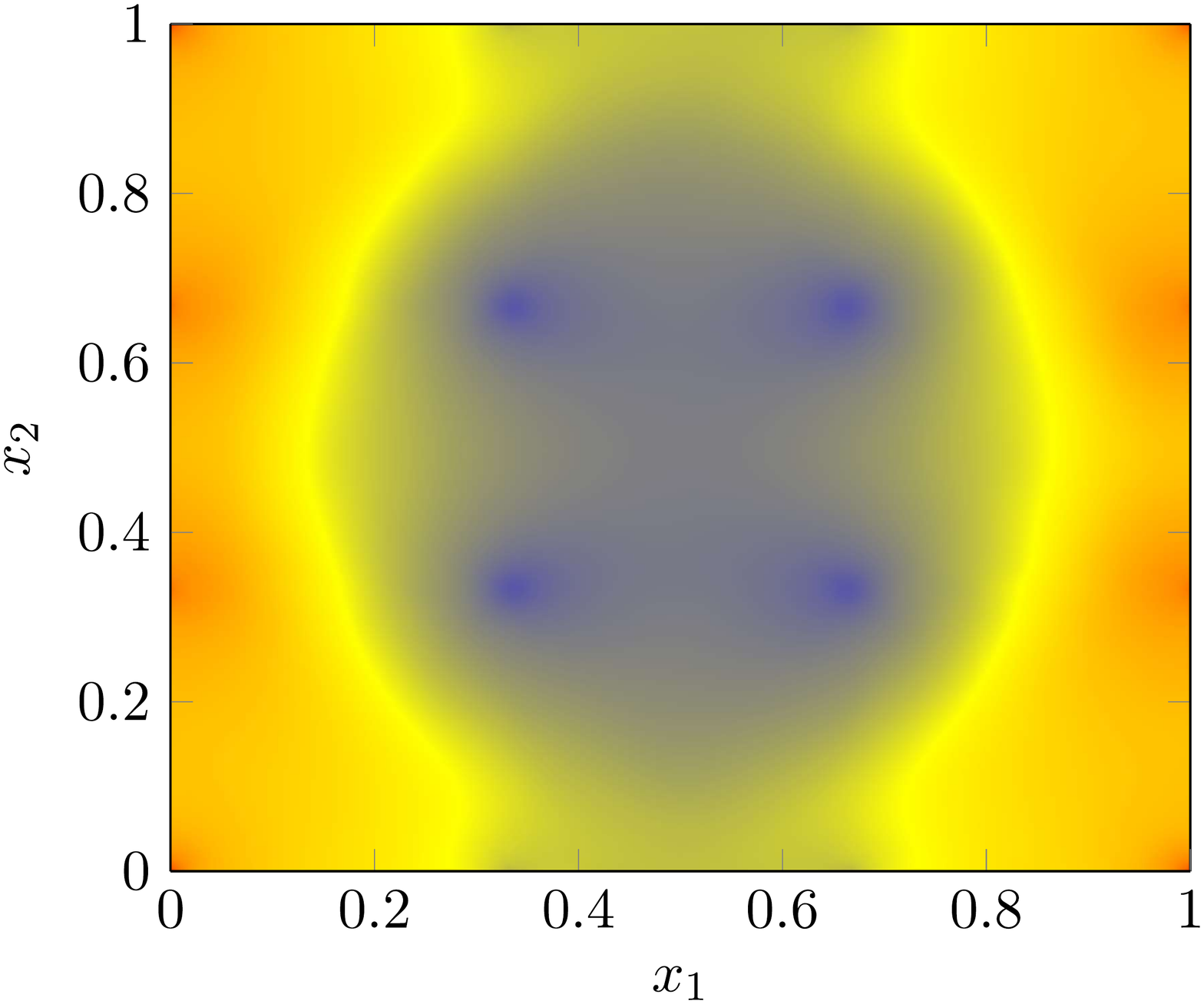}\label{Figure:npLapTrigKDE}}
\hspace{0.02\textwidth}
\subfloat{\includegraphics[width=0.07\textwidth]{Figs/GenerateFigures-figure30.png}}
\addtocounter{subfigure}{-1}
\\
\subfloat[][]{\includegraphics[width=0.30\textwidth]{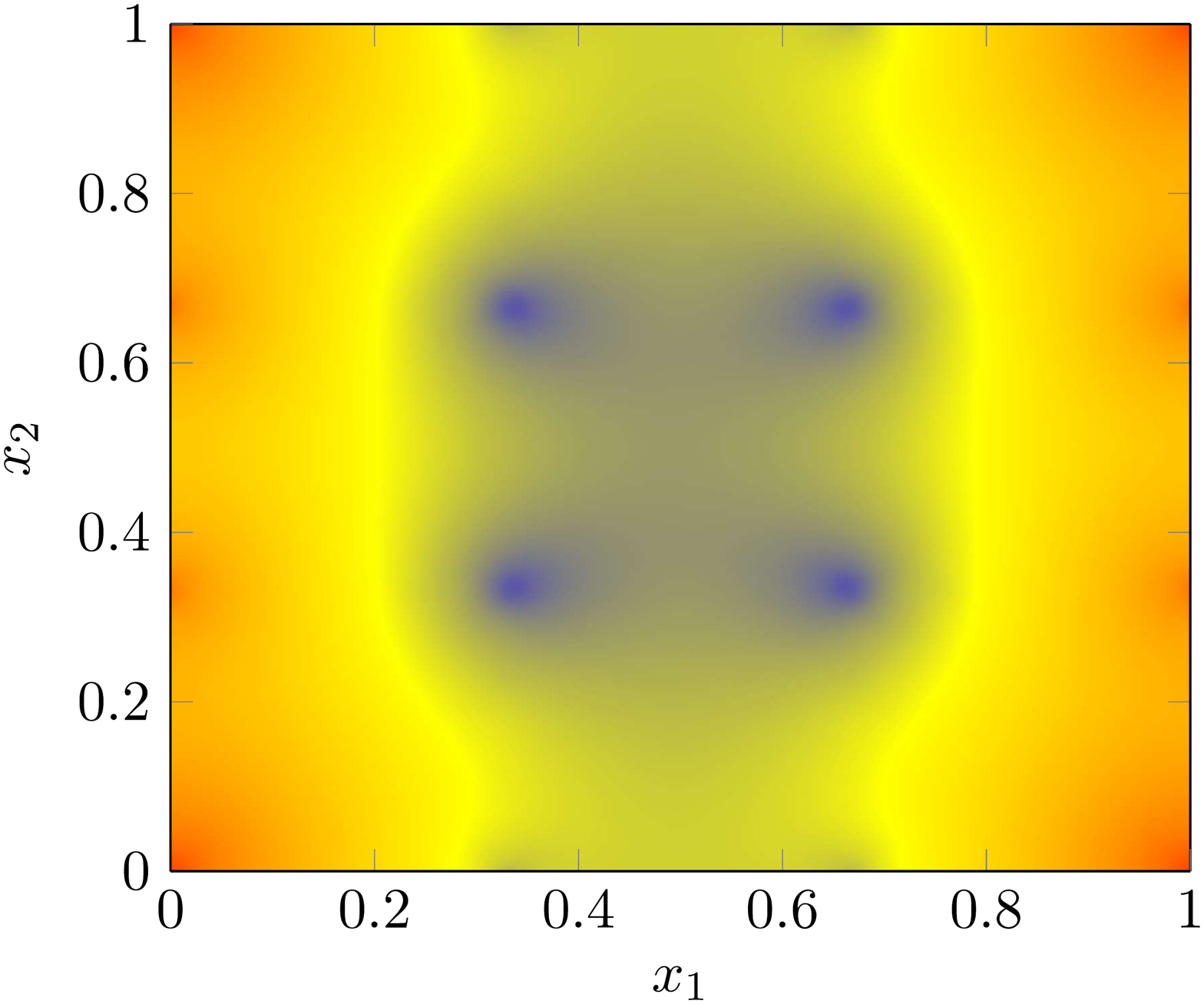}\label{Figure:npLapUniformSKDE}}
\subfloat[][]{\includegraphics[width=0.30\textwidth]{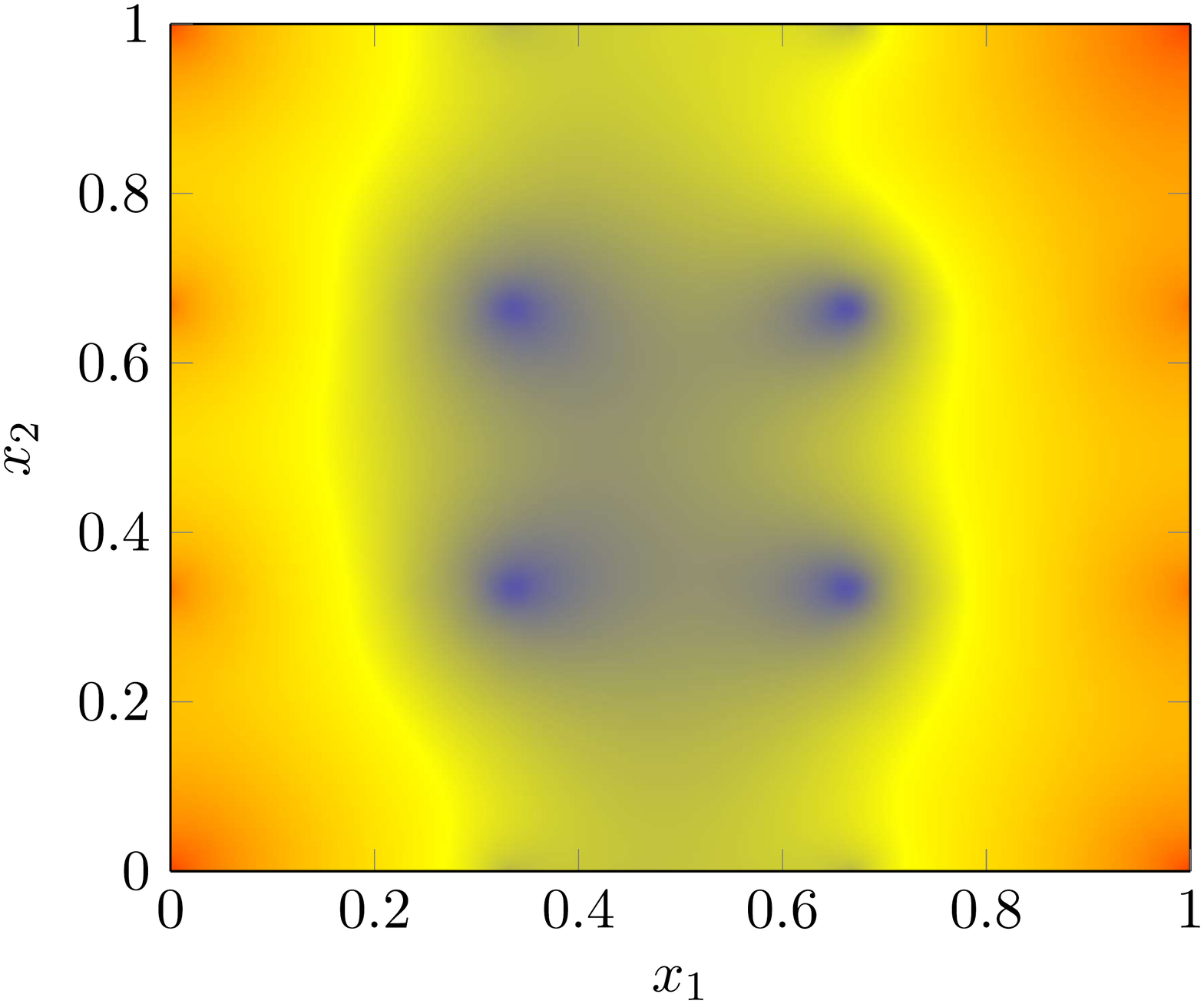}\label{Figure:npLapxySKDE}}	
\subfloat[][]{\includegraphics[width=0.30\textwidth]{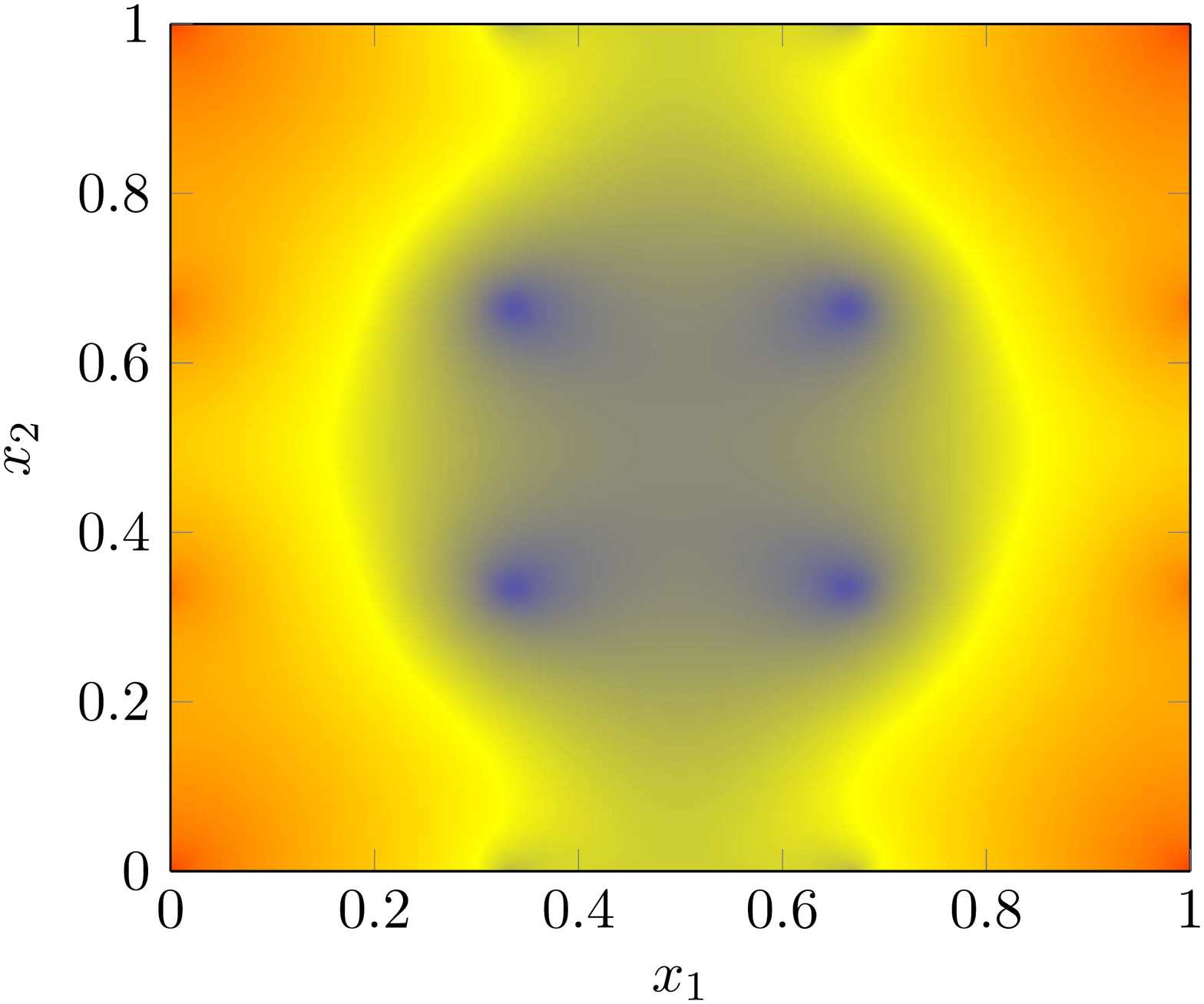}\label{Figure:npLapTrigSKDE}}
\hspace{0.02\textwidth}
\subfloat{\includegraphics[width=0.07\textwidth]{Figs/GenerateFigures-figure30.png}}
\addtocounter{subfigure}{-1}
\\
\caption{Numerically constructed minimisers for $\cEpnCon$ (first row), ground truth $\cEpinfty(\cdot;\rho)$ (second row), the KDE $\cEpinftyCon(\cdot;\rho_{n,h})$ (third row), and the SKDE $\cEpinftyCon(\cdot;S_{T,\lambda}(\rho_{n,h}))$ (fourth row).
The first column is for data points sampled from $\rho_1$, the second column for data points sampled from $\rho_2$, and the third column is for data points sampled from $\rho_3$. In these examples we have taken $p=3$.
} \label{Figure:pLapExamples}
\end{figure}

In Figure~\ref{fig:Numerics:np-pError}, we give the $L^\infty$ errors between the minimisers of $\cEpinftyCon(\cdot;\rho)$ and $\cEpinftyCon(\cdot;\rho_{n,h})$ or $\cEpinftyCon(\cdot;S_{T,\lambda}(\rho_{n,h}))$, for different values of $n$. The results reflect the conclusions in Section~\ref{subsec:Numerics:DensityEstimates}, the SKDE is an improvement for $\rho_1$ and $\rho_2$, but due to the large fluctuations in gradient in $\rho_3$, smoothing the KDE increases the associated error. Although using the SKDE in the minimisation problem improves the accuracy of the result, we note that the choices of $\lambda$ and $T$ were not necessarily optimal in these examples, as the optimal choice of parameters can lead to density estimates which have negative values, causing numerical errors during gradient descent. It is a topic of future work to consider and apply the associated smoothing spline problem for strictly positive functions in the SKDE. Finally, we present the CPU time for each gradient descent computation in Figure~\ref{Figure:pLapComputationTimes}.  For low dimensional problems with large amounts of data, a continuum approach is shown to be computationally cheaper, and we note that density estimation can be parallelised to reduce computation time of the continuum method.

\begin{figure}[ht!]	
	\centering
	\includegraphics[width=0.45\textwidth]{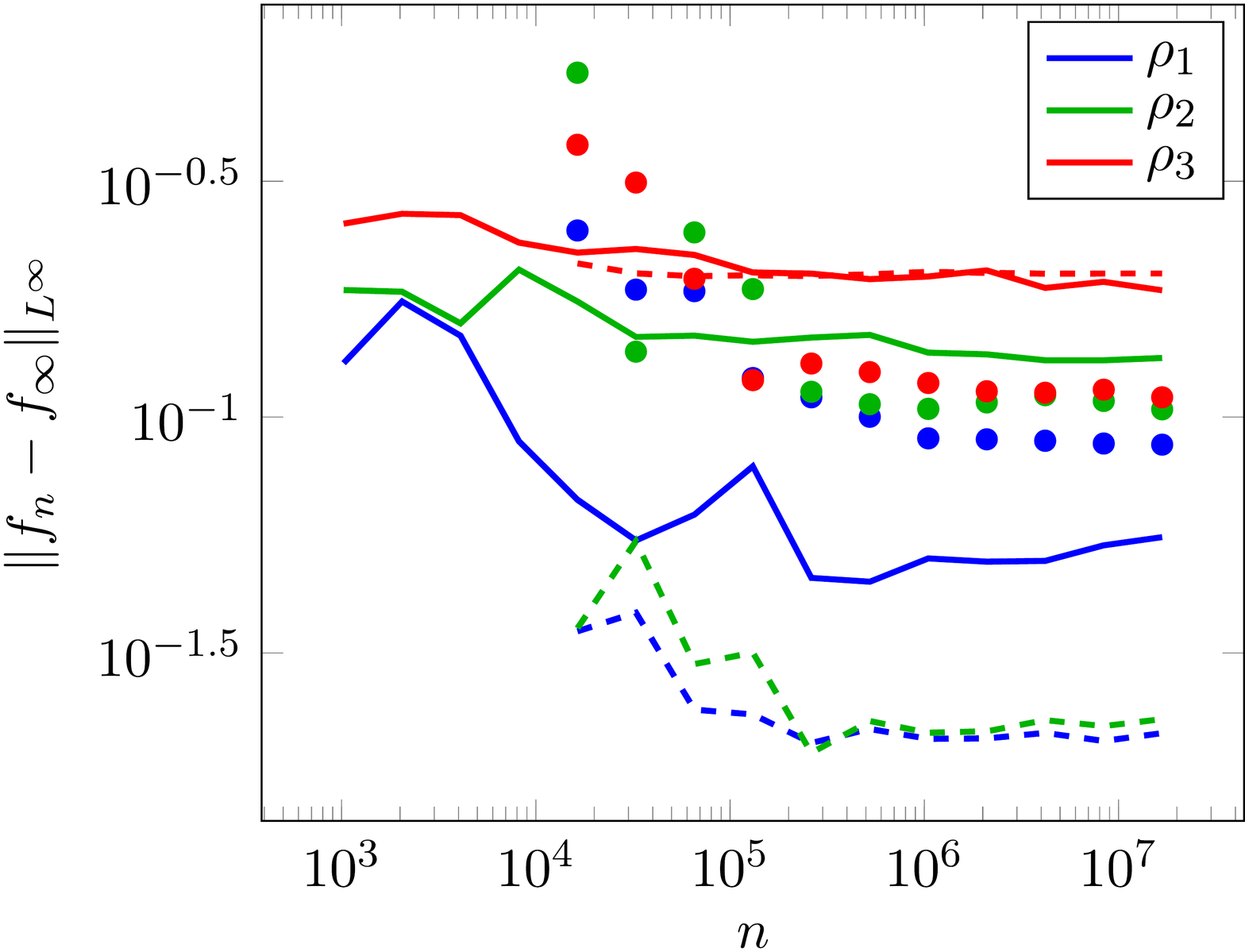}\label{Figure:pLapLInfError}
	\caption{
		$L^{\infty}$ error between the computed minimizers $f_n$ of $\mathcal{E}_\infty^{(p)}(\cdot;\rho_n)$ against the minimizer $f_\infty$ of $\mathcal{E}_\infty^{(p)}(\cdot,\rho)$, for the KDE (dotted) and SKDE (dashed) in the $p$-Dirichlet energy minimization problem, where $\lambda=10^{-6}$ and $T=2^{12}$ and $p=3$, and results from the method discussed in \cite{floresrios19AAA} (solid).
	}\label{fig:Numerics:np-pError}	
\end{figure}

\begin{figure}[ht!]
	\centering
	\subfloat[][]{	\includegraphics[width=0.45\textwidth]{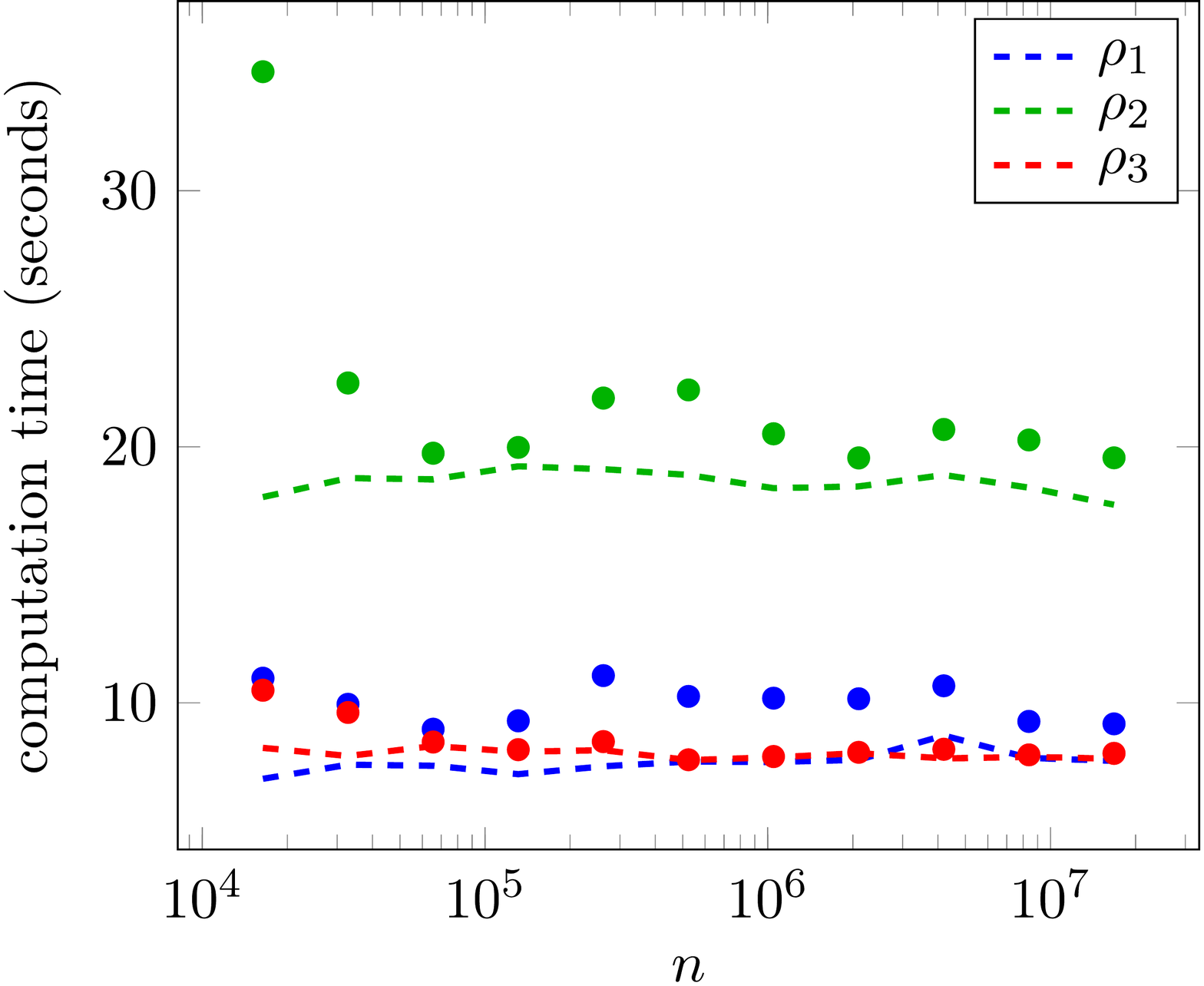}\label{Figure:NonDensityComputationTime}}	
	\subfloat[][]{\includegraphics[width=0.45\textwidth,valign=b]{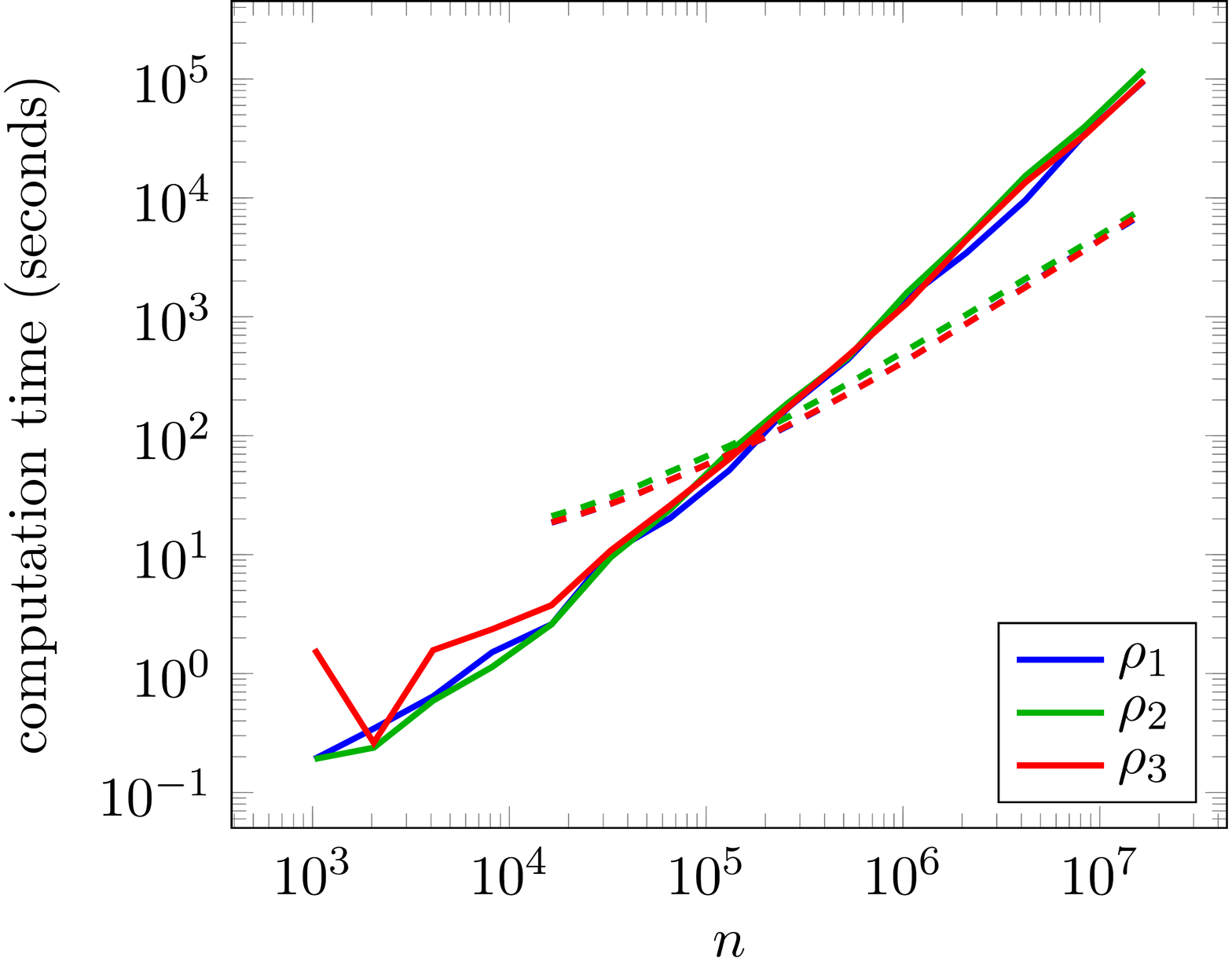}\label{Figure:ComputationTimesComparison}}
	\caption{Computation times for the different methods. Left: computation times for the continuum $p$ Dirichlet energy calculation, once the density estimate is constructed. Dotted lines represent the minimisation problem using the KDE, dashed are using the SKDE, where $\lambda=10^{-6}$ and $T=2^{12}$ and $p=3$. We note that computation time using the exact density is 11.03, 11.32, and 21.77 seconds respectively for $\rho_1, \rho_2$ and $\rho_3$. Right: Comparison of the continuum method presented (using SKDE to approximate the density, dashed line) in this paper and the discrete methodology discussed in \cite{floresrios19AAA} (solid line).
	}\label{Figure:pLapComputationTimes}			
\end{figure}

\section{Conclusions and Future Work} \label{sec:Conc}
We have shown that the appropriate limit is attained when using a density estimate in the constrained continuum $p$-Dirichlet minimization problem, provided the density estimate converges uniformly almost surely. In addition, we have shown that the kernel density estimate meets the convergence criterion, and using smoothing splines can improve the approximation without affecting convergence. The non-local $p$-Dirichlet energy convergence result also provides an insight into the link between the discrete $p$-Dirichlet energy and density estimation in the continuum analogue.

We have also provided numerical examples using different probability densities, which show that the constrained continuum $p$-Dirichlet energy can be used effectively in problems with a large amount of data and low dimension.

Future improvements to the scheme include incorporating a positivity constraint in the smoothing spline calculation for more robust density estimation. We would also like to consider different density (parametric or non-parametric) estimation methods in the minimization problem, and construct numerically stable continuum methods for when $p$ is large.

\section*{Acknowledgements}

This work was supported by The Alan Turing Institute under the EPSRC grant EP/N510129/1.
OMC is a Wellcome Trust Mathematical Genomics and Medicine student supported financially by the School of Clinical Medicine, University of Cambridge.
TH was supported by The Maxwell Institute Graduate School in Analysis and its Applications, a Centre for Doctoral Training funded by the EPSRC (EP/L016508/01), the Scottish Funding Council, Heriot-Watt University and the University of Edinburgh.
CBS acknowledges Leverhulme Trust (Breaking the non-convexity barrier, and Unveiling the Invisible), the Philip Leverhulme Prize, the EPSRC grants EP/M00483X/1 and EP/N014588/1, the European Union Horizon 2020 Marie Skodowska-Curie (NoMADS, grant agreement No 777826, and CHiPS, grant agreement No 691070), the Cantab Capital Institute for the Mathematics of Information (CCIMI) and the Alan Turing Institute.
MT is grateful for the support of the CCIMI, Cambridge Image Analysis (CIA) and has received funding from the European Research Council (ERC) under the European Union's Horizon 2020 research and innovation programme (grant agreement No 647812).
KCZ was supported by the Alan Turing Institute under the EPSRC grant EP/N510129/1.
We would like to thank Dr B. Goddard for the spectral code used in our numerical experiments.
\vspace{2\baselineskip}

\bibliographystyle{plain}
\bibliography{projectBibliography}

\end{document}